\documentclass{amsart}

\usepackage[toc,page]{appendix}
\usepackage[headheight=12.1pt]{geometry}
\usepackage[all,cmtip]{xy}
\usepackage{amsmath}
\usepackage{amsfonts}
\usepackage{amssymb}
\usepackage{mathrsfs}   
\usepackage{amsthm}
\usepackage{xcolor}
\usepackage{cite}
\usepackage{stmaryrd}
\usepackage{graphicx}
\usepackage{pgf}
\usepackage{eepic}
\usepackage{pstricks}
\usepackage{epsfig}
\usepackage{fancyhdr}
\usepackage{tikz,caption,float}
\usepackage[backref=page]{hyperref}

\usepackage{amsbsy}

\hypersetup{
     colorlinks   = true,
     citecolor    = red,
     linkcolor = blue,
     urlcolor = purple
}
    
\usepackage{mathptmx}  
  \newcommand\imCMsym[4][\mathord]{%
  \DeclareFontFamily{U} {#2}{}
  \DeclareFontShape{U}{#2}{m}{n}{
    <-6> #25
    <6-7> #26
    <7-8> #27
    <8-9> #28
    <9-10> #29
    <10-12> #210
    <12-> #212}{}
  \DeclareSymbolFont{CM#2} {U} {#2}{m}{n}
  \DeclareMathSymbol{#4}{#1}{CM#2}{#3}
}
\newcommand\alsoimCMsym[4][\mathord]{\DeclareMathSymbol{#4}{#1}{CM#2}{#3}}

\imCMsym{cmmi}{124}{\CMjmath}
\imCMsym[\mathop]{cmsy}{113}{\CMamalg}
\imCMsym[\mathop]{cmex}{96}{\CMcoprod}
\alsoimCMsym[\mathop]{cmex}{97}{\CMbigcoprod}

\theoremstyle{plain}

\newtheorem{theorem}{Theorem}[section]

\newtheorem{proposition}[theorem]{Proposition}

\newtheorem{corollary}[theorem]{Corollary}

\newtheorem{lemma}[theorem]{Lemma}

\theoremstyle{definition}

\newtheorem{definition}[theorem]{Definition}
\newtheorem{setup}[theorem]{Setup}

\theoremstyle{remark}
\newtheorem{remark}[theorem]{Remark}

\newtheorem{example}[theorem]{Example}

\newcommand{\N}{{\mathbb N}}
\newcommand{\Z}{{\mathbb Z}}
\newcommand{\Q}{{\mathbb Q}}
\newcommand{\R}{{\mathbb R}}

\newcommand{\F}{{\mathbb F}}

\renewcommand{\P}{{\mathbb P}}
\newcommand{\A}{{\mathbb A}}

\newcommand{\isomto}{\overset{\sim}{\rightarrow}}

\newcommand{\pow}[1]{\llbracket #1 \rrbracket}

\newcommand{\weak}[1]{\langle #1\rangle^\dagger}
\newcommand{\tate}[1]{\langle #1 \rangle}

\newcommand{\spec}[1]{\mathrm{Spec}\left(#1\right)}

\newcommand{\spf}[1]{\mathrm{Spf}\left(#1\right)}

\newcommand{\cur}[1]{\mathcal{#1}}
\newcommand{\fr}[1]{\mathfrak{#1}}

\newcommand{\Norm}[1]{\left\Vert #1\right\Vert}
\newcommand{\norm}[1]{\left\vert#1\right\vert}

\newcommand{\pn}{(\varphi,\nabla)}

\newcommand{\rig}{\mathrm{rig}}

\usepackage{euscript}

\usepackage{relsize}
\usepackage[bbgreekl]{mathbbol}
\usepackage{amsfonts}
\DeclareSymbolFontAlphabet{\mathbb}{AMSb} 
\DeclareSymbolFontAlphabet{\mathbbl}{bbold}

\usepackage{bm}

\title{Local acyclicity in $\bm{p}$-adic cohomology}

\author{Christopher Lazda}
       \address{ Mathematics Institute \\ Zeeman Building \\ University of Warwick \\ Coventry \\ CV4 7AL \\ United Kingdom }
       \email{chris.lazda@warwick.ac.uk}

\SetSymbolFont{stmry}{bold}{U}{stmry}{m}{n}
\usepackage{bm}

\begin{document}

\begin{abstract} We prove an analogue for $p$-adic coefficients of the Deligne--Laumon theorem on local acyclicity for curves. That is, for an overconvergent $F$-isocrystal $E$ on a relative curve $f:U\rightarrow S$ admitting a good compactification, we show that the cohomology sheaves of $\mathbf{R}f_!E$ are overconvergent isocrystals if and only if $E$ has constant Swan conductor at infinity.
\end{abstract}

\maketitle 

\tableofcontents

\section*{Introduction}

Given a morphism $f:X\rightarrow S$ of algebraic varieties over a field $k$, it is natural to ask when the higher direct images $\mathbf{R}^if_!E$ of some smooth coefficient object (such as a vector bundle with integrable connection, a lisse $\ell$-adic sheaf, or an overconvergent $F$-isocrystal) are smooth coefficient objects on $S$. Of course, this will always happen `generically', i.e. on a dense open subset of $S$, but one may hope to be able to say something about when this happens on the whole of $S$.

For example, the smooth and proper base change theorem in \'etale cohomology says that whenever $f$ is smooth and proper, and $E$ is a lisse $\ell$-adic sheaf (with $\ell\neq \mathrm{char}(k)$), then the relative cohomology sheaves $\mathbf{R}^if_!E=\mathbf{R}^if_*E$ are also lisse. Similarly, Berthelot's conjecture (versions of which have been proved by Shiho \cite{Shi08a} and Caro \cite{Car15a}) states that when $k$ is perfect of characteristic $p>0$, and $E$ is an overconvergent $F$-isocrystal on $X$, then each $\mathbf{R}^if_*E$ is an overconvergent $F$-isocrystal on $S$.

In characteristic $0$, with $\ell$-adic \'etale coefficients, these smoothness results often persist for families of open varieties, provided that the morphism $f:X\rightarrow S$ admits a `good' compactification, that is a compactification $\overline{X}$ smooth over $S$, such that the complement $\overline{X}\setminus X$ is a relative normal crossings divisor. However, the same is not true with `de\thinspace Rham' style coefficients in characteristic $0$, nor with $\ell$-adic coefficients in positive characteristic.  For example, one can use Artin--Schreier extensions to produce examples lisse $\F_\ell$-sheaves $E$ on $\A^2_k$ (with $\ell\neq p$) such that the rank of the cohomology groups jumps in fibres of the projection $\A^2_k\rightarrow \A^1_k$.

This is explained by the fact that the Swan conductor of $E$ at infinity, which is a numerical measure of the wild ramification of $E$, itself jumps along these fibres. It turns out, however, that for curves at least, the Swan conductor exactly controls the failure of the higher direct images to be lisse. Indeed, the main result of \cite{Lau81} shows that for a relative smooth curve $f:U\rightarrow S$, and a lisse $\F_\ell$-sheaf $E$ on $U$, ``if the wild ramification of $E$ at infinity is locally constant, then the higher direct images $\mathbf{R}^if_!E$ are lisse''. Concretely, the wild ramification of $E$ being (locally) constant means that the Swan conductor of $E$ at infinity is (locally) constant. One can use this to deduce a similar result with $\Z_\ell$ or $\Q_\ell$ coefficients.

The purpose of this article is to prove an analogue of this result for $p$-adic coefficients, that is for overconvergent $F$-isocrystals; in this case the correct analogue of the Swan conductor is the irregularity of a $p$-adic differential equation studied in \cite{CM00}. We have two main results in this direction. The first of these, Theorem \ref{theo: main dagger} is phrased in the language of relative Monsky--Washnitzer cohomology (in the spirit of \cite{Ked06a}), it assumes that the base variety $S$ is smooth, affine and connected, and also imposes relatively strong conditions on the curve $f:U\rightarrow S$ (see Setup \ref{setup: main}). The second, Theorem \ref{theo: main general}, is phrased using the theory of arithmetic $\mathcal{D}^\dagger$-modules, as developed by Berthelot and Caro, and while it allows more general bases $S$, as well as for any relative curve admitting a good compactification, it assumes that $k$ is perfect. In both cases, the result says that if $E$ is an overconvergent $F$-isocrystal on a suitable relative smooth curve $f:U\rightarrow S$, and $E$ has constant irregularity at infinity, then appropriately defined higher direct images are overconvergent $F$-isocrystals. (In fact, we work everywhere with `$F$-able isocrystals', that is extensions of subquotients of objects admitting some $F^n$-structure.) Results along these lines were previously obtained by Kedlaya \cite[Proposition 3.4.3]{Ked06b}, the the proof of which provided part of the inspiration for the methods used in \S\ref{sec: universal}.

The majority of this article is concerned with the Monsky--Washnitzer case, that is Theorem \ref{theo: main dagger}; it is not too difficult to then use the general $\mathcal{D}^\dagger$-module machinery (nicely summarised in \cite[\S1]{AC18a}) to deduce Theorem \ref{theo: main general} when $k$ is perfect. The basic idea of the proof is rather simple, and the bulk of the work consists of facing down the technical difficulties involved in actually getting this idea to work. To explain the approach, suppose that we have some relative curve $f:U \rightarrow S$ over a smooth, affine, base, and an overconvergent $F$-isocrystal $E$ on $U$. We know by results of Kedlaya (see Theorem \ref{theo: generic pushforwards} below) that for some open subset $V\subset S$ the higher direct images $\mathbf{R}^if_*(E|_{U_V})$ are overconvergent $F$-isocrystals on $U$, and by Noetherian induction we can assume that the complement $Z\subset S$ is also smooth over $k$, and that the higher direct images $\mathbf{R}^if_*(E|_{U_Z})$ are themselves overconvergent $F$-isocrystals on $Z$ of the same rank. The key Lemma \ref{lemma: gluing} then tells us that we can use the overconvergence of these objects to `glue' them together along a suitable punctured tube of $Z$ inside some formal lift of $S$, to get an overconvergent $F$-isocrystal on the whole of $S$.

Interestingly enough, the deduction of the result for arithmetic $\mathcal{D}^\dagger$-modules only uses the fact that the higher direct images are convergent $F$-isocrystals on $S$. However, the above strategy would not work if we tried to work everywhere in the convergent category, since we would not be able to `glue' along the stratification $V \hookrightarrow S \hookleftarrow  Z$. Thus it is important to be working with overconvergent objects from the beginning.

Theorem \ref{theo: main general} is weaker than the Deligne--Laumon theorem in one crucial aspect. Namely, the curve $f:U\rightarrow S$ is assumed to have a good compactification, i.e. a smooth compactification $\bar f:C\rightarrow S$ such that the complement $C\setminus U$ is \'etale over $S$; in \cite{Lau81} it is only assumed to be finite and flat. Our proof is completely different to that given in \cite{Lau81}, which uses the formalism of nearby and vanishing cycles in \'etale cohomology. While this article was being written, an analogue of the nearby and vanishing cycles formalism for $p$-adic coefficients appeared in \cite{Abe19}, which in fact was what allowed us to extend our main results to singular bases (at least when $k$ is perfect). It would be interesting to see whether or not Abe's theory can be used to give another proof of local acyclicity, more similar in spirit to the $\ell$-adic case, and that would be able to handle the more general case where the divisor at infinity is only finite flat over the base.

Let us now give a summary of the various parts of the article. In \S\ref{sec: background} we introduce some basic notations and definitions concerning rigid cohomology and the theory of arithmetic $\mathcal{D}^\dagger$-modules, in particular the approach to the  6 operations formalism taken in \cite{AC18a}. In \S\ref{sec: irreg} we recall the definition of the irregularity of a $p$-adic differential equation, as well as some results of Kedlaya and Kedlaya--Xiao concerning extending break decompositions in families. In \S\ref{sec: src and gp} we introduce the basic geometric setup that we will work with. We recall results of Kedlaya on generic higher direct images and state our first main result on the relative Monsky--Washnitzer cohomology of curves. In \S\ref{sec: unip bc} we investigate the cohomology of $\nabla$-modules over relative Robba rings, and prove a base change result under the assumption of constant irregularity, which in \S\ref{sec: bc for R10} is used to prove Theorem \ref{theo: main dagger} for the `lower' direct image $\mathbf{R}^0f_*$. Section \ref{sec: strong} is devoted to a rather grisly study of relative cohomology on tubes and punctured tubes, which forms the key input for the proof in \S\ref{sec: universal} of finiteness and base change for $\mathbf{R}^1f_*$ via the gluing argument outlined above. In \S\ref{sec: conv} we then use this, together with a little functional analysis, to deduce finiteness and base change for certain partially overconvergent cohomology groups, which in \S\ref{sec: Dmod2} then allows us to obtain our second main result, Theorem \ref{theo: main general}, by reduction to the smooth and affine case.

\subsection*{Acknowledgements} The author was supported by the Netherlands Organisation for Scientific Research (NWO). He would like to thank T. Abe and A. Shiho for helpful conversations regarding various parts of this article. He would also like to greatly thank the anonymous referee for a careful reading of the manuscript, resulting both in improved exposition and in several earlier mistakes being corrected.

\section{Background on rigid cohomology and arithmetic \texorpdfstring{$\mathcal{D}^\dagger$}{D}-modules}\label{sec: background}

Throughout this article, we will denote by $K$ (the ground field) a complete, normed field of characteristic $0$, by $\cur{V}$ its ring of integers and by $k$ (the residue field) its residue field, which will be assumed to be of characteristic $p>0$. From \S\ref{sec: src and gp} onwards we will assume that $K$ is discretely valued, and denote by $\varpi$ a uniformiser for $\cur{V}$. To begin with, however, we will need to make certain constructions more generally, in which case $\varpi$ will be a non-zero element of the maximal ideal $\mathfrak{m}$ of $\mathcal{V}$ (a pseudo-uniformiser). In \S\ref{sec: Dmod2} we will want to assume that $k$ is perfect, however, for the most part we will allow arbitrary $k$.  We will assume that $K$ admits a lifting $\sigma$ of the absolute $q=p^a$-power Frobenius on $k$, and fix such a $\sigma$.

In this first section we will recall some definitions and constructions in rigid cohomology and the theory of arithmetic $\mathcal{D}^\dagger$-modules, and review the 6 opertaions formalism for varieties and couples introduced in \cite{AC18a}. We will generally assume the reader has some basic familiarity with the theory of rigid cohomology, as developed in \cite{Ber96b,LS07}, and will mostly use this section for fixing definitions and notations. 

\begin{definition} \begin{enumerate}
\item A \emph{variety} over $k$ is a separated and finite type $k$-scheme. 
\item A \emph{formal scheme} over $\cur{V}$ is a $\varpi$-adic formal $\cur{V}$-scheme which is separated and topologically of finite type. 
\item A \emph{rigid variety} over $K$ is a rigid analytic space in the sense of Tate, which in addition is separated  over $\mathrm{Sp}(K)$.
\item A \emph{couple} $(X,Y)$ over $k$ is an open immersion $j:X\hookrightarrow Y$ of $k$-varieties.
\item A \emph{frame} over $\cur{V}$ is a triple $(X,Y,\mathfrak{P})$ consisting of a couple $(X,Y)$ and a closed immersion $Y\hookrightarrow \mathfrak{P}$ of formal $\mathcal{V}$-schemes.
\item An \emph{l.p. frame} over $\cur{V}$  is a quadruple $(X,Y,\mathfrak{P},\mathfrak{Q})$ such that $(X,Y,\mathfrak{P})$ is a frame, and $\mathfrak{P}\hookrightarrow \mathfrak{Q}$ is an open immersion of formal schemes, such that $\mathfrak{Q}$ is proper over $\mathcal{V}$.
\end{enumerate}
\end{definition}

A morphism of couples
\[ (X',Y')\rightarrow (X,Y)\]
is said to be flat (resp. smooth, \'etale) if $X'\rightarrow X$ is, and proper (resp. finite) if $Y'\rightarrow Y$ is. It is said to be Cartesian if the diagram
\[  \xymatrix{ X' \ar[r] \ar[d] & Y' \ar[d] \\ X \ar[r] & Y   } \]
is Cartesian. A couple $(X,Y)$ is said to be smooth or proper if the natural morphism
\[ (X,Y) \rightarrow (\spec{k},\spec{k})\]
is. Similarly, a morphism of frames
\[ (X',Y',\mathfrak{P}')\rightarrow (X,Y,\mathfrak{P}) \]
is said to be flat (resp. smooth, \'etale) if $\mathfrak{P}'\rightarrow \mathfrak{P}$ is flat (resp. smooth, \'etale) in a neighbourhood of $X'$, and proper (resp. finite) if $Y'\rightarrow Y$ is. It is said to be Cartesian if both squares in the diagram
\[  \xymatrix{ X' \ar[r] \ar[d] & Y' \ar[d] \ar[r] & \mathfrak{P}' \ar[d]  \\ X \ar[r] & Y  \ar[r] & \mathfrak{P}  } \]
are Cartesian. A frame $(X,Y,\mathfrak{P})$ is flat (resp. smooth, \'etale, proper, finite) if the natural morphism
\[ (X,Y,\mathfrak{P}) \rightarrow (\spec{k},\spec{k},\spf{\mathcal{V}})\]
is.  We will not need to define particular properties of morphisms of l.p. frames. 

If $(X,Y)$ is a pair, we will denote by $\mathrm{Isoc}((X,Y)/K)$ the category of isocrystals on $X$, overconvergent along $Y$. A \emph{Frobenius structure} on an isocrystal $E$ is an isomorphism
\[ \varphi: F^{n*}E\isomto E \]
for some $n\geq 1$, and we will denote by
\[ \mathrm{Isoc}_F((X,Y)/K) \subset \mathrm{Isoc}((X,Y)/K) \]
the full subcategory consisting of iterated extensions of subquotients of objects admitting Frobenius structures. 

\begin{lemma} A partially overconvergent isocrystal $E\in  \mathrm{Isoc}((X,Y)/K)$ lies in $E\in  \mathrm{Isoc}_F((X,Y)/K) $ if and only if the irreducible constituents of $E$ admit Frobenius structures.
\end{lemma} 

\begin{proof}
One direction is clear, so suppose that $E\in  \mathrm{Isoc}_F((X,Y)/K)$; we must show that its irreducible constituents admit Frobenius structures. We are free to replace $E$ by any subquotient, in particular we may therefore assume that $E$ itself is a subquotient of some $E'$ admitting a Frobenius structure. 

Choose $n$ such that $F^{n*}E'\cong E'$, thus $F^{n*}$ permutes the (finite) set $\mathcal{I}$ of isomorphism classes of irreducible constituents of $E'$. By possibly increasing $n$, then, we can assume that in fact $F^{n*}$ acts trivially on $\mathcal{I}$. In particular, if $E_0\subset E\subset E'$ is irreducible, then $F^{n*}E_0\cong E_0$ and we are done. 
\end{proof}

Thus $\mathrm{Isoc}_F((X,Y)/K)$ is the thick abelian subcategory of $\mathrm{Isoc}((X,Y)/K)$ generated by objects admitting Frobenius structures. When $X=Y$ we will write $\mathrm{Isoc}_F(X/K)$ and $\mathrm{Isoc}(X/K)$ respectively. When $Y$ is proper over $k$ these do not depend on $Y$ and we will write $\mathrm{Isoc}^\dagger_F(X/K)$ and $\mathrm{Isoc}^\dagger(X/K)$ instead. We will refer to objects in $\mathrm{Isoc}_F((X,Y)/K)$ as `$F$-able isocrystals'.

\subsection{Monsky--Washnitzer cohomology}

For most of this article, we will use the Monsky--Washnitzer approach to rigid cohomology, the basics of which we very briefly review here. For the reader wishing for more details, a better introduction is given in \cite[\S\S2-3]{Ked06a}. 

\begin{definition}[\cite{GK00}] A $K$-dagger algebra is a $K$-algebra isomorphic to a quotient $K\weak{x_1,\ldots,x_n}/I$ of an overconvergent power series algebra over $K$.
\end{definition}

We denote by $\Norm{\cdot}_\mathrm{sup}$ the supremum seminorm on a $K$-dagger algebra, if $A$ is reduced then this is a norm \cite[Theorem 1.7]{GK00}. For any real number $\lambda>1$, we consider the subalgebra 
\[ K\tate{\lambda^{-1} x_1,\ldots,\lambda^{-1}x_n }:= \left\{ \left.\sum_{i_1,\ldots,i_n\geq 0} a_{i_1,\ldots,i_n}x_1^{i_d}\ldots x_n^{i_n} \right\vert   \norm{a_{i_1,\ldots,i_n}}\lambda^{i_1+\ldots +i_n}\rightarrow 0\right\} \subset K\weak{x_1,\ldots,x_n} \]
of series converging for $\norm{x_i}\leq \lambda$;
the image of any such subalgebra under any surjection $K\weak{x_1,\ldots,x_n}\twoheadrightarrow A$ will be called a fringe algebra of $A$. This admits a norm $\Norm{\,\cdot\,}_\lambda$ coming from the $\lambda$-Gauss norm on $K\tate{\lambda^{-1} x_1,\ldots,\lambda^{-1}x_n }$. We will let $\otimes^\dagger$ denote the weakly completed tensor product of $K$-dagger algebras. 

For $A$ reduced, we denote by $A^\mathrm{int}\subset A$ the subring of integral elements, consisting of those elements of supremum norm $\leq 1$, and by $\overline{A}$ the reduction of $A$, that is the quotient of $A^\mathrm{int}$ by the ideal of topologically nilpotent elements. We let $\widehat{A}$ denote the completion of $A$ with respect to this ideal.

\begin{definition}[Definition 3.2.1, \cite{Ked06a}] A $K$-dagger algebra $A$ is said to be of MW-type if it is integral, and its reduction $\overline{A}$ is smooth over $k$.
\end{definition}

If $A$ is a $K$-dagger algebra, we will let $\Omega^1_{A/K}$ denote the module of $p$-adically continuous differentials, a $\nabla$-module over $A$ is then by definition a finitely generated $A$-module $M$ together with an integrable connection
\[ \nabla: M \rightarrow M\otimes_A \Omega^1_{A/K}. \]
If $A$ is of MW-type then the $A$-module underlying $M$ is automatically projective \cite[Lemma 3.3.4]{Ked06b}. If $A$ is a MW-type dagger algebra, with reduction $\overline{A}$, then there is a fully faithful functor
\[ \mathrm{Isoc}^\dagger(\spec{\overline{A}}/K) \rightarrow \underline{\mathrm{Mod}}^\nabla_A   \]
from overconvergent isocrystals on $\overline{A}$ to $\nabla$-modules on $A$, which we call `realisation on $A$' \cite[Proposition 8.1.13]{LS07}. A $\nabla$-module is called overconvergent if it is in the essential image of this functor. This construction is compatible with pullback, and hence whenever we fix a lift $\sigma$ of Frobenius on $A$, we obtain a functor
\[ F\text{-}\mathrm{Isoc}^\dagger(\spec{\overline{A}}/K) \rightarrow \underline{\mathrm{Mod}}^{\pn}_A  \]
from overconvergent isocrystals on $\spec{\overline{A}}$ to $\pn$-modules over $A$, that is $\nabla$-modules equipped with a horizontal isomorphism $\varphi:\sigma^*M\isomto M$. This functor is an equivalence of categories \cite[Theorem 8.3.10]{LS07}. We will say that $M \in \underline{\mathrm{Mod}}^\nabla_A$ admits a Frobenius structure if there exists an isomorphism $\varphi:\sigma^{n*}M\isomto M$ for some $n\geq 1$, any such $M$ is automatically overconvergent. We say that $M$ is $F$-able if its irreducible constituents admit Frobenius structures, this is equivalent to being in the essential image of 
\[ \mathrm{Isoc}_F^\dagger(\spec{\overline{A}}/K) \rightarrow \underline{\mathrm{Mod}}^\nabla_A .\]
If $\bar{g}\in \overline{A}$ and $g\in A^\mathrm{int}\subset A$ is a lift of $\bar{g}$, then we may form the MW-type dagger algebra 
\[ A\weak{g^{-1}} := \frac{A\weak{x}}{(gx-1)}\]
whose reduction can be identified with the localisation $\overline{A}[\bar{g}^{-1}]$. We call any such $A\weak{g^{-1}}$ a `dagger localisation' of $A$, and we can use the specialisation map
\[ \mathrm{sp}:\mathrm{maxSpec}(	A) \rightarrow \spec{\overline{A}}  \]
to see that (as an $A$-algebra) this only depends on $\bar{g}$ up to canonical isomorphism. A collection of dagger localisations $\{ A\rightarrow A_i\}_{i\in I}$ will be called a dagger open cover if $\mathrm{maxSpec}(A)=\bigcup_{i\in I} \mathrm{maxSpec}(A_i)$. Again using the specialisation map we see that this will happen if and only if the reductions $\spec{\overline{A}_i}$ form a Zariski open cover of $\spec{\overline{A}}$.

\begin{definition}[Definition 2.5.3, \cite{Ked06a}] Let $u$ be a variable. The Robba ring $\mathcal{R}_A^u$ over $A$ consists of those series
\[ \sum_{i\in \Z} a_iu^i \in A\pow{u,u^{-1}} \]
satisfying the following convergence condition:
\begin{itemize}
\item[$(\star)$] there exists $\eta<1$ such that for all $\eta<\rho<1$ there  there exists some fringe algebra $A_\lambda\subset A$ such that $a_i \in A_\lambda$ for all $i$ and
\begin{align*} \Norm{a_i}_\lambda \rho^{i} &\rightarrow 0\;\;\text{ as }\;\;i \rightarrow \pm \infty.
\end{align*}
\end{itemize} 
The plus part $\mathcal{R}_A^{u+}$ of the Robba ring over $A$ consists of those series for which $a_i=0$ for all $i<0$.
\end{definition}

The relative Robba ring $\mathcal{R}^u_A$ over a dagger algebra can be expressed as a colimit of a limit of a colimit of Banach $K$-algebras, it therefore comes equipped with a natural locally convex topology simply by taking these (co)limits in the category of locally convex topological vector spaces over $K$. With this topology, the module of continuous ($K$-linear) derivations of $\mathcal{R}_A^u$ is isomorphic to
\[ \Omega^1_{A/K}\otimes_A \mathcal{R}_A^u \oplus \mathcal{R}_A^u\cdot du. \]
A $\nabla$-module over $\mathcal{R}_A^u$ is defined to be a finitely presented module together with an integrable connection relative to $K$. The notion of a $\nabla$-module over $\mathcal{R}_A^{u+}$ is defined similarly.

\begin{remark} For any of the rings $A,\mathcal{R}_A^u$, $\mathcal{R}_A^{u+}$ (and some others that we will introduce later) a $\nabla$-module will always (unless explicitly stated otherwise) mean a $\nabla$-module relative to $K$, whose underlying module is finitely presented. Extra adjectives, such as projective, stably free, free \&c. are understood to apply to the underling module. 
\end{remark}

\begin{definition} \label{defn: mod} A frame $(X,\overline{X},\mathfrak{X})$ is called a Monsky--Washnitzer frame  (or an MW frame for short) if:
\begin{enumerate}
\item $X$ is smooth, affine and connected;
\item $X\hookrightarrow \overline{X}$ has dense image;
\item $\mathfrak{X}$ is projective over $\cur{V}$ (and in particular algebraisable);
\item the map $\overline{X}\rightarrow \mathfrak{X}_k$ is an isomorphism.
\end{enumerate}
\end{definition}

If $(X,\overline{X},\mathfrak{X})$ is a MW frame, and $X=\spec{\overline{A}}$, then $A= \Gamma(\mathfrak{X}_K,j_X^\dagger \mathcal{O}_{\mathfrak{X}_K})$ is a $K$-dagger algebra of MW-type, whose reduction is exactly $\overline{A}$. We will call such an $A$ an MW lift of $\overline{A}$. 
%
We will need the following useful result about lifting \'etale morphisms from characteristic $p$ to characteristic $0$.

\begin{lemma} \label{lemma: finite etale lift} Assume that $K$ is discretely valued, let $A,B$ be dagger algebras of MW-type, and $\bar\alpha:\overline{A}\rightarrow \overline{B}$ a finite \'etale morphism between their reductions. Then there exists a finite \'etale morphism $\alpha:A\rightarrow B$ lifting $\bar\alpha$.
\end{lemma}

\begin{proof}
Note that by \cite[Theorem 2.4.4]{vdP86} any two MW lifts of either $\overline{A}$ or $\overline{B}$ are isomorphic as $K$-dagger algebras, so in fact it suffices to construct \emph{some} MW lifts $A'$ and $B'$ of $\overline{A}$ and $\overline{B}$ respectively, together with a finite \'etale morphism $\alpha':A '\rightarrow B'$ lifting $\bar\alpha$. That this is possible follows from \cite[Th\'eor\`eme 3]{Ete02a}.
\end{proof}

While it will be important for us to only consider MW lifts arising from MW frames (in order to make certain `geometric' constructions in \S\ref{sec: strong}), we will need to know that we can change the frame with impunity.

\begin{lemma} \label{lemma: change frame} Let $A,A'$ be MW lifts of a smooth $k$-algebra $\overline{A}$, and $B,B'$ MW lifts of a smooth $k$-algebra $\overline{B}$. Suppose we are given a smooth morphism $\overline{A}\rightarrow \overline{B}$ together with lifts $A\rightarrow B$ and $A'\rightarrow B'$. Then there exists a commutative diagram
\[ \xymatrix{ B \ar[r]^-{\cong} & B' \\ A \ar[u]\ar[r]^-{\cong}  & A' \ar[u] } \]
of $K$-algebras, with both horizontal arrows isomorphisms and inducing the identity on $\overline{A}$ and $\overline{B}$ respectively.
\end{lemma}

\begin{proof}
Fixing an isomorphism $A\rightarrow A'$ inducing the identity on $\overline{A}$, and base changing $B$ along this morphism we can assume that $A=A'$. Since the morphism $\overline{A}\rightarrow \overline{B}$ is smooth, we can certainly find an isomorphism $\widehat{B}\rightarrow \widehat{B}'$ of affinoid completions, fixing the affinoid completion $\widehat{A}$ and inducing the identity on $\overline{B}$. Hence we may apply \cite[Corollaey 2.4.3]{vdP86} to produce a morphism $B\rightarrow B'$ of $A$-algebras inducing the identity on $\overline{B}$, which has to be an isomorphism by \cite[Theorem 3.2]{MW68}. 
\end{proof}

\subsection{Arithmetic \texorpdfstring{$\mathcal{D}^\dagger$}{D}-modules on varieties and couples}

The purpose of this section is to recall how the 6 operations formalism works for arithmetic $\mathcal{D}^\dagger$-modules on $k$-varieties and couples, as described in \cite{AC18a}. We will assume that the ground field $K$ is discretely valued, and that residue field $k$ is perfect (these assumptions will be dropped again at the beginning of \S\ref{sec: irreg}).

\begin{definition} \begin{enumerate} \item A variety $X/k$ is realisable if there exists a frame $(X,Y,\mathfrak{P})$  such that $\mathfrak{P}$ is smooth and proper over $\cur{V}$.
\item A couple $(X,Y)/k$ is realisable if there exists an l.p. frame $(X,Y,\mathfrak{P},\mathfrak{Q})$ such that $\mathfrak{Q}$ is smooth over $\mathcal{V}$. 
\end{enumerate}
\end{definition}

Note that both these conditions are marginally stronger than might be expected from the definition of rigid cohomology. If $\mathfrak{P}$ is a smooth formal $\mathcal{V}$-scheme, we let
\[ \mathrm{Hol}(\mathcal{D}^\dagger_{\mathfrak{P},\Q})\;\;\;\;\text{ and }\;\;\;\;D^b_\mathrm{hol}(\mathcal{D}^\dagger_{\mathfrak{P},\Q}) \]
denote the categories of overholonomic (complexes of) $\mathcal{D}^\dagger_{\mathfrak{P},\Q}$-modules respectively, in the sense of \cite{Car09b}. We denote by
\[ \mathrm{Hol}_F(\mathcal{D}^\dagger_{\mathfrak{P},\Q}) \subset \mathrm{Hol}(\mathcal{D}^\dagger_{\mathfrak{P},\Q}) \]
the thick abelian subcategory generated by objects which admit an $F^n$-Frobenius structure for some $n\geq 1$, and 
\[ D^b_{\mathrm{hol},F}(\mathcal{D}^\dagger_{\mathfrak{P},\Q}) \subset D^b_\mathrm{hol}(\mathcal{D}^\dagger_{\mathfrak{P},\Q})\]
the full subcategory of objects whose cohomology sheaves lie in $\mathrm{Hol}_F(\mathcal{D}^\dagger_{\mathfrak{P},\Q})$. If $\mathbb{X}=(X,Y)$ is a realisable couple, and $(X,Y,\mathfrak{P},\mathfrak{Q})$ is an l.p. frame with $\mathfrak{Q}$ smooth over $\mathcal{V}$, then Abe and Caro define the category
\[ D^b_{\mathrm{hol},F}(\mathbb{X}/K) \subset  D^b_{\mathrm{hol},F}(\mathcal{D}^\dagger_{\mathfrak{P},\Q}) \]
of overholonomic complexes of $\mathcal{D}^\dagger$-modules on $\mathbb{X}$ to be the full subcategory of overholonomic complexes of $\mathcal{D}^\dagger_{\mathfrak{P},\Q}$-modules $\mathcal{M}$ which satisfy 
\[ \mathcal{M}\isomto \mathbf{R}\underline{\Gamma}_Y^\dagger \mathcal{M} \;\;\;\;\text{ and }\;\;\;\; \mathcal{M}\isomto (^\dagger Y\setminus X)\mathcal{M}.\]
Here $\mathbf{R}\underline{\Gamma}_Y^\dagger$ and $(^\dagger Y\setminus X)$ are the functors of support and overconvergent sections defined in \cite[\S2.2]{Car04}. This does not depend on the choice of l.p. frame $(X,Y,\mathfrak{P},\mathfrak{Q})$ extending $\mathbb{X}$ \cite[\S1.1, Definition on p. 885]{AC18a}. There is a dual functor
\[ \mathbf{D}_\mathbb{X}: D^b_{\mathrm{hol},F}(\mathbb{X}/K)^{\mathrm{op}} \rightarrow D^b_{\mathrm{hol},F}(\mathbb{X}/K) \]
and a tensor product functor
\[ -\widetilde{\otimes}_\mathbb{X} - : D^b_{\mathrm{hol},F}(\mathbb{X}/K) \times D^b_{\mathrm{hol},F}(\mathbb{X}/K) \rightarrow D^b_{\mathrm{hol},F}(\mathbb{X}/K)\]
which are defined as follows. Let $(X,Y,\mathfrak{P},\mathfrak{Q})$ be an l.p. frame extending $\mathbb{X}$ with $\mathfrak{Q}$ smooth over $\mathcal{V}$. Then 
\[ \mathbf{D}_\mathbb{X}:= \mathbf{R}\underline{\Gamma}_Y^\dagger \circ (^\dagger Y\setminus X) \circ \mathbf{D}_\mathfrak{P}. \]
where $\mathbf{D}_\mathfrak{P}$ is the dual functor as defined in \cite[\S4]{Ber02}. Similarly, 
\[ \mathcal{M}\widetilde{\otimes}_\mathbb{X} \mathcal{N} := \mathcal{M} \otimes^{\mathbf{L},\dagger}_{\mathcal{O}_{\mathfrak{P},\Q}} \mathcal{N}[-\dim\mathfrak{P}], \]
where $\otimes^{\mathbf{L},\dagger}_{\mathcal{O}_{\mathfrak{P},\Q}}$ is the tensor product defined in \cite{Ber02}.\footnote{The reason for using the notation $\widetilde{\otimes}$ is that the tensor product defined here is strictly speaking analogous to the `twisted' tensor product $(\mathcal{F},\mathcal{G})\mapsto \mathbf{D}(\mathbf{D}(\mathcal{F})\otimes^\mathbf{L}_{\Q_\ell} \mathbf{D}(\mathcal{G}))$ for constructible $\Q_\ell$ sheaves, where $\mathbf{D}$ is the Verdier dual functor.} The resulting functors only depend on $\mathbb{X}$ up to canonical isomorphism \cite[\S1.1.6]{AC18a}. If $u:\mathbb{X}'\rightarrow \mathbb{X}$ is a morphism of couples then there are functors
\[ u^!,u^+: D^b_{\mathrm{hol},F}(\mathbb{X}/K)\rightarrow D^b_{\mathrm{hol},F}(\mathbb{X}'/K), \]
and if $u$ is proper there are functors
\[ u_!,u_+: D^b_{\mathrm{hol},F}(\mathbb{X}'/K) \rightarrow D^b_{\mathrm{hol},F}(\mathbb{X}/K). \]
These are defined as follows: choose a morphism
\[ \tilde{u}: (X',Y',\mathfrak{P}',\mathfrak{Q}') \rightarrow (X,Y,\mathfrak{P},\mathfrak{Q}) \]
of l.p. frames extending $u$, and set
\begin{align*} u^! : = \mathbf{R}\underline{\Gamma}_{Y'}^\dagger \circ (^\dagger Y'\setminus X') \circ \tilde{u}^!,\;\;&\;\; u^+ = \mathbf{D}_{\mathbb{X}'} \circ  u^! \circ  \mathbf{D}_{\mathbb{X}}\\
u_+=\tilde{u}_+ \;\;&\;\; u_! = \mathbf{D}_{\mathbb{X}} \circ  u_+ \circ  \mathbf{D}_{\mathbb{X}'},
\end{align*}
where $\bar{u}_+$ and $\bar{u}^!$ are the pushforward and pullback functors defined in \cite[\S4]{Ber02}. All these functors commute with Frobenius pullback \cite[\S4]{Ber02}. Both $(u^+,u_+)$ and $(u_!,u^!)$ are adjoint pairs \cite[Lemma 1.1.10]{AC18a}, and if
\[\xymatrix{ \mathbb{Y}' \ar[r]^{a'}\ar[d]_{u'} & \mathbb{X}'\ar[d]^u \\ \mathbb{Y} \ar[r]^a & \mathbb{X} } \]
is a Cartesian morphism of couples, with $u$ proper, then by \cite[Lemma 1.3.10]{AC18a} there is a natural isomorphism
\[ a^!u_+ \cong u'_+a'^! \]
of functors
\[ D^b_{\mathrm{hol},F}(\mathbb{Y}'/K) \rightarrow  D^b_{\mathrm{hol},F}(\mathbb{Y}/K) .\]
The triangulated category $D^b_{\mathrm{hol},F}(\mathbb{X}/K)$ admits a `holonomic' $t$-structure \cite[\S1.2]{AC18a}, whose heart we will denote by $\mathrm{Hol}_F(\mathbb{X}/K)$. The duality functor $\mathbf{D}_\mathbb{X}$ is exact with respect to this $t$-structure \cite[Proposition 1.3.1]{AC18a}, and hence induces an anti-equivalence
\[ \mathbf{D}_\mathbb{X}:\mathrm{Hol}_F(\mathbb{X}/K)^\mathrm{op} \isomto \mathrm{Hol}_F(\mathbb{X}/K). \]
When $\mathbb{X}$ is \emph{smooth}, there exists a fully faithful functor
\[ \mathrm{sp}_+:\mathrm{Isoc}_F(\mathbb{X}/K) \rightarrow \mathrm{Hol}_F(\mathbb{X}/K) \]
constructed in \cite{Car12} which is compatible with duality in the sense that there are isomorphisms
\[ \mathrm{sp}_+E^\vee \cong  \mathbf{D}_\mathbb{X} \mathrm{sp}_+E, \]
natural in $E$.

If $\mathbb{X}=(X,\overline{X})$ is a couple proper over $k$, then $D^b_{\mathrm{hol},F}(\mathbb{X}/K)$ and $\mathrm{Hol}_F(\mathbb{X}/K)$ only depend on $X$ up to canonical equivalence \cite[\S1.3.14]{AC18a}. In this case we write $D^{b,\dagger}_{\mathrm{hol},F}(X/K)$ and $\mathrm{Hol}^\dagger_F(X/K)$ respectively. For $\mathbb{X}=(X,X)$ then we write $D^{b}_{\mathrm{hol},F}(\mathbb{X}/K)= D^{b}_{\mathrm{hol},F}(X/K)$ and $\mathrm{Hol}_F(\mathbb{X}/K) = \mathrm{Hol}_F(X/K)$. These categories correspond to overconvergent and convergent holonomic modules on $X/K$ respectively, and there are obvious forgetful functors
\begin{align*}  D^{b,\dagger}_{\mathrm{hol},F}(X/K) \rightarrow D^{b}_{\mathrm{hol},F}(X/K) \;\;\;\;&\text{ and }\;\;\;\;\mathrm{Hol}^\dagger_F(X/K) \rightarrow \mathrm{Hol}_F(X/K) \\
\mathcal{M} &\mapsto \hat{\mathcal{M}}.
\end{align*}
The above defined duality and tensor product functors induce functors $\mathbf{D}_X$ and $\-\widetilde{\otimes }_X$ on $D^b_{\mathrm{hol},F}(X/K)$ and $D^{b,\dagger}_{\mathrm{hol},F}(X/K)$, such that $\mathbf{D}_X$ is exact for the holonomic $t$-structures. 
If $u:X'\rightarrow X$ is a morphism of varieties, we therefore have functors
\[ u^!,u^+: D^{b,(\dagger)}_{\mathrm{hol},F}(X/K)\rightarrow D^{b,(\dagger)}_{\mathrm{hol},F}(X'/K) \]
and
\[ u_!,u_+: D^{b,\dagger}_{\mathrm{hol},F}(X'/K)\rightarrow D^{b,\dagger}_{\mathrm{hol},F}(X/K) ,\]
as well as 
\[ u_!,u_+: D^{b}_{\mathrm{hol},F}(X'/K)\rightarrow D^{b}_{\mathrm{hol},F}(X/K)  \]
whenever $u$ is proper. These have exactly the same properties as in the case of couples, and there is an analogous base change result. If $X$ is smooth, then we have full subcategories
\[ D^{b,(\dagger)}_{\mathrm{isoc},F}(X/K) \subset D^{b,(\dagger)}_{\mathrm{hol},F}(X/K)
\]
consisting of objects whose cohomology sheaves are in the essential image of 
\[ \mathrm{sp}_+: \mathrm{Isoc}_F^{(\dagger)}(X/K) \rightarrow \mathrm{Hol}_F^{(\dagger)}(X/K).\]
At least in the overconvergent case, it is explained how to extend all these definitions to the not-necessarily-realisable case in \cite{Abe18a}. That is, Abe shows that the category $\mathrm{Hol}^\dagger_F(X/K)$ is of a Zariski-local nature for realisable varieties, and thus for a general variety we may define $\mathrm{Hol}^\dagger_F(X/K)$ by taking an open affine cover and gluing (affine varieties being realisable). We can then define
\[ D^{b,\dagger}_{\mathrm{hol},F}(X/K) := D^b(\mathrm{Hol}^\dagger_F(X/K)),\]
which is justified by the main result of \cite{AC18}, showing that this recovers the previous definition in the realisable case. Abe explains in \cite[\S2.3]{Abe18a} how to define the 6 functors $u^+,u_+,u^!,u_!,\widetilde{\otimes}$ and $\mathbf{D}$ in the non-realisable case, and shows that all the same properties hold. Again, if $X$ is smooth then we have the fully faithful functor
\[\mathrm{sp}_+: \mathrm{Isoc}_F^{\dagger}(X/K) \rightarrow \mathrm{Hol}_F^{\dagger}(X/K)  \]
and the corresponding full subcategory  $D^{b,\dagger}_{\mathrm{isoc},F}(X/K) \subset D^{b,\dagger}_{\mathrm{hol},F}(X/K)$.

If $X$ is not assumed to be smooth, then we can still consider overconvergent isocrystals as holonomic complexes on $X$ using the approach of \cite{Abe19}. Indeed, in \cite[\S1.3]{Abe18a} Abe defines another $t$-structure on $D^{b,\dagger}_{\mathrm{hol},F}(X/K)$, called the constructible $t$-structure. The heart of this $t$-structure is denoted $\mathrm{Cons}(X/K)$, and the corresponding cohomology objects by ${}^c\hspace{-1mm}\mathcal{H}^i$. The pullback functor $u^+$ is $t$-exact for the constructible $t$-structure \cite[Lemma 1.3.4]{Abe18a}. Abe constructs in \cite[\S3]{Abe19} a fully faithful functor
\[ \rho: \mathrm{Isoc}^\dagger_F(X/K) \rightarrow \mathrm{Cons}(X/K) \] 
such that:
\begin{itemize}
\item$u^+\rho(E) \cong \rho(u^*E)$ for any morphism $u:X'\rightarrow X$;
\item $\rho(E) \cong \mathrm{sp}_+E[-\dim X]$ whenever $X$ is smooth.
\end{itemize}
We can therefore define
\[ D^{b,\dagger}_{\mathrm{isoc},F}(X/K) \subset D^{b,\dagger}_{\mathrm{hol},F}(X/K) \]
to be the full subcategory whose \emph{constructible} cohomology objects are in the essential image of $\rho$. This coincides with the previous definition when $X$ is smooth, in which case
\[D^{b,\dagger}_{\mathrm{isoc},F}(X/K) \subset D^{b,\dagger}_{\mathrm{hol},F}(X/K)  \]
is stable under $\mathbf{D}_X$, however, this stability does not hold in general. It will be helpful to isolate the following result, which is simply a restatement of various results of Caro and Caro--Tsuzuki.

\begin{lemma} \label{lemma: isoc conv} Let $X/k$ be a smooth, realisable variety, and $\mathcal{M}\in D^{b,\dagger}_{\mathrm{hol},F}(X/K)$. Then $\mathcal{M}\in D^{b,\dagger}_{\mathrm{isoc},F}(X/K)$ if and only if $\hat{\mathcal{M}}\in D^b_{\mathrm{isoc},F}(X/K)$. 
\end{lemma}

\begin{proof} The question is local on $X$, which we may assume to be affine. By further localising if necessary, we may assume that there exists an immersion $X\hookrightarrow \mathfrak{P}$, with $\mathfrak{P}$ smooth and proper over $\mathcal{V}$, such that there exists a divisor $T$ of $P:=\mathfrak{P}_k$ with $X$ is closed in $P\setminus T$. Let $\mathfrak{U} \subset \mathfrak{P}$ be an open formal subscheme such that $X$ is closed inside $\mathfrak{U}$. In this case, the holonomic $t$-structure on
\[  D^{b,\dagger}_{\mathrm{isoc},F}(X/K) \subset D^{b}_{\mathrm{coh}}(\mathcal{D}^\dagger_{\mathfrak{P},\Q})\]
is the restriction of the obvious $t$-structure on $D^{b}_{\mathrm{coh}}(\mathcal{D}^\dagger_{\mathfrak{P},\Q})$ by \cite[Proposition 1.3.13]{AC18a}. In particular, the $\mathscr{D}^\dagger_{\mathfrak{P},\Q}$-module $\mathcal{H}^i(\mathcal{M})$ is overholonomic for all $i$; we may therefore replace $\mathcal{M}$ by $\mathcal{H}^i(\mathcal{M})$, and thus assume that $\mathcal{M}$ is an overholonomic $\mathcal{D}^\dagger_{\mathfrak{P},\Q}$-module. By \cite[Lemme 1.2.13]{Car15a}, $\mathcal{M}$ arises by restriction of scalars from a $\mathcal{D}^\dagger_{\mathfrak{P},\Q}(^\dagger T)$-module, which we will denote by the same letter $\mathscr{M}$. Let $\mathbf{D}_{\mathfrak{P},T}$ denote the dual functor for $\mathcal{D}^\dagger_{\mathfrak{P},\Q}(^\dagger T)$-modules in the sense of \cite{Vir00}.

By \cite[Th\'eor\`eme 6.1.11]{Car11} $\mathcal{M}$ is in the essential image of $\mathrm{sp}_+$ if and only if it is an `overcoherent isocrystal' in the sense of \cite[D\'efinition 6.2.1]{Car06a}, that is, if both $\mathcal{M}$ and $\mathbf{D}_{\mathfrak{P},T}(\mathcal{M})\cong (T^\dagger)\mathbf{D}_{\mathfrak{P}}(\mathcal{M})$ are overcoherent as $\mathcal{D}^\dagger_{\mathfrak{P},\Q}(^\dagger T)$-modules, and if the restriction of $\mathcal{M}$ to $\mathfrak{U}$ is in the essential image of $\mathrm{sp}_+$. To prove the lemma, then, we need to explain why the first two conditions are automatically satisfied.

In other words, we want to show that if $\mathcal{M}$ is a coherent $\mathcal{D}^\dagger_{\mathfrak{P},\Q}(^\dagger T)$-module, overholonomic as a $\mathcal{D}^\dagger_{\mathfrak{P},\Q}$-module, then both $\mathcal{M}$ and $(T^\dagger)\mathbf{D}_{\mathfrak{P}}(\mathcal{M})$ are overcoherent as $\mathcal{D}^\dagger_{\mathfrak{P},\Q}(^\dagger T)$-modules. However, these are both overholonomic as $\mathcal{D}^\dagger_{\mathfrak{P},\Q}$-modules, and since we are dealing with objects admitting Frobenius structures (or rather, the thick abelian subcategory generated by objects admitting a Frobenius structure), we may therefore appeal to \cite[Theorem 2.3.17]{CT12} to conclude.
\end{proof}

In the definition of the functor
\[ u_+: D^b_{\mathrm{hol},F}(\mathbb{X}'/K)\rightarrow D^b_{\mathrm{hol},F}(\mathbb{X}/K) \]
coming from a morphism of pairs $u:\mathbb{X}'\rightarrow \mathbb{X}$, we had to choose a morphism of l.p. frames $(X',Y',\mathfrak{P}',\mathfrak{Q}') \rightarrow (X,Y,\mathfrak{P},\mathfrak{Q})$ extending $u$. Note that neither the formal schemes $\mathfrak{Q}$ and $\mathfrak{Q}'$, nor the morphism between them, play any role in the definition of either the categories or the functors involved, however, one still needs to know that they exist. It will be important for us to shows that in certain situations we can completely ignore this technicality, and work simply with immersions of couples into smooth formal $\mathcal{V}$-schemes.

Our setup will be the following. We will take a base couple $\mathbb{S}=(S,S)$, with $S$ smooth and affine, and admitting a smooth, affine lift $\mathfrak{S}$ to a formal scheme over $\mathcal{V}$. We assume that we are given a smooth and projective morphism $\tilde{u}: \mathfrak{X}\rightarrow \mathfrak{S}$, and an open immersion
\[ U\hookrightarrow X:=\mathfrak{X}_k \]
of $k$-varieties. We let $\mathbb{U}=(U,X)$, and we assume that both $\mathbb{U}$ and $\mathbb{S}$ are realisable as couples. The proper morphism $\tilde u$ induces a functor
\[ \tilde{u}_+:D^b_{\mathrm{hol}}(\mathcal{D}^\dagger_{\mathfrak{X},\Q}) \rightarrow D^b_{\mathrm{hol}}(\mathcal{D}^\dagger_{\mathfrak{S},\Q})\]
between the categories of complexes of overholonomic $\mathcal{D}^\dagger$-modules on $\mathfrak{X}$ and $\mathfrak{S}$ respectively, as in \cite{Car09b}. We define
\[ D^b_{\mathrm{isoc}}(\mathcal{D}^\dagger_{\mathfrak{S},\Q}) \subset D^b_{\mathrm{hol}}(\mathcal{D}^\dagger_{\mathfrak{S},\Q}) \]
to be the subcategory whose cohomology sheaves are coherent as $\mathcal{O}_{\mathfrak{S},\Q}$-modules. Again, the following lemma is just a rephrasing of various results of Caro. 

\begin{lemma} \label{lemma: good lifts} There are fully faithful embeddings
\[ D^b_{\mathrm{hol},F}(\mathbb{U}/K)\hookrightarrow D^b_{\mathrm{hol}}(\mathcal{D}^\dagger_{\mathfrak{X},\Q})\;\;\;\;\text{ and }\;\;\;\;D^b_{\mathrm{hol},F}(\mathbb{S}/K)\hookrightarrow D^b_{\mathrm{hol}}(\mathcal{D}^\dagger_{\mathfrak{S},\Q})\]
such that the diagram
\[ \xymatrix{ D^b_{\mathrm{hol},F}(\mathbb{U}/K) \ar[r]\ar[d]_{u_+} &  D^b_{\mathrm{hol}}(\mathcal{D}^\dagger_{\mathfrak{X},\Q}) \ar[d]^{\tilde{u}_+}\\ 
D^b_{\mathrm{hol},F}(\mathbb{S}/K) \ar[r] &  D^b_{\mathrm{hol}}(\mathcal{D}^\dagger_{\mathfrak{S},\Q}) 
} \]
is $2$-commutative. Moreover, the square
\[ \xymatrix{ D^b_{\mathrm{isoc},F}(\mathbb{S}/K) \ar[r]\ar[d] &  D^b_{\mathrm{isoc}}(\mathcal{D}^\dagger_{\mathfrak{S},\Q}) \ar[d]\\ 
D^b_{\mathrm{hol},F}(\mathbb{S}/K) \ar[r] &  D^b_{\mathrm{hol}}(\mathcal{D}^\dagger_{\mathfrak{S},\Q}) 
} \]
is $2$-Cartesian.
\end{lemma}

\begin{proof}
Since $\mathfrak{S}$ is affine, there exists an immersion
\[ \mathfrak{S} \overset{i}{\hookrightarrow} \widehat{\A}^N_\mathcal{V}\hookrightarrow \widehat{\P}^N_\mathcal{V} \]
for some $N$. Similarly, since $\mathfrak{X}\rightarrow \mathfrak{S}$ is projective, we can extend this to a commutative diagram
\[ \xymatrix{  \mathfrak{X} \ar@{^{(}->}[r] \ar[d]_{\tilde{u}} & \widehat{\P}^N_\mathcal{V} \times_\mathcal{V} \widehat{\P}^M_\mathcal{V}  \ar[d]^v \\ \mathfrak{S} \ar@{^{(}->}[r] &  \widehat{\P}^N_\mathcal{V} }  \]
where $v$ is the first projection. Then $(S,S,\widehat{\A}^N_\mathcal{V},\widehat{\P}^N_\mathcal{V})$ is an l.p. frame, and setting $\mathfrak{Y} := v^{-1}(\widehat{\A}^N_\mathcal{V})$ gives a closed immersion $i':\mathfrak{X}\hookrightarrow \mathfrak{Y}$, an l.p. frame $(U,X,\mathfrak{Y},\widehat{\P}^N_\mathcal{V} \times_\mathcal{V} \widehat{\P}^M_\mathcal{V})$, and a morphism of l.p. frames
\[ v: (U,X,\mathfrak{Y},\widehat{\P}^N_\mathcal{V} \times_\mathcal{V} \widehat{\P}^M_\mathcal{V}) \rightarrow (S,S,\widehat{\A}^N_\mathcal{V},\widehat{\P}^N_\mathcal{V}) \]
extending $\mathbb{U}\rightarrow \mathbb{S}$. By definition, $D^b_{\mathrm{hol},F}(\mathbb{S}/K)$ is a full subcategory of $D^b_\mathrm{hol}(\mathcal{D}^\dagger_{\widehat{\A}^N_\mathcal{V},\Q})$, consisting of objects supported on $S$. Hence by \cite[Th\'eor\`eme 2.11]{Car09b} it is contained in the essential image of the fully faithful functor
\[ i_+: D^b_\mathrm{hol}(\mathcal{D}^\dagger_{\mathfrak{S},\Q}) \rightarrow D^b_\mathrm{hol}(\mathcal{D}^\dagger_{\widehat{\A}^N_\mathcal{V},\Q}),\]
in other words we can view it as a full subcategory of $D^b_\mathrm{hol}(\mathcal{D}^\dagger_{\mathfrak{S},\Q})$. An entirely similar argument applies for $D^b_{\mathrm{hol},F}(\mathbb{U}/K)$. For the claim concerning the pushforward functor, it suffices to verify that the diagram
\[ \xymatrix{ D^b_\mathrm{hol}(\mathcal{D}^\dagger_{\mathfrak{X},\Q}) \ar[d]_{\tilde{u}_+} \ar[r]^{i'_+} & D^b_\mathrm{hol}(\mathcal{D}^\dagger_{\mathfrak{Y},\Q}) \ar[d]^{\mathbf{R}\underline{\Gamma}_S^\dagger \circ v_+}\\ D^b_\mathrm{hol}(\mathcal{D}^\dagger_{\mathfrak{S},\Q}) \ar[r]^{i_+} & D^b_\mathrm{hol}(\mathcal{D}^\dagger_{\widehat{\A}^N_\mathcal{V},\Q}) }  \]
commutes up to natural isomorphism, which follows for example from \cite[Th\'eor\`eme 3.8]{Car09b}. The final claim simply follows from the construction of
\[\mathrm{sp}_+: \mathrm{Isoc}(S/K) \rightarrow D^b_\mathrm{coh}(\mathcal{D}^\dagger_{\widehat{\A}^N_\mathcal{V},\Q}) \]
as the composite
\[\mathrm{Isoc}(S/K) \overset{\mathrm{sp}_*}{\longrightarrow}  D^b_\mathrm{coh}(\mathcal{D}^\dagger_{\mathfrak{S},\Q}) \overset{i_+}{\longrightarrow}  D^b_\mathrm{coh}(\mathcal{D}^\dagger_{\widehat{\A}^N_\mathcal{V},\Q}) ,\]
where the first functor is Berthelot's equivalence \cite[Proposition 4.1.4]{Ber96a} between convergent isocrysals on $S/K$ and $\mathcal{O}_{\mathfrak{S},\Q}$-coherent $\mathcal{D}^\dagger_{\mathfrak{S},\Q}$-modules.
\end{proof}

\section{Irregularity of \texorpdfstring{$p$}{p}-adic differential equations} \label{sec: irreg}

The $p$-adic analogue of the Swan conductor is the \emph{irregularity} of a $p$-adic differential equation, as defined by Christol and Mebkhout. In this section we will recall the definition and basic properties of this irregularity. We will continue to allow $K$ to be any complete, normed field of mixed characteristic $(0,p)$, unless specifically stated otherwise. Let $\mathcal{R}^u_K$ denote the Robba ring over $K$, with co-ordinate $u$, say, and let $M$ be a projective $\nabla$-module over $\mathcal{R}^u_K$. Then for all $\rho<1$ sufficiently close to $1$, we can base change $M$ to obtain a (necessarily free) $\nabla$-module $M_\rho$ over the completion $K(u)_\rho$ of $K(u)$ for the $\rho$-Gauss norm. Define the radius of convergence
\[ R(M_\rho):=\min\left\{ \rho, \mathrm{lim\,inf}_{k\rightarrow\infty} \norm{G_k}_\rho^{-1/k}   \right\} ,\]
where $G_k$ is the matrix of the operator $\frac{1}{k!}\frac{d^k}{du^k}$ acting on $M_\rho$.

\begin{definition} We say that $M$ is overconvergent if $\lim_{\rho\rightarrow 1} R(M_\rho)=1$. 
\end{definition}

\begin{remark} The more standard terminology for such a $\nabla$-module is `solvable', however, we will also want to work with $\nabla$-modules over relative Robba rings $\mathcal{R}^u_A$ arising from overconvergent isocrystals. Thus we have chosen to use a more uniform terminology.
\end{remark}

The $\nabla$-module $M$ is said to have uniform break $b$ if for all $\rho$ sufficiently closed to $1$, and all sub-quotients $N$ of $M_\rho$, $R(N)=\rho^{b+1}$.

\begin{theorem}\cite[Corollaire 2.4-1]{CM01} \label{theo: break 1} For any projective, overconvergent $\nabla$-module over $\mathcal{R}^u_K$, there exists a unique decomposition
\[ M=\bigoplus_{b\geq0} M_b \]
of $\nabla$-modules, called the break decomposition, such that each $M_b$ has uniform break $b$.
\end{theorem}

\begin{remark} In \cite{CM01} the ground field $K$ is assumed to be spherically complete, in which case $M$ is free. It is explained how to extend this to the general case in \cite[Lemma 2.7.3]{Ked07a}.
\end{remark}

\begin{definition}\cite[D\'efinition 8.3-8]{CM00} The irregularity of $M$ is defined to be $\mathrm{Irr}(M):=\sum_b b\cdot \mathrm{rank}_{\mathcal{R}^u_K} M_b$.
\end{definition}

\begin{remark} \begin{enumerate} \item We will often want to consider cases when $K$ itself is equipped with a natural derivation $\partial_t$, for example when $K$ is the completion of a rational function field $K_0(t)$ for the Gauss norm induced by a norm on $K_0$. In this case Kedlaya \cite{Ked07a} has developed a more refined notion of irregularity, that takes this horizontal derivation $\partial_t$ into account. We will only consider the `na\"ive' irregularity coming from the vertical derivation $\partial_u$.  
\item If $K\rightarrow K'$ is an isometric extension of complete fields, then a projective $\nabla$-module $M$ over $\mathcal{R}^u_K$ is overconvergent if and only if $M\otimes \mathcal{R}_{K'}^u$ is, in which case they have the same irregularity \cite[Proposition 1.2-4]{Meb02}.
\end{enumerate}
\end{remark}

\begin{lemma} \label{lemma: irr push} Let $L/K$ be a finite extension, and $M$ an overconvergent $\nabla$-module over $\mathcal{R}_L^u$. Let $\mathrm{Res}^L_KM$ denote $M$ considered as an overconvergent $\nabla$-module over $\mathcal{R}_K^u$ via the map $\mathcal{R}_K^u\rightarrow \mathcal{R}_L^u$. Then
\[ \mathrm{Irr}(\mathrm{Res}_{L/K}M) = [L:K]\mathrm{Irr}(M).\]
\end{lemma}

\begin{proof}
First assume that $L/K$ is Galois. In this case, we have
\[ \left(\mathrm{Res}^L_K M\right) \otimes_K L \cong \bigoplus_{\sigma\in \mathrm{Gal}(L/K)} M \]
as $\nabla$-modules over $\mathcal{R}_L^u$, and so we can apply \cite[Proposition 1.2-4]{Meb02}. In the general case we take a Galois closure $F/L/K$ and deduce that
\[ \mathrm{Irr}(\mathrm{Res}^L_K \mathrm{Res}^F_L(M\otimes_L F)) = \mathrm{Irr}(\mathrm{Res}^F_K(M\otimes_L F))=[F:K]\mathrm{Irr}(M) \]
again using \cite[Proposition 1.2-4]{Meb02}. Finally, we use the fact that $\mathrm{Res}^F_L(M\otimes_L F) \cong M^{[F:L]}$ to conclude.
\end{proof}

\subsection{Irregularity in families}

We will be interested in studying how the irregularity varies in families, and so we will want to replace the field $K$ in the above discussion by a $K$-dagger algebra $A$ of MW-type. (It seems entirely likely that the results here will extend to more general $K$-dagger algebras, but we will only need this restricted case.) For such a $K$-dagger algebra $A$, suppose that we have a projective $\nabla$-module $M$ over the relative Robba ring $\mathcal{R}_A^u$. Let $L$ be the completion of the fraction field of $A$ for the supremum norm.

\begin{definition} We say that $M$ is overconvergent if the generic fibre $M_L:=M\otimes \mathcal{R}_L^u$ of $M$ is overconvergent.
\end{definition}

The generic fibre $M_L$ therefore admits a break decomposition
\[ M_L = \bigoplus_{b\geq 0} M_{L,b} \]
by Theorem \ref{theo: break 1}.

\begin{theorem}[\cite{Ked11a}, Theorem 1.3.2] There exists a dagger localisation $A\rightarrow B$ and a unique decomposition of $M_B:= M\otimes \mathcal{R}^u_B$ which restricts to the break decomposition over $\mathcal{R}^u_L$.
\end{theorem}

\begin{proof}
The proof of \cite[Theorem 1.3.2]{Ked11a} assumes that the derivation $\partial_u$ is `eventually dominant' relative to the derivations of $L/K$ which also act on $M\otimes \mathcal{R}_L$. However, this assumption is only used to \emph{interpret} the break decomposition obtained as a genuine break decomposition for the full collection of derivations, and is not used in showing that such a decomposition exists.
\end{proof}

It will be important to have conditions for extending this break decomposition over the whole of $A$. Let $\mathcal{M}(A)$ denote the Berkovich spectrum of $A$ (consisting of $p$-adically bounded multiplicative semi-norms on $A$); for any $v\in \mathcal{M}(A)$ we let $\mathscr{H}(v)$ denote the completed residue field at $v$. Thus base changing $M$ via $\mathcal{R}^u_A\rightarrow \mathcal{R}^u_{\mathscr{H}(v)}$ we obtain a projective $\nabla$-module $M_v$ over $\mathcal{R}^u_{\mathscr{H}(v)}$. Since the reduction $\overline{A}$ of $A$ is smooth and integral over $k$, it follows from \cite[Proposition 2.4.4]{Ber90} that the Berkovich space $\mathcal{M}(A)\cong \mathcal{M}(\widehat{A})$ has a unique point $\xi$ in its Shilov boundary. In this case, the completed residue field $\mathscr{H}(\xi)$ is equal to the completed fraction field $L$ of $A$.

\begin{proposition} \label{prop: lower} For any point $v\in \cur{M}(A)$ the $\nabla$-module $M_v$ is overconvergent, and we have $\mathrm{Irr}(M_v)\leq \mathrm{Irr}(M_L)$.
\end{proposition}

Before we can give the proof of this proposition, we need to introduce another function governing the variation of the irregularity along a $2$-dimensional Berkovich space. Let $M$ be a $\nabla$-module over the ring
\[ K\tate{\tau u^{-1},\rho^{-1}u,x}\]
of functions converging for $\rho\leq \norm{u} \leq \tau$ and $\norm{x}\leq 1$. Then for any $(-\log \alpha,-\log \beta)\in [-\log\tau,-\log \rho ] \times [0,\infty ]$ we obtain by base change a $\nabla$-module $M_{\alpha,\beta}$ over the field $K(u,x)_{\alpha,\beta}$ obtained by completing $K(u,x)$ with respect to the norm for which $\norm{u}=\alpha$ and $\norm{x}=\beta$. (When $\beta=0$ this should be interpreted as $K(u)_\alpha$.) For every irreducible constituent $N$ of $M_{\alpha,\beta}$ we can therefore consider the radius of convergence $R(N)$ with respect to the $u$-derivation exactly as before, that is 
\[ R(N):=\min\left\{ \rho, \mathrm{lim\,inf}_{k\rightarrow\infty} \norm{G_k}_\rho^{-1/k}   \right\} ,\]
where $G_k$ is the matrix of the operator $\frac{1}{k!}\frac{\partial^k}{\partial u^k}$ acting on $N$. We then define
\[ F_M(-\log\alpha,-\log\beta) = -\sum_N \mathrm{dim}_{K_{\alpha,\beta}}N \cdot\log R(N),\]
the sum being over all such irreducible constituents $N$. This gives a function
\[ F_M(-,-): [-\log\tau,-\log \rho ] \times [0,\infty ] \rightarrow \R_{\geq 0}. \]

\begin{proof}[Proof of Proposition \ref{prop: lower}] It is harmless to replace $A$ by completion $\widehat{A}$, so we may instead prove the corresponding claim for a smooth affinoid $K$-algebra $A$ with good reduction. We let $A_{\mathscr{H}(v)}$ denote the base change of $A$ to $\mathscr{H}(v)$ and consider the Cartesian diagram
\[ \xymatrix{\mathcal{M}(A_{\mathscr{H}(v)})\ar[r] \ar[d] & \mathcal{M}(A) \ar[d]  \\ \mathcal{M}(\mathscr{H}(v)) \ar[r] & \mathcal{M}(K). } \]
By construction, there exists a rigid point of $\mathcal{M}(A_{\mathscr{H}(v)})$ lying above $v\in \mathcal{M}(A)$, and the unique point in the Shilov boundary of $\mathcal{M}(A_{\mathscr{H}(v)})$ lies above $\xi$. By invariance under isometric extensions, we may replace $K$ by $\mathscr{H}(v)$ and thus assume that $v$ is a rigid point of $\cur{M}(A)$. Taking a completed localisation of $A$ around $v$, applying \cite[Theorem 1]{Ked05} and lifting, we can assume we have a finite \'etale map $K\tate{\bm{x}} \rightarrow A$, for $\bm{x}=(x_1,\ldots,x_d)$. By Lemma \ref{lemma: irr push} it suffices to prove the claim for the pushforward of $M$ along $K\tate{\bm{x}}\rightarrow A$, hence we may assume that $A=K\tate{\bm{x}}$. By translating, and possibly increasing $K$, we may assume that $v=0$. 

We let $v_i$ denote the image in $\mathcal{M}(K\tate{\bm{x}})$ of the unique point in the Shilov boundary of
\[ \mathcal{M}(K\tate{x_1,\ldots,x_d}/(x_1,\ldots,x_i))\] under the canonical closed immersion. Thus $v_0=\xi$ and $v_d=v$, and it therefore suffices to show that $M_{v_i}$ overconvergent $\Rightarrow $ $M_{v_{i+1}}$ overconvergent, and that $\mathrm{Irr}(M_{v_i})\geq \mathrm{Irr}(M_{v_{i+1}})$. But now looking at the commutative (although not in general Cartesian) diagram
\[ \xymatrix{ \mathcal{M}(K_{v_i}\tate{x_{i}}) \ar[r]\ar[d] & \mathcal{M}(K\tate{\bm{x}}) \ar[d] \\ \mathcal{M}(K_{v_i}) \ar[r] & \mathcal{M}(K) } \]
we can see that the zero point of $\mathcal{M}(K_{v_i}\tate{x_{i}})$ lies above $v_{i+1}$, and the unique point in the Shilov boundary of $\mathcal{M}(K_{v_i}\tate{x_{i}})$ lies above $v_{i}$. Again by invariance under isometric extensions we can therefore reduce to the case $d=1$, i.e. $A=K\tate{x}$.

Now let $\rho$ be close enough to $1$ such that $M$ comes from a $\nabla$-module defined over the ring
\[  \cap_{\rho\leq \tau<1}K\tate{\rho u^{-1},\tau^{-1}u,x}\] 
of functions converging for $\rho\leq \norm{u} <1$ and $\norm{x}\leq 1$, and let
\[ F_M(-,-):(0,-\log \rho ] \times [0,\infty ] \rightarrow \R  \]
be the function defined above. After possibly increasing $\rho$, we may assume by Theorem \ref{theo: break 1} quoted above that the function $F_M(r,0)$ is given by
\[ F_M(r,0) = (\mathrm{rank}_{\mathcal{R}_L^u}M_L + \mathrm{Irr}(M_L))r, \]
where $M_L$ is the base change to $\mathcal{R}_L^u$. Now fix some $r_0\in (0,-\log \rho]$ and consider the $\nabla$-module $M\otimes K(u)_{e^{-r_0}}\tate{x}$. The field $K(u)_{e^{-r_0}}$ is of rational type in the sense of \cite[Definition 1.4.1]{KX10}, hence we may apply \cite[Theorem 2.2.6]{KX10} to deduce that the function
\[ F_M(r_0,-) : [0,\infty] \rightarrow \R_{\geq 0} \]
is decreasing and continuous. Thus we find that
\[ \lim_{r\rightarrow 0} F_M(r,\infty) \leq \lim_{r\rightarrow 0} F_M(r,0) = 0 \]
from which we deduce that the base change $M_0$ of $M$ to $\mathcal{R}_K^u$ via $x\mapsto 0$ is also overconvergent. Thus after possibly increasing $\rho$ we may assume again by Theorem \ref{theo: break 1} that $F_M(r,\infty)$ is given by 
\[ F_M(r,\infty) = (\mathrm{rank}_{\mathcal{R}_K^u}M_0 + \mathrm{Irr}(M_0))r. \]
Now again using the fact that $F_M(r_0,-)$ is a decreasing function we can deduce that $\mathrm{Irr}(M_0)\leq \mathrm{Irr}(M_L)$ as required.
\end{proof}

We therefore obtain a function
\[ \mathrm{Irr}_M : \mathcal{M}(A) \rightarrow \Z_{\geq 0} \]
bounded above by $\mathrm{Irr}(M_L)$.

\begin{proposition} \label{prop: spread} Assume that $K$ is discretely valued. Let $M$ be an overconvergent, projective $\nabla$-module over $\mathcal{R}^u_A$, and assume that the function $\mathrm{Irr}_M$ is constant. Then the break decomposition extends uniquely across $A$, that is, there exists a unique decomposition
\[ M = \bigoplus_b M_b \]
of $\nabla$-modules over $\mathcal{R}_A^u$ which restricts to the break decomposition of $M\otimes_A \mathcal{R}_L^u$. Moreover, for any closed point $s:A\rightarrow K'$ the induced decomposition
\[ M_s = \bigoplus_b (M_b)_s \]
of $M_s := M\otimes_s \mathcal{R}_{K'}^u$ coincides with the break decomposition of $M_s$.
\end{proposition}

We start with a simple special case.

\begin{lemma} \label{lemma: br dec ext} Let $M$ be an overconvergent, projective $\nabla$-module over $\mathcal{R}^u_{K\weak{x}}$. Assume that the break decomposition extends uniquely across $K\weak{x,x^{-1}}$ and that the function $\mathrm{Irr}_M$ is constant on $\cur{M}(K\weak{x})$. Then the break decomposition extends uniquely across $K\weak{x}$.
\end{lemma}

\begin{proof}
Via the bijection between decompositions of a module and representations of the identity map as a sum of orthogonal idempotents,  uniqueness of any extension of the break decomposition follows from injectivity of the map
\[ \mathrm{End}_{\mathcal{R}_{K\weak{x}}^u}(M) \rightarrow \mathrm{End}_{\mathcal{R}_{K\weak{x,x^{-1}}}^u}(M \otimes\mathcal{R}_{K\weak{x,x^{-1}}}^u ) ;\]
existence boils down to whether or not the given orthogonal idempotents in $\mathrm{End}_{\mathcal{R}_{K\weak{x,x}}^u}(M \otimes\mathcal{R}_{K\weak{x,x^{-1}}}^u)$ actually lie in $\mathrm{End}_{\mathcal{R}_{K\weak{x}}^u}(M)$. 

To see that they do, take $\rho<1$ close enough to $1$ such that $M$ comes from a $\nabla$-module over
\[ \cap_{\rho\leq \tau<1}\cup_{\lambda>1} K\tate{\tau u^{-1},\rho^{-1}u,\lambda^{-1}x}. \]
By the proof of \cite[Lemma 1.3.4]{Ked11a} it suffices to show that for any such $\rho$ the induced decomposition of $ M \otimes K(u)_\rho \weak{x,x^{-1}}$ extends to a decomposition of $M\otimes K(u)_\rho\weak{x}$. 

Having fixed $\rho$, we may, for any $0\leq \eta\leq 1$, consider the $\nabla_u$-module $M_{\rho,\eta}$ over $K(u,x)_{\rho,\eta}$ as before, and its $n=\mathrm{rank}(M)$ (extrinsic) radii of convergence $r_1(\eta)\geq \ldots \geq r_n(\eta)$ (for the $u$-derivation). We may therefore define functions
\begin{align*} F_i:[0,\infty] &\rightarrow \R  \\
 F_i(-\log \eta) &= \sum_{j\leq i} -\log r_j(\eta),
 \end{align*}
in particular we find $F_n(r)=F_M(-\log\rho,r)$  where $F_M$ is the function we defined previously. The functions $F_i$ are decreasing by \cite[Theorem 2.2.6]{KX10}, and by the constancy of the irregularity of $M$ we know that $F_{n}(0)=F_{n}(\infty)$. This implies that in fact all of the $F_i$ are constant on $[0,\infty]$, which by \cite[Theorem 2.3.10]{KX10} implies that there exists a unique decomposition of $M$ over $K(u)_\rho\pow{x}_0$, the ring of convergent series on $\norm{x}<1$ which are bounded as $\norm{x} \rightarrow 1$, which restricts to the break decomposition on $M\otimes K(u,x)_{\rho,\eta}$ for each $\eta\in (0,1)$. Since we have a break decomposition over $K(u)_\rho\weak{x,x^{-1}}$ and over $K(u)_\rho\pow{x}_0$, we therefore have one over $K(u)_\rho\weak{x}=K(u)_\rho\weak{x,x^{-1}}\cap K(u)_\rho\pow{x}_0$.
\end{proof}

We can then reduce the general case to this as follows.

\begin{proof}[Proof of Proposition \ref{prop: spread}] Once we know that the break decomposition extends across $A$, the final claim that it induces the break decomposition at every closed point follows from Proposition \ref{prop: lower} above. Also, the uniqueness claim follows exactly as in the proof of Lemma \ref{lemma: br dec ext} above. 

To see that the break decomposition does indeed extend across $A$, we begin by showing that if $\{ A \rightarrow A_i \}_{i\in I}$ is a finite dagger open cover of $A$, and we set $A_{ij}=A_i\otimes_A^\dagger A_j$, then, for any finite projective $\mathcal{R}_A^u$-module $N$, the sequence
\[ 0 \rightarrow N \rightarrow \prod_i N \otimes_{\mathcal{R}_A^u} \mathcal{R}_{A_i}^u \rightarrow \prod_{i,j}  N \otimes_{\mathcal{R}_A^u} \mathcal{R}_{A_{ij}}^u  \] 
is exact. Indeed, since $N$ is a direct summand of a finite free $\mathcal{R}^u_A$-module, we reduce to the case where $N$ is finite free, and thus to the case $N=\mathcal{R}_A^u$. If we let $A\weak{\rho u^{-1},\eta^{-1}u}$ denote the ring of overconvergent series on the  relative annulus over $A$ of radius $[\rho,\eta]$, then we can give an alternative description of $\mathcal{R}_A^u$ as $\mathrm{colim}_{\eta<1} \mathrm{lim}_{\rho<1} A\weak{ \rho u^{-1},\eta^{-1}u}$. Thus using the dagger form of Tate's acyclicity theorem \cite[Proposition 2.6]{GK00} we can see that the sequence
\[ 0\rightarrow  A\weak{\rho u^{-1},\eta^{-1}u} \rightarrow \prod_iA_i\weak{\rho u^{-1},\eta^{-1}u} \rightarrow \prod_{ij} A_{ij}\weak{\rho u^{-1},\eta^{-1}u} \]
is exact for all $\rho,\eta$. Since $\mathrm{colim}$ is exact and $\mathrm{lim}$ is left exact, the claim follows. Applying this to $N=\mathrm{End}_{\mathcal{R}_A^u}(M)$ we can see that if the break decomposition extends across all $A_i$, then it extends across $A$. Thus the question is `dagger local' on $A$.

Now let $\mathcal{C}=\{A_i \}_{i\in I}$ denote the collection of all possible dagger localisations of $A$ such that the break decomposition extends across $A_i$, we wish to show that $A\in \mathcal{C}$. Suppose, then, for contraction, that $A\not\in \mathcal{C}$. Then after passing to the reductions modulo the maximal ideal of $\mathcal{V}$ the open immersion $\bigcup_{i\in I} \spec{\overline{A}_i} \subsetneq \spec{\overline{A}}$ is strict. Thus after possibly making a finite extension of $K$ (which is harmless) we may assume that there exists a smooth $k$-rational point $z$ on the reduced complement $\left(\spec{\overline{A}}\setminus \bigcup_{i\in I} \spec{\overline{A}_i}\right)_{\mathrm{red}}$. To contradict the maximality of $\mathcal{C}$, then, it suffices to produce a dagger localisation $A\rightarrow A'$ such that $z\in\spec{\overline{A}'}$ and such that the break decomposition extends across $A'$.

As we have already seen, the question is dagger local on $A$, hence we may localise around $z$, and use \cite[Theorem 1]{Ked05} together with Lemma \ref{lemma: finite etale lift} to obtain a finite \'etale map $K\weak{x_1,\ldots,x_d} \rightarrow A$ such that the induced map $\spec{\overline{A}}\rightarrow \A^d_k$ on reductions sends $z$ to the origin, and the break decomposition extends across $A\weak{x_d^{-1}}$. Restricting along this finite \'etale map we may assume that $A=K\weak{x_1,\ldots,x_d}$ and that the break decomposition extends across $K\weak{x_1,\ldots,x_d,x_d^{-1}}$. Now let $F$ be the completed fraction field of $K\weak{x_1,\ldots,x_{d-1}}$, so we have
\[ \mathcal{R}^u_{K\weak{x_1,\ldots,x_d}} = \mathcal{R}^u_{K\weak{x_1,\ldots,x_d,x_d^{-1}}} \cap \mathcal{R}^u_{F\weak{x_d}} \subset \mathcal{R}^u_L.   \]
Hence applying \cite[Lemma 1.2.7]{Ked11a} to $\mathrm{End}_{\mathcal{R}_{K\weak{x_1,\ldots,x_d}}^u}(M)$ we can see that it suffices to prove that the break decomposition extends across $\mathcal{R}^u_{F\weak{x_d}}$. Finally replacing $K$ by $F$ we can appeal to Lemma \ref{lemma: br dec ext} above to conclude.
\end{proof}

\section{Relative curves and generic pushforwards} \label{sec: src and gp}

We shall assume for the rest of the article that the ground field $K$ is discretely valued. The residue field $k$ will continue (for now) to be arbitrary of characteristic $p$.

\subsection{The basic geometric setup} \label{sec: bgs} Here we will describe the basic geometric setup for our first $p$-adic acyclicity theorems.

\begin{definition} \begin{enumerate}
\item An affine curve is a smooth, affine morphism $f:U\rightarrow S$ of $k$-varieties, of relative dimension $1$.
\item A smooth compactification $\bar f:C\rightarrow S$ of an affine curve $f:U\rightarrow S$ is called \emph{good} if the complement $C\setminus U$ is \'etale over $S$. 
\end{enumerate}
\end{definition}

\begin{remark} It might be more usual to require our curves to have geometrically connected fibres. However, it will be important for us \emph{not} to assume this.
\end{remark}

Our most general acyclicity result, Theorem \ref{theo: main general} below, will apply to affine curves admitting good compactifications (at least when $k$ is perfect). Our first goal, however, will be to prove a similar result in a much more restrictive setting. 

\begin{setup} \label{setup: main}
We consider an affine curve $f:U\rightarrow S$ over a smooth, affine base, which admits a good compactification $\bar{f}:C\rightarrow S$, and a lift
\[ \mathfrak{f}:(C,\overline{C},\mathfrak{C}) \rightarrow (S,\overline{S},\mathfrak{S})\]
of $\bar{f}$ to a smooth, proper, Cartesian morphism of frames, such that:
\begin{itemize}
\item both $(S,\overline{S},\mathfrak{S})$ and $(U,\overline{C},\mathfrak{C})$ are of MW-type;
\item the complement $C\setminus U$ is a disjoint union of sections $\sigma_j$ of $\bar{f}$, and there exists an open neighbourhood of each on which it is defined by the vanishing of a single function $u_j\in \mathcal{O}_C$. 
\end{itemize}
\end{setup}

That proving cohomological results in this more restrictive situation will suffice follows from the next proposition.

\begin{proposition} \label{prop: good to ssl} Let $f:U\rightarrow S$ be an affine curve over a smooth base $S$, admitting a good compactification $\bar f:C\rightarrow S$. Then, after possibly  passing to an \'etale cover of $S$, there exists a morphism of frames
\[ \bar{f}: (C,\overline{C},\mathfrak{C})\rightarrow (S,\overline{S},\mathfrak{S}) \]
enclosing $\bar{f}$, such that the conditions of Setup \ref{setup: main} apply.
\end{proposition}

\begin{proof}
First of all, we may assume that $S$ is affine and connected. Let $S'$ be a common Galois closure of all the connected components of $C\setminus U$, which by assumption are finite \'etale over $S$. Then after base changing to $S'$ the complement $C\setminus U$ is a disjoint union of sections, which by smoothness of $C\rightarrow S$ must all be regular closed immersions. Hence after passing to a Zariski cover of $S'$ they are each defined in some neighbourhood by the vanishing of a single function on $C$.

Next, we let $g$ denote the genus of the family $\bar{f}:C\rightarrow S$, and $S\rightarrow \mathcal{M}_g$ the corresponding morphism to the moduli stack of curves. Since $\mathcal{M}_g$ is a smooth Artin stack (for any $g$), we can find a smooth surjective morphism $M\rightarrow \mathcal{M}_g$ from a smooth affine scheme (over $\Z$). After passing to an \'etale cover over $S$, then, we can lift the given map $S\rightarrow \mathcal{M}_g$ to a map $S\rightarrow M$. 

Now, by \cite[Th\'eor\`eme 6]{Elk73} we can choose a smooth affine scheme $\mathcal{S}$ over $\mathcal{V}$ with special fibre $S$. Let $\mathcal{S}^h$ denote the Henselisation of $\mathcal{S}$ along $V(\varpi)$, this is therefore a $\varpi$-adically Henselian affine scheme, whose reduction mod $\varpi$ is agan $S$. Hence by \cite[Th\'eor\`eme 2]{Ray72} the map $S\rightarrow M$ lifts to a map $\mathcal{S}^h\rightarrow M$, so composing with $M\rightarrow \mathcal{M}_g$ and pulling back the universal family we obtain a lift $\mathcal{C}^h\rightarrow \mathcal{S}^h$ of $C\rightarrow S$ to a smooth and proper curve over $\mathcal{S}^h$. 

Hence there exists an \'etale morphism $\mathcal{S}'\rightarrow \mathcal{S}$ of affine $\cur{V}$-schemes, inducing an isomorphism on special fibres, and a smooth and proper curve $\mathcal{C}'\rightarrow \mathcal{S}'$ lifting $C\rightarrow S$. Finally, we choose a compactification $\overline{\mathcal{S}}'$ of $\mathcal{S}'$ over $\cur{V}$, and a compactification $\overline{\mathcal{C}}'\rightarrow \overline{\mathcal{S}}'$ of $\mathcal{C}'\rightarrow \mathcal{S}'$, and set $\mathfrak{S}=\widehat{\overline{\mathcal{S}}'}$ and $\mathfrak{C}=\widehat{\overline{\mathcal{C}}'}$.
\end{proof}

\subsection{Generic pushforwards \`a la Kedlaya} \label{sec: gnalk}

Having set things up relatively geometrically, we will for a while revert to a more algebraic viewpoint on relative rigid cohomology, at least until \S\ref{sec: strong}. In the situation of Setup \ref{setup: main}, we set
\begin{align*} A&:=\Gamma(\mathfrak{S}_K,j_S^\dagger\mathcal{O}_{\mathfrak{S}_K}) \\
B&:=\Gamma(\mathfrak{C}_K,j_U^\dagger\mathcal{O}_{\mathfrak{C}_K}).
\end{align*}
These are therefore MW-type $K$-dagger algebras, pullback induces a homomorphism $A\rightarrow B$, and the module of continuous differentials $\Omega^1_{B/A}$ is a finite projective $B$-module of rank $1$.

Ordinary higher direct images will be defined in terms of the morphism of $K$-dagger algebras $A\rightarrow B$: for any overconvergent isocrystal $E$ on $U/K$, we can realise $E$ on the frame $(U,\overline{C},\mathfrak{C})$ and take global sections to obtain an overconvergent $\nabla$-module $M$ over $B$, and thus define the cohomology groups
\[ \mathbf{R}^0f_*M := \ker\left( M \overset{\nabla}{\rightarrow} M\otimes_B \Omega^1_{B/A} \right) \;\;\text{ and }\;\; \mathbf{R}^1f_*M := \mathrm{coker}\left( M \overset{\nabla}{\rightarrow} M\otimes_B \Omega^1_{B/A} \right).\]
We will also need to make use of other higher direct images defined using relative Robba rings, as in \cite{Ked06a}. In the situation of Setup \ref{setup: main}, let $\bar{\sigma}_j$ denote the closure of the image of $\sigma_j$ inside $\overline{C}$, and set
\begin{align*}
\mathcal{R}_{A,\sigma_j}^+ &= \Gamma(]\bar{\sigma}_j[_{\mathfrak{C}_K},j^\dagger_{\sigma_j}\mathcal{O}_{\mathfrak{C}_K}) \\
\mathcal{R}_{A,\sigma_j} &= \mathrm{colim}_V \Gamma(]\bar{\sigma}_j[_{\mathfrak{C}_K}\cap V ,j^{\dagger}_{\sigma_j}\mathcal{O}_{\mathfrak{C}_K}),
\end{align*}
the colimit in the second definition being over strict neighbourhoods $V$ of $]\overline{C}\setminus \overline{\sigma}_j[_\mathfrak{C}$ inside $\mathfrak{C}_K$. Since each $\sigma_j$ has a neighbourhood on which it is locally cut out by a single function $u_j\in\mathcal{O}_C$, by lifting these $u_j$ to some dagger localisation of $B$ and using the strong fibration theorem, we can identify
\begin{align*}  \mathcal{R}_{A,\sigma_j}^+  &\cong \mathcal{R}_{A}^{u_j+} \\
\mathcal{R}_{A,\sigma_j}  &\cong \mathcal{R}_{A}^{u_j}
\end{align*}
with copies of the relative Robba ring over $A$. For all $j$ there is a natural embedding
\[ B\rightarrow \mathcal{R}_{A,\sigma_j} \]
of $A$-algebras, and we define $\mathcal{Q}_A^{\{\sigma_j\}}$ to be the quotient
\[ 0 \rightarrow B \rightarrow \bigoplus_j \mathcal{R}_{A,\sigma_j} \rightarrow \mathcal{Q}_A^{\{\sigma_j\}} \rightarrow 0.\]
We can therefore define further higher direct images
\begin{align*}   
\mathbf{R}^0_\mathrm{loc}f_*M &:= \bigoplus_j  \ker\left( M \otimes_B \mathcal{R}_{A,\sigma_j} \overset{\nabla}{\rightarrow} M\otimes_B \mathcal{R}_{A,\sigma_j}\otimes_B \Omega^1_{B/A} \right)  \\ 
\mathbf{R}^1f_!M &:=   \ker \left( M \otimes_B \mathcal{Q}_A^{\{\sigma_j\}} \overset{\nabla}{\rightarrow} M \otimes_B \mathcal{Q}_A^{\{\sigma_j\}}\otimes_B \Omega^1_{B/A} \right)
 \\
 \mathbf{R}_\mathrm{loc}^1f_*M &:=\bigoplus_j \mathrm{coker} \left( M \otimes_B \mathcal{R}_{A,\sigma_j} \overset{\nabla}{\rightarrow} M\otimes_B \mathcal{R}_{A,\sigma_j}\otimes_B \Omega^1_{B/A} \right)  \\ 
 \mathbf{R}^2f_!M &:=  \mathrm{coker} \left( M \otimes_B \mathcal{Q}_A^{\{\sigma_j\}} \overset{\nabla}{\rightarrow} M \otimes_B \mathcal{Q}_A^{\{\sigma_j\}}\otimes_B \Omega^1_{B/A} \right),
\end{align*}
these sit in an exact sequence
\[ 0\rightarrow\mathbf{R}^0f_*M\rightarrow\mathbf{R}^0_\mathrm{loc}f_*M \rightarrow \mathbf{R}^1f_!M \rightarrow \mathbf{R}^1f_*M  \rightarrow\mathbf{R}^1_\mathrm{loc}f_*M \rightarrow \mathbf{R}^2f_!M \rightarrow 0.\]
When $A=K$ (or a finite extension thereof) we will usually write
\[ H^0(M),\;H^0_\mathrm{loc}(M),\;H^1_c(M),\;H^1(M),\;H^1_\mathrm{loc}(M),\;H^2_c(M)\]
instead. When $A\rightarrow A'$ is a morphism of MW-type $K$-dagger algebras, we will write either $B_{A'}$ or simply $B'$ for $B\otimes^\dagger_A A'$ and either $M_{A'}$ or simply $M'$ for $M\otimes_{B} B'$, thus $M'$ is an overconvergent $\nabla$-module over $B'$. 

\begin{theorem}\cite[Theorem 7.3.3, Remark 7.2.2, Proposition 8.6.1]{Ked06a} \label{theo: generic pushforwards} Assume Setup \ref{setup: main}, and let $M$ be an $F$-able $\nabla$-module on $B$.
\begin{enumerate} 
\item \label{parti} There exists a dagger localisation $A\rightarrow A'$ such that $\mathbf{R}^if_*M',\;\mathbf{R}^i_\mathrm{loc}f_*M',\; \mathbf{R}^if_!M'$ are finitely generated over $A'$, formation of which commutes with flat base change $A'\rightarrow A''$ of $MW$-type dagger algebras.
\item \label{partii} For any $A'$ such that the conclusions of (\ref{parti}) hold for $M'$ and ${M'}^\vee$, there are canonical perfect pairings
\begin{align*}
 \mathbf{R}^if_*M' \otimes_{A'} \mathbf{R}^{2-i}f_!{M'}^\vee &\rightarrow A'(-1) \\
  \mathbf{R}^i_\mathrm{loc}f_*M' \otimes_{A'} \mathbf{R}^{1-i}_\mathrm{loc}f_*{M'}^\vee &\rightarrow A'(-1)
\end{align*}
of $\nabla$-modules over $A'$.
\end{enumerate}
\end{theorem}

\begin{remark} \label{rem: Fstructure}
\begin{enumerate} \item \label{Fstruct} The base change claim implies that any Frobenius structure on $M$ induces one on all of the higher direct images $\mathbf{R}^if_*M',\;\mathbf{R}^i_\mathrm{loc}f_*M',\; \mathbf{R}^if_!M'$, in a way compatible with the Poincar\'e pairings in (\ref{partii}).
\item \label{bc to L1} It was also shown in \cite{Ked06a} that formation of these higher direct images commutes with base change to the completed fraction field $L$ of $A'$.
\end{enumerate}
\end{remark}

\begin{proof}
In \cite{Ked06a} the case when $U=\A^1_S$, $B=A\weak{x}$ and $M$ admits a Frobenius structure was treated, we will explain here how to reduce to this case. First of all, passing to the irreducible constituents of $M$ we may assume that $M$ itself admits a Frobenius structure. Applying \cite[Theorem 1]{Ked05} at the generic point of $S$ and spreading out we can find an open immersion $S' \rightarrow S$ and a finite \'etale map $U_{S'} \rightarrow \A^1_{S'}$ of $S'$-schemes. After possibly further localising and lifting to characteristic $0$ via Lemma \ref{lemma: finite etale lift} we can therefore find a dagger localisation $A\rightarrow A'$ such that there exists a finite \'etale morphism $A'\weak{x}\rightarrow B\otimes_A^\dagger A'$ of $A'$-algebras. Taking the pushforward along this finite \'etale map doesn't change any of the higher direct images, so we can replace $A$ by $A'$ and $B\otimes^\dagger_A A'$ by $A'\weak{x}$, and thus reduce to considering the case where $B=A\weak{x}$. 
\end{proof}

Roughly speaking, our goal will be to use the irregularity of a $\nabla$-module to give conditions under which we can take $A=A'$ in the above theorem. For each $\sigma_j$ we have the base change of $M$ along
\[ B \rightarrow \mathcal{R}_{\sigma_j} \cong \mathcal{R}_A^{u_j}, \]
this is an overconvergent $\nabla$-module $M_{\mathrm{loc},j}$ over $\mathcal{R}_A^{u_j}$, with an associated irregularity function
\[ \mathrm{Irr}_{M_{\mathrm{loc},j}} : \mathcal{M}(A) \rightarrow \Z_{\geq 0}. \]
We define the total irregularity of $M$ to be the function
\[\mathrm{Irr}^\mathrm{tot}_M := \sum_j\mathrm{Irr}_{M_{\mathrm{loc}_j}} : \mathcal{M}(A) \rightarrow \Z_{\geq 0}. \]
We can now state our first partial $p$-adic analogue of \cite[Corollaire 2.1.2]{Lau81} as follows.

\begin{theorem} \label{theo: main dagger} Assume Setup \ref{setup: main}, and let $M$ be an $F$-able $\nabla$-module on $B$. Then the following are equivalent:
\begin{enumerate}
\item \label{tmd part i} the total irregularity $\mathrm{Irr}^\mathrm{tot}_M:\mathcal{M}(A)\rightarrow \Z_{\geq 0} $ is constant;
\item \label{tmd part ii} the higher direct images $\mathbf{R}^0f_*M$ and $\mathbf{R}^1f_*M$ are finitely generated over $A$, and their formation commutes with arbitrary base change $A\rightarrow A'$ of MW-type $K$-dagger algebras.
\end{enumerate}
\end{theorem}

\begin{remark} \label{rem: bc to L2} As in Remark \ref{rem: Fstructure}, it follows from the base change claim that any Frobenius structure on $M$ induces one on $\mathbf{R}^if_*M$. Formation of $\mathbf{R}^if_*M$ also commutes with base change to the completed fraction field $L$ of $A$. 
\end{remark}

Note that the implication (\ref{tmd part ii})$\Rightarrow$(\ref{tmd part i}) follows from the Grothendieck--Ogg--Shafarevich formula \cite[Corollaire 5.0-12]{CM01}, the proof that (\ref{tmd part i})$\Rightarrow$(\ref{tmd part ii}) will occupy us until the end of \S\ref{sec: universal}. To start with, we will record a consequence of Theorem \ref{theo: generic pushforwards} that is not explicitly spelled out in \cite{Ked06a}, but can nonetheless be easily deduced from results there.

\begin{lemma} \label{lemma: base change 1}
In the situation of Theorem \ref{theo: generic pushforwards}, assume that the conclusions of the theorem hold for $M$ and $M^\vee$ without further localisation of $A$, and that $\mathbf{R}^0f_*M=\mathbf{R}^0f_*M^\vee=0$. Then the formation of the cohomology groups $\mathbf{R}^if_*M,\;\mathbf{R}^i_\mathrm{loc}f_*M\; \mathbf{R}^if_!M$ commutes with arbitrary base change $A\rightarrow A'$ of MW-type dagger algebras. 
\end{lemma}

\begin{proof} By choosing a set of topological generators for $A'$ over $A$ we can treat separately the cases when $A\rightarrow A'$ is surjective and when $A'=A\weak{x_1,\ldots,x_n}$. The latter case is covered by Theorem \ref{theo: generic pushforwards}, we will therefore consider the former. By Poincar\'e duality we know that $\mathbf{R}^2f_!M=(\mathbf{R}^0f_!M^\vee)^\vee=0$, and that $\mathbf{R}^2f_!M^\vee=(\mathbf{R}^0f_!M)^\vee=0$. Since the base change map
\[ ( \mathbf{R}^1_\mathrm{loc}f_*M ) \otimes_A A' \rightarrow \mathbf{R}^1_\mathrm{loc}f_*M' \]
is trivially surjective, we deduce that the latter is finitely generated over $A'$, which is enough to show that the conclusions of Theorem \ref{theo: generic pushforwards} hold for $M'$ without further localisation (see for example the proof of \cite[Theorem 7.3.3]{Ked06a}). By arguing similarly for $M'^\vee$, we deduce that Poincar\'e duality also holds for $M'$. The map 
\[ ( \mathbf{R}^2f_!M ) \otimes_A A' \rightarrow \mathbf{R}^2f_!M' \]
is also trivially surjective, thus we deduce that $\mathbf{R}^2f_!M'=0$; replacing $M$ by $M^\vee$ and applying Poincar\'e duality we can also see that $\mathbf{R}^0f_*M'=0$. Hence base change holds for $\mathbf{R}^0f_*M$ and $\mathbf{R}^2f_!M$. We now consider the diagram
\[ \xymatrix{  0 \ar[r] & (\mathbf{R}^0_\mathrm{loc}f_*M) \otimes_A A' \ar[r]\ar@{^(->}[d] & (\mathbf{R}^1f_!M) \otimes_A A' \ar[r]\ar@{^(->}[d] & (\mathbf{R}^1f_*M) \otimes_A A' \ar[r]\ar@{->>}[d] & (\mathbf{R}^1_\mathrm{loc}f_*M) \otimes_A A' \ar@{->>}[d] \ar[r] & 0 \\ 
  0 \ar[r] &  \mathbf{R}^0_\mathrm{loc}f_*M' \ar[r] & \mathbf{R}^1f_!M' \ar[r] & \mathbf{R}^1f_*M' \ar[r] & \mathbf{R}^1_\mathrm{loc}f_*M' \ar[r] & 0}\]
where the rows are the canonical exact sequences, and the vertical arrows come from base change. Immediately from the definitions we find that the right two vertical arrows are surjective; by replacing $M$ by $M^\vee$ and using Poincar\'e duality we can see that the left hand vertical arrows are injective. Moreover, we know from \cite[Corollaire 1.3-2]{Meb02} and base change to the completed fraction field of $A$ that 
\begin{align*}  \mathrm{rank}_A \mathbf{R}^0_\mathrm{loc}f_*M &= \mathrm{rank}_A \mathbf{R}^1_\mathrm{loc}f_*M \\
 \mathrm{rank}_{A'} \mathbf{R}^0_\mathrm{loc}f_*M' &= \mathrm{rank}_{A'} \mathbf{R}^1_\mathrm{loc}f_*M',
 \end{align*}
and hence base change has to hold for $\mathbf{R}^i_\mathrm{loc}f_*M$. Since
\begin{align*} 
\mathrm{rank}_A \mathbf{R}^1f_!M &\leq \mathrm{rank}_{A'} \mathbf{R}^1f_!M' \\
\mathrm{rank}_A \mathbf{R}^1f_*M &\geq \mathrm{rank}_{A'} \mathbf{R}^1f_*M' \\
\mathrm{rank}_A \mathbf{R}^1f_!M  - \mathrm{rank}_A \mathbf{R}^1f_*M &=  \mathrm{rank}_{A'} \mathbf{R}^1f_!M' -\mathrm{rank}_{A'} \mathbf{R}^1f_*M'
\end{align*}
we can also deduce that $\mathrm{rank}_A \mathbf{R}^1f_!M = \mathrm{rank}_{A'} \mathbf{R}^1f_!M'$ and $\mathrm{rank}_A \mathbf{R}^1f_*M = \mathrm{rank}_{A'} \mathbf{R}^1f_*M'$. This gives base change for $\mathbf{R}^1f_*M$ and $\mathbf{R}^1f_!M$, and completes the proof.
\end{proof}

We can use this to give the following minor strengthening of Theorem \ref{theo: generic pushforwards}.

\begin{corollary} \label{cor: base change 2} In the situation of Theorem \ref{theo: generic pushforwards} there exists a dagger localisation $A\rightarrow A'$ such that the higher direct images $\mathbf{R}^if_*M',\;\mathbf{R}^i_\mathrm{loc}f_*M',\; \mathbf{R}^if_!M'$ are finitely generated over $A'$, whose formation commutes with arbitrary base change $A'\rightarrow A''$ of $MW$-type dagger algebras.
\end{corollary} 

\begin{proof}
As explained above, we can reduce to the case where $B=A\weak{x}$, and $\mathbf{R}^if_*M,\;\mathbf{R}^i_\mathrm{loc}f_*M,\; \mathbf{R}^if_!M$ are finitely generated over $A$ with formation commuting with flat base change, as well as to the completed fraction field $L$ of $A$. In particular, we can verify that the natural map 
\[ \mathbf{R}^0f_*M \otimes_A A\weak{x} \rightarrow M \]
is injective, since it is so over $L$. After replacing $A$ by a further localisation, we can also assume that the conclusions of Theorem \ref{theo: generic pushforwards} hold for the quotient $N$ of $M$ by $\mathbf{R}^0f_*M \otimes_A A\weak{x}$. The base change claim holds for $\mathbf{R}^0f_*M \otimes_A A\weak{x}$ by using the projection formula, hence by the five lemma it suffices to prove it for $N$. But again by comparing with the situation over $L$ we can see that $\mathbf{R}^0f_*N=0$; in other words we may assume that $\mathbf{R}^0f_*M=0$. We now consider the dual $M^\vee$: after possibly further localising $A$ we can assume that the conclusions of Theorem \ref{theo: generic pushforwards} also hold for $M^\vee$, and so we obtain an exact sequence
\[ 0 \rightarrow N \rightarrow M \rightarrow \left( \mathbf{R}^0f_*M^\vee \right)^\vee \otimes_A A\weak{x} \rightarrow 0 \]
of $F$-able $\nabla$-modules. Again, possibly localising $A$, and using the five lemma to replace $M$ by $N$, we can assume that $\mathbf{R}^0f_*M=\mathbf{R}^0f_*M^\vee=0$. Thus we may apply Lemma \ref{lemma: base change 1}.
\end{proof}

\section{Unipotence and base change for \texorpdfstring{$\nabla$}{n}-modules over relative Robba rings} \label{sec: unip bc}

We will begin the proof of Theorem \ref{theo: main dagger} with a study of the behaviour of $\mathbf{R}^0_\mathrm{loc}f_*M$ in the case when the generic fibre of $M$ has unipotent monodromy around all missing points. In this section we will not need the geometric setup of \S\ref{sec: bgs}, and will instead simply work with dagger algebras and $\nabla$-modules. Since we are only interested in $ \mathbf{R}^0_\mathrm{loc}f_*$, we will let $A$ be a $K$-dagger algebra of MW-type, $L$ its completed fraction field, and $M$ an overconvergent $\nabla$-module over $\mathcal{R}_A^u$.

As part of the proof of \cite[Theorem 7.3.3]{Ked06a} Kedlaya shows that if $M$ is free, and the generic fibre $M\otimes \mathcal{R}_L^u$ is unipotent (as a $\nabla$-module relative to $L$), then there exists a dagger localisation $A\rightarrow A'$ such that $M\otimes \mathcal{R}^u_{A'}$ admits a strongly unipotent basis relative to $A'$. That is, it admits a basis $\left\{\bm{e}_i\right\}$ such that
\[ u\nabla_u(\bm{e}_i) \in \left( A\bm{e}_1  +\ldots+A\bm{e}_{i-1}\right)\otimes du ,\] 
where $\nabla_u$ is the `$u$-component' of the connection on $M$. It will be important for us to work with $\nabla$-modules that are not known \emph{a priori} to be free, and to still have a version of this result. Luckily, we will only need it in the case where $A=K\weak{\bm{x}}=K\weak{x_1,\ldots,x_d}$ is a free $K$-dagger algebra. We will keep the assumption that $K$ is discretely valued.

\begin{theorem} \label{theo: unipotent free} Let $M$ be a projective $\nabla$-module over $\mathcal{R}^u_{K\weak{\bm{x}}}$, whose generic fibre $M\otimes \mathcal{R}_L^u$ is unipotent relative to $L$. Then $M$ is free, and admits a strongly unipotent basis relative to $K\weak{\bm{x}}$.
\end{theorem}

The proof of this result will occupy the rest of \S\ref{sec: unip bc}. As in \cite[\S5]{Ked06a}, we will find it easier to introduce an auxiliary ring in place of $\mathcal{R}^u_{K\weak{\bm{x}}}$.

\begin{definition} Fix some $\eta\in \sqrt{\norm{K^*}}$ and define $R^u_{L,\eta}$ to be the ring of series $\sum_ia_iu^i$ with $a_i\in L$, such that there exist $\eta_{-}<\eta<\eta_{+}$ such that 
\[ \Norm{a_i}\eta_{\pm}^i  \rightarrow 0\]
as $i\rightarrow \pm\infty$. Similarly, define $R^u_{K\weak{\bm{x}},\eta}$ to be the ring of series $\sum_ia_iu^i$ with $a_i\in K\weak{\bm{x}}$ such that there exist $\eta_{-}<\eta<\eta_{+}$ and $\lambda >1$ such that $a_i\in K\tate{\lambda^{-1}\bm{x}}$ for all $i$ and
\[ \Norm{a_i}_\lambda \eta_{\pm}^i \rightarrow 0 \]
as $i\rightarrow \pm\infty$.
\end{definition}

As usual, a $\nabla$-module over $R^u_{K\weak{\bm{x}},\eta}$ will mean a $\nabla$-module relative to $K$, whose underlying $R^u_{K\weak{\bm{x}},\eta}$-module is finitely presented. Any extra adjectives such as projective, stably free, free, \&c. are understood to apply to the underlying $R^u_{K\weak{\bm{x}},\eta}$-module.

\begin{lemma} \begin{enumerate}
\item The map $K\weak{\bm{x}}\rightarrow L$ is flat.
\item The natural map $R^u_{K\weak{\bm{x}},\eta} \otimes_{K\weak{\bm{x}}} L \rightarrow R^u_{L,\eta}$ is injective.
\end{enumerate}
\end{lemma}

\begin{proof}
The first claim is clear, since $K\weak{\bm{x}}\rightarrow L$ is a monomorphism into a field. For the second, we choose a bijection $\Z\rightarrow \N$ and consider $R^u_{K\weak{\bm{x}},\eta}$ and $R^u_{L,\eta}$ as subspaces of the infinite products $\prod_{\N}K\weak{\bm{x}}$ and $\prod_{\N}L$ respectively; it suffices to show that the natural map
\[ \left( \prod_{\N}K\weak{\bm{x}} \right) \otimes_{K\weak{\bm{x}}} L \rightarrow \prod_{\N} L\]
is injective. If we let we let $L_0$ denote the (uncompleted) fraction field of $K\weak{\bm{x}}$, and factor $K\weak{\bm{x}}\rightarrow L$ through $L_0$, then it suffices to show the following:
\begin{enumerate}
\item if $R$ is an integral domain, with fraction field $F$, then the natural map $\left(\prod_{\N} R \right) \otimes_R F \rightarrow \prod_{\N} F$ is injective;
\item if $F\rightarrow F'$ is an arbitrary (non-zero) morphism of fields, then the natural map  $\left( \prod_{\N} F\right) \otimes_{F} F' \rightarrow \prod_{\N} F' $ is injective. 
\end{enumerate}
For the first, let us take some element $\sum_{i=1}^n (\lambda_{ij})_{j=1}^\infty \otimes f_i$ in the kernel of $\left(\prod_{\N} R \right) \otimes_R F \rightarrow \prod_{\N} F$, with $\lambda_{ij}\in R$ and $f_i\in F$. Expressing all the $f_i$ as fractions with a single common denominator, we may assume that  $f_i=\frac{1}{f}$ for some $f\in R\setminus \{0\}$, and all $i$. Hence we can in fact write our element simply as $(\lambda'_j)_{j=1}^\infty \otimes \frac{1}{f}$ for suitable $\lambda_j'$. Thus we have $\frac{\lambda'_{j}}{f}=0$ in $F$ for all $j$, whence $\lambda'_{j}=0$ in $R$ for all $j$. 

For the second, again suppose that we have an element
\[ \sum_{i=1}^n (\lambda_{ij})_{j=1}^\infty\otimes f_i \in \left( \prod_{\N} F \right) \otimes_F F'\]
which maps to zero in $\prod_{\N}F'$, i.e. such that $\sum_i\lambda_{ij}f_i=0$ for all $j$.

Let $V=F^n$ be the standard $n$-dimensional vector space equipped with the standard bilinear form. Let $\bm{\lambda}_j = (\lambda_{1j},\ldots,\lambda_{nj})\in V$ and $\bm{f} = (f_1,\ldots,f_n)\in V\otimes_F F'$. Let
\[ W_j = \bigcap_{j'=1}^j \left\{\left. \bm{v}\in V \right\vert   \bm{\lambda}_{j'} \cdot \bm{v} =0 \right\} \]
be the intersection of the annihilators of all the $\bm{\lambda}_{j'}$ for $1\leq j' \leq j$. Thus $W_j$ is a descending sequence of subspaces of $V$, which  must therefore eventually stabilise. Hence we have
\[ \left(\cap_j W_j \right) \otimes_F F' = \bigcap_j \left( W_j \otimes_F F' \right). \]
Pick a basis $\bm{e}_1,\ldots,\bm{e}_k$ for $\cap_j W_j$, and write these as $\bm{e}_l = (e_{l1},\ldots,e_{ln})$ with $e_{lm}\in F$. The fact that $\bm{e}_l\in \cap_j W_j$ means that
\[ \sum_i \lambda_{ij}e_{li} = 0 \]
for all $j,l$. Since $\bm{f}\in \bigcap_j \left( W_j \otimes_F F' \right)$ we must be able to write $\bm{f}=\sum_{l=1}^n \alpha_l\bm{e}_l$ for some $\alpha_l \in F'$. Putting this all together with have 
\begin{align*}
\sum_{i=1}^n (\lambda_{ij})_{j=1}^\infty\otimes f_i & = \sum_{i,l=1}^n (\lambda_{ij})_{j=1}^\infty \otimes \alpha_le_{li} = \sum_{l=1}^n \left( \sum_{i=1}^n (\lambda_{ij}e_{li})_{j=1}^\infty \right) \otimes \alpha_l =0
\end{align*}
and the proof is complete.
\end{proof}

We can now extend some of the results in \cite[\S5]{Ked06a} from free to stably free $\nabla$-modules over $R_{K\weak{\bm{x}},r}^u$. For any $\nabla$-module over any of the rings $R_{K\weak{\bm{x}},r}^u$, $R_{L,r}^u$, $\mathcal{R}_{K\weak{\bm{x}}}^u$ or $\mathcal{R}_L^u$ we will write $H^0_{\nabla_u}$ for the kernel of the derivation $\frac{\partial}{\partial u}$. For $M$  a $\nabla$-module over $R^u_{K\weak{\bm{x}},\eta}$, we will write $M_L$ for $M\otimes R^u_{L,\eta}$

\begin{lemma} \label{lemma: Rrfin} Let $M$ be a stably free $\nabla$-module over $R^u_{K\weak{\bm{x}},\eta}$, such that $M_L$ is unipotent relative to $L$. Then $H_{\nabla_u}^0(M)$ is a finite free $\nabla$-module over $K\weak{\bm{x}}$.
\end{lemma}

\begin{proof} All finitely generated $\nabla$-modules over $K\weak{\bm{x}}$ are projective, and therefore free, via the analogue of the Quillen--Suslin theorem for $K\weak{\bm{x}}$ \cite[Theorem 6.7]{Ked04c}. Thus $H_{\nabla_u}^0(M)$ will be free as soon as it is finitely generated. To see that it is finitely generated we may choose some $n$ such that
\[ M'= M \oplus R^{u,\oplus n}_{K\weak{\bm{x}},\eta} \]
is a free $\nabla$-module over $R^u_{K\weak{\bm{x}},\eta}$. The claim for $M$ can therefore be deduced from the claim for $M'$, we may therefore assume that $M$ is in fact free. In this case, \cite[Proposition 5.2.6]{Ked06a} shows that $H_{\nabla_u}^0(M_L)$ is finite dimensional over $L$, and it follows from \cite[Lemma 7.3.4]{Ked06a} that $H_{\nabla_u}^0(M)$ is finitely generated over $K\weak{\bm{x}}$. 
\end{proof}

In fact, the restriction to stably free modules is unnecessary.

\begin{lemma} \label{lemma: stab free} Any finitely presented $\nabla$-module over $R^u_{K{\weak{\bm{x}}},\eta}$ is projective, and becomes stably free after making a finite base extension $K\rightarrow K'$. 
\end{lemma}

\begin{proof}
Choose $\eta_-<\eta< \eta_+$ and $\lambda>1$ such that the module $M$ under consideration arises via base change from a $\nabla$-module over $K\tate{\lambda^{-1}\bm{x},\eta_-u^{-1},\eta_+^{-1}u}$. Since the completed local ring of any smooth affinoid $K$-algebra at any maximal ideal is a power series ring, it follows in the usual way that $M$ must be projective.
%

To see the claim on stable freeness, note that after replacing $K$ by a finite extension (and possibly shrinking the interval $[\eta_-,\eta_+]$ if necessary) we can assume that $\lambda, \eta_-,\eta_+ \in \norm{K^*}$, and that $\eta_-/\eta_+ = \norm{\varpi}$. In this case, by choosing $\alpha,\beta_\pm\in K$ with $\norm{\alpha}=\lambda$ and $\norm{\beta_{\pm}}=\eta_{\pm}$, we can construct an isomorphism
\begin{align*} K\tate{\lambda^{-1}\bm{x},\eta_-u^{-1},\eta_+^{-1}u} \cong \frac{K\tate{\bm{x},u,v}}{(uv-\varpi)} \\ \bm{x}\mapsto \alpha\bm{x},\;\; u\mapsto \beta_+ u ,\;\; u^{-1}\mapsto \beta_-v.  \end{align*}
Since $M$ is projective, it suffices to show that any projective $K\tate{\bm{x},u,v}/(uv-\varpi)$-module is stably free, or equivalently that $K_0(\tate{\bm{x},u,v}/(uv-\varpi))=\Z$, where $K_0(R)$ denotes the Grothendieck group of the category of finite projective $R$-modules (see \cite[Chapter II, Lemma 2.1]{Wei13}). We consider the diagram
\[  K_0\left(\frac{K\tate{\bm{x},u,v}}{(uv-\varpi)}\right) \leftarrow  K_0\left(\frac{\mathcal{O}_{K}\tate{\bm{x},u,v}}{(uv-\varpi)}\right)   \rightarrow   K_0\left(\frac{k[\bm{x},u,v]}{(uv)}\right). \]
Since $\mathcal{O}_K\tate{\bm{x},u,v}/(uv-\varpi)$ is $\varpi$-adically complete, it follows from \cite[Chapter II, Lemma 2.2]{Wei13} that the right hand map is an isomorphism. We can calculate $K_0(k[\bm{x},u,v]/(uv))$ using the Mayer-Vietoris exact sequence from \cite[p.28]{Mil71}. Namely, the diagram of rings
\[ \xymatrix{ k[\bm{x},u,v]/(uv) \ar[r]\ar[d] & k[\bm{x},u] \ar[d] \\ k[\bm{x},v] \ar[r] & k[\bm{x}]  } \]
is Cartesian, and all maps are surjective, so we have an exact sequence
\[ K_1(k[\bm{x},u]) \oplus K_1(k[\bm{x},v]) \rightarrow K_1(k[\bm{x}]) \rightarrow K_0(k[\bm{x},u,v]/(uv)) \rightarrow K_0(k[\bm{x},u]) \oplus K_0(k[\bm{x},v]) \rightarrow K_0(k[\bm{x}]) \]
where $K_1$ here refers to the Whitehead group as defined in \cite[\S3]{Mil71}. Since the map $k[\bm{x},u]\rightarrow k[\bm{x}]$ admits a section, the first map in this sequence is surjective. Since $K_0$ of any polynomial ring over a field is $\Z$, the sequence
\[  0\rightarrow K_0(k[\bm{x},u,v]/(uv)) \rightarrow \Z \oplus \Z \overset{\mathrm{sum}}{\longrightarrow} \Z \]
is exact and we deduce that $K_0(\mathcal{O}_K\tate{\bm{x},u,v}/(uv-\varpi)\cong K_0(k[\bm{x},u,v]/(uv))\cong \Z$. If we now let $G_0(R)$ denote the Grothendieck group of the category of finitely generated $R$-modules (over a Noetherian ring $R$), then since $\mathcal{O}_{K}\tate{\bm{x},u,v}/(uv-\varpi)$ and $K\tate{\bm{x},u,v}/(uv-\varpi)$ are regular it follows from \cite[Chapter II, Theorem 7.8]{Wei13} that 
\begin{align*}
K_0\left( \frac{\mathcal{O}_{K}\tate{\bm{x},u,v}}{(uv-\varpi)}\right) &\cong G_0\left( \frac{\mathcal{O}_{K}\tate{\bm{x},u,v}}{(uv-\varpi)}\right) \\ 
K_0\left( \frac{K\tate{\bm{x},u,v}}{(uv-\varpi)}\right) &\cong G_0\left( \frac{K\tate{\bm{x},u,v}}{(uv-\varpi)}\right).
\end{align*} Since
\[ G_0\left(\frac{\mathcal{O}_{K}\tate{\bm{x},u,v}}{(uv-\varpi)}\right)   \twoheadrightarrow G_0\left(\frac{K\tate{\bm{x},u,v}}{(uv-\varpi)}\right)\]
(see \cite[Chapter II, Application 6.4.1]{Wei13}) we deduce that
\[ K_0\left(\frac{\mathcal{O}_{K}\tate{\bm{x},u,v}}{(uv-\varpi)}\right)   \twoheadrightarrow K_0\left(\frac{K\tate{\bm{x},u,v}}{(uv-\varpi)}\right),\]
hence $K_0(K\tate{\bm{x},uv}/(uv-\varpi))=\Z$ as required.
\end{proof}

We can now prove an analogue of Theorem \ref{theo: unipotent free} with $R^u_{K\weak{\bm{x}},\eta}$ in place of $\mathcal{R}^u_{K\weak{\bm{x}}}$.  

\begin{proposition} \label{prop: reluni} Suppose that $M$ is a finitely presented $\nabla$-module over $R^u_{K\weak{\bm{x}},\eta}$, such that $M_L$ is unipotent relative to $L$. Then $M$ is free, and admits a strongly unipotent basis relative to $K\weak{\bm{x}}$.
\end{proposition}

\begin{proof} As observed above, $M$ is projective; we will induct on the rank of $M$. If the rank is zero then there is nothing to prove. If the rank is $>0$, we claim that the base change map
\[ H_{\nabla_u}^0(M) \otimes_{K\weak{\bm{x}}} R^u_{K\weak{\bm{x}},\eta} \rightarrow M \]
is injective, and $H_{\nabla_u}^0(M)$ is non-zero. Both these claims may be verified after making a finite base extension $K\rightarrow K'$, we may therefore assume by Lemma \ref{lemma: stab free} that $M$ is stably free. 

To see that $H_{\nabla_u}^0(M)$ is non-zero, choose $n$ such that
\[ N := M \oplus R^{u,\oplus n}_{K\weak{\bm{x}},\eta}\]
is free. Choose a strongly unipotent basis for $M_L$, and extend this to a basis for $N_L$ using the canonical basis of $R_{L,\eta}^{u,\oplus n}$. Let $D$ be the nilpotent matrix over $L$ giving the action of the connection on this strongly unipotent basis, and $e$ its nilpotency index. In \cite[\S5.3]{Ked06a} Kedlaya defines an $L$-linear function 
\[ f: N_L \rightarrow N_L \]
which satisfies $f(N) \subset N$ and $f(N\otimes_{K\weak{\bm{x}}} L) = \mathrm{im}(D^{e-1})$. In fact, it follow from the definition of $f$ (and our particular choice of strongly unipotent basis) that we also have $f(M_L) \subset M_L$, from which we deduce that $f(M) \subset M_L \cap N = M$. If we let $V$ denote the $L$-span of our original strongly unipotent basis for $M_L$, we therefore find that $f(M \otimes_{K\weak{\bm{x}}} L)= \mathrm{im}(D|_V^{e-1})$. In particular we must have $0\neq f(M) \subset H^0_{\nabla_u}(M)$.

The injectivity claim is a rather convoluted diagram chase. To begin, we observe that projectivity of $M$ over $R^u_{K\weak{\bm{x}},\eta}$ together with injectivity of $R_{K\weak{\bm{x}},\eta}^u \otimes_{K\weak{\bm{x}}} L\hookrightarrow R^u_{L,\eta}$ implies injectivity of
\[ M\otimes_{K\weak{\bm{x}}} L \rightarrow M_L. \]
Next, flatness of $K\weak{\bm{x}}\rightarrow L$ implies injecitvity of $H^0_{\nabla_u}(M)\otimes_{K\weak{\bm{x}}} L \rightarrow M \otimes_{K\weak{\bm{x}}} L$ and hence of 
\[ H^0_{\nabla_u}(M)\otimes_{K\weak{\bm{x}}} L \rightarrow H^0_{\nabla_u}(M_L). \]
Tensoring with $R_{L,\eta}^u$ gives injectivity of 
\[ H^0_{\nabla_u}(M)\otimes_{K\weak{\bm{x}}} R_{L,\eta}^u \rightarrow H^0_{\nabla_u}(M_L) \otimes_L R_{L,\eta}^u, \]
and since $H^0_{\nabla_u}(M)$ is projective over $K\weak{\bm{x}}$ we get injectivity of
\[ H^0_{\nabla_u}(M)\otimes_{K\weak{\bm{x}}} R_{K\weak{\bm{x}},\eta}^u \rightarrow H^0_{\nabla_u}(M)\otimes_{K\weak{\bm{x}}} R_{L,\eta}^u;\]
putting these two together gives injectivity of
\[ H^0_{\nabla_u}(M)\otimes_{K\weak{\bm{x}}} R_{K\weak{\bm{x}},\eta}^u \rightarrow H^0_{\nabla_u}(M_L) \otimes_L R_{L,\eta}^u.\]
Now, we already know that $H^0_{\nabla_u}(M_L) \otimes_L R_{L,\eta}^u \rightarrow M_L$ has to be injective (since $L$ is a complete, discretely valued field), and hence we deduce that the map $H^0_{\nabla_u}(M)\otimes_{K\weak{\bm{x}}} R_{K\weak{\bm{x}},\eta}^u\rightarrow M_L$ is injective, and therefore (finally!) obtain injectivity of
\[ H^0_{\nabla_u}(M)\otimes_{K\weak{\bm{x}}} R_{K\weak{\bm{x}},\eta}^u \rightarrow M.\]
Now, we have already seen in Lemma \ref{lemma: Rrfin} that $ H_{\nabla_u}^0(M)$ is free over $K\weak{\bm{x}}$, and the quotient $Q$ of $M$ by $H_{\nabla_u}^0(M) \otimes R^u_{K\weak{\bm{x}},\eta}$ is a finitely presented $\nabla$-module over $R^u_{K\weak{\bm{x}},\eta}$, thus again projective of strictly smaller rank. We may therefore apply the induction hypothesis to see that $Q$ admits a strongly unipotent basis relative to $K\weak{\bm{x}}$. This implies that $M$ is free and unipotent relative to $K\weak{\bm{x}}$, hence we may argue as in \cite[Proposition 5.2.6]{Ked06a} to show that it has a strongly unipotent basis relative to $K\weak{\bm{x}}$.
\end{proof}

\begin{proof}[Proof of Theorem \ref{theo: unipotent free}]
This is identical to the proof of \cite[Proposition 5.4.1]{Ked06a}. Let $M$ be an projective $\nabla$-module over $\mathcal{R}^u_{K\weak{\bm{x}}}$ with unipotent generic fibre. Then for $\eta$ close enough to $1$, $M$ comes from a projective $\nabla$-module $M_\eta$ over $\mathcal{R}^u_{K\weak{\bm{x}}} \cap R^u_{K\weak{\bm{x}},\eta}$ such that $M_\eta \otimes (\mathcal{R}_L^u \cap R_{L,\eta}^u)$ admits a strongly unipotent basis $\{ \bm{e}_i\}$ relative to $L$. Hence $M_\eta \otimes R^u_{K\weak{\bm{x}},\eta}$ admits a strongly unipotent basis $\{\bm{f}_i\}$ relative to $K\weak{\bm{x}}$ by Proposition \ref{prop: reluni}. Moreover, these two bases must have the same $L$-span inside $M \otimes \mathcal{R}_{L}^u$. Hence $\{\bm{f}_i\}$ forms a basis of $M_\eta \cap ( M\otimes \mathcal{R}_L^u ) = M$. 
\end{proof}

For us, the most important consequence of this result is the following. 

\begin{corollary} \label{cor: base change R0loc} Let $M$ be a projective, $F$-able $\nabla$-module over $\mathcal{R}^u_{K\weak{\bm{x}}}$ with constant irregularity. Then $H^0_{\nabla_u}(M)$ is finitely generated over $K\weak{\bm{x}}$, and for any closed point $s:K\weak{\bm{x}}\rightarrow K'$ the base change map
\[ H^0_{\nabla_u}(M) \otimes_{K\weak{\bm{x}}} K' \rightarrow H^0_{\nabla_u}(M \otimes_{\mathcal{R}^u_{K\weak{\bm{x}}}} \mathcal{R}_{K'}^u)\]
is an isomorphism.
\end{corollary}

\begin{proof}
After possibly enlarging $K$ and translating we may assume that $s=0$. If $K$ contains all $n$th roots of unity, and $n$ is coprime to $p$, then we have an identification
\[ H_{\nabla_u}^0(M) = H_{\nabla_{u^{1/n}}}^0\left( M\otimes_{\mathcal{R}^u_{K\weak{\bm{x}}}} \mathcal{R}^{u^{1/n}}_{K\weak{\bm{x}}} \right)^{\mathbb{Z}/n\mathbb{Z}} \]
and similarly after base change via $s$. We are therefore free to make such a tamely ramified base change at any point we wish.

Write $M_L=M\otimes_{\mathcal{R}^u_{K\weak{\bm{x}}}} \mathcal{R}^u_L$, and let $M_s=M\otimes_{\mathcal{R}^u_{K\weak{\bm{x}}}} \mathcal{R}_K$ be the fibre over $s$. Let $M_0$ denote the break $0$ part of $M$ provided by Proposition \ref{prop: spread}, thus the fibre $M_{0,s}$ over $s$ is the break $0$ part of $M_s$. Since $M$ is $F$-able, it follows that $M_{0,L}$ has rational exponents, whose denominators are necessarily coprime to $p$. We may therefore by \cite[Th\'eor\`eme 1.3-1]{Meb02} take a tamely ramified base change $\mathcal{R}_L^u\rightarrow \mathcal{R}_L^{u^{1/n}}$ such that $M_{0,L} \otimes_{\mathcal{R}^u_{L}} \mathcal{R}^{u^{1/n}}_L$ is unipotent relative to $L$ as a $\nabla$-module over $\mathcal{R}^{u^{1/n}}_L$. Since such a base extension preserves the break $0$ part we may therefore assume that $M_{0,L}$ is unipotent. But now since $M^{\nabla_u=0} = M_0^{\nabla_u=0}$ and $M_s^{\nabla_u=0}=M_{0,s}^{\nabla_u=0}$ we may moreover replace $M$ by $M_0$ and therefore assume that $M_L$ itself is unipotent relative to $L$.

Hence by Theorem \ref{theo: unipotent free} above, $M$ is in fact unipotent relative to $K\weak{\bm{x}}$, that is it is an iterated extension of $\nabla$-modules pulled back from $K\weak{\bm{x}}$. For any such pullback module, the base change claim is easily deduced using the projection formula, and the relatively unipotent case then follows by induction on the rank and the five lemma.
\end{proof}

\section{Base change for \texorpdfstring{$\mathbf{R}^0f_*$}{R0f}}\label{sec: bc for R10}

We can now use the results of the previous section to prove the $\mathbf{R}^0f_*M$ case of Theorem \ref{theo: main dagger}. We will therefore consider Setup \ref{setup: main}, and use notations as in \S\ref{sec: bgs}. Thus (in particular) we have a morphism of MW-type frames 
\[ (U,\overline{C},\mathfrak{C})\rightarrow ( S,\overline{S},\mathfrak{S} ) \]
enclosing an affine curve $f:U\rightarrow S$, with induced morphism $A\rightarrow B$ of $K$-dagger algebras as in \S\ref{sec: gnalk}. We let $M$ be a $\nabla$-module on $B$ obtained as the realisation of an overconvergent isocrystal $E$ on $U/K$. The main result of this section is the following.

\begin{theorem} \label{theo: bc for R0f} Assume that $M$ is $F$-able and has constant total irregularity $\mathrm{Irr}^\mathrm{tot}_M$. Then $\mathbf{R}^0f_*M$ is finitely generated over $A$, and for any closed point $s:A\rightarrow K'$ the base change map
\[ \mathbf{R}^0f_*M \otimes_A K' \rightarrow H^0(M_s) \]
is an isomorphism.
\end{theorem}

\begin{remark} \label{rem: r0 bc} It follows from the theorem that in fact formation of  $\mathbf{R}^0f_*M$ commutes with arbitrary base change $A\rightarrow A'$ of MW-type $K$-dagger algebras. As in Remark \ref{rem: Fstructure}, it then follows that any Frobenius structure on $M$ induces one on $\mathbf{R}^0f_*M$.
\end{remark}

In fact, showing that $\mathbf{R}^0f_*M$ is finitely generated is relatively straightforward, and does not depends on $M$ having constant irregularity. The base change claim is much harder, and is false without at least some extra assumption.\footnote{For a counter-example, let $f:\A^2_k\rightarrow \A^1_k$ be the projection, and take an Artin--Schreier cover $X\rightarrow \A^2_k$ whose  fibre over the generic point of $\A^1_k$ is connected, and whose fibre over $0$ is disconnected. Then the pushforward of the constant isocrystal on $X$ will not satisfy base change to closed points for $\mathbf{R}^0f_*$.} Our first reduction in the proof of Theorem \ref{theo: bc for R0f} is to show that the claim is local on $A$. 

\begin{lemma} Hypothesis as in Theorem \ref{theo: bc for R0f}. Let $\{ A\rightarrow A_i \}_{i\in I}$ be a finite dagger open cover of $A$, and set $A_{ij}=A_i\otimes^\dagger_A A_j$. Write $M_{A_i}=M\otimes_A A_i$ and $M_{A_{ij}} = M\otimes_A A_{ij}$. If the conclusions of Theorem \ref{theo: bc for R0f} hold for each $M_{A_i}$ and each $M_{A_{ij}}$, then they hold for $M$.
\end{lemma} 

\begin{proof}
By assumption, formation of $\mathbf{R}^0f_*M_{A_{i}}$ and $\mathbf{R}^0f_*M_{A_{ij}}$ commute with base change to closed points of $A_i$ and $A_{ij}$ respectively. It follows that formation of $\mathbf{R}^0f_*M_{A_i}$ commutes with base change along $A_i\rightarrow A_{ij}$. The dagger version of Tate's acyclicity theorem \cite[Proposition 2.6]{GK00} now gives the existence of a unique finite projective $A$-module $N$ whose base change to each $A_i$ is exactly $\mathbf{R}^0f_*M_{A_i}$, and the fact that 
\[ \mathbf{R}^0f_*M = \bigcap_i \mathbf{R}^0f_*M_{A_i} \]
(intersection inside $\mathbf{R}^0f_*M_{L}$) implies that $\mathbf{R}^0f_*M$ is canonically isomorphic to $N$. Hence formation of $\mathbf{R}^0f_*M$ commutes with each base change $A\rightarrow A_i$, and therefore to all closed points of $A$.
\end{proof}

Next, by localising on $S$, we may assume that $S$ admits a finite \'etale map to $\A^d_k$. Using Lemma \ref{lemma: finite etale lift}, we may therefore assume that there exists a finite \'etale map $K\weak{\bm{x}}\rightarrow A$. Corresponding to this we have a finite \'etale map $\mathcal{R}_{K\weak{\bm{x}}}^{u_j}\rightarrow \mathcal{R}_A^{u_j}$ for each $j$.

\begin{lemma} If $\mathrm{Irr}^\mathrm{tot}_M$ is constant, then each $M\otimes_B \mathcal{R}^{u_j}_{A}$ has constant irregularity.
\end{lemma} 

\begin{proof}
If $M\otimes_B \mathcal{R}^{u_j}_{A}$ does not have constant irregularity, then by Proposition \ref{prop: lower}  there exists a point of $\mathcal{M}(A)$ at which the irregularity of $M\otimes_B \mathcal{R}^{u_j}_{A}$ is strictly smaller than the irregularity at the maximal point. Since all the other irregularities cannot increase under such a specialisation, it follows that $\mathrm{Irr}^\mathrm{tot}_M$ cannot be constant.
\end{proof}

Consider $M\otimes_B \mathcal{R}_A^{u_j}$, as a $\nabla$-module over $\mathcal{R}_{K\weak{\bm{x}}}^{u_j}$ via restriction of scalars. By Lemma \ref{lemma: irr push} this has constant irregularity, and we can apply Corollary \ref{cor: base change R0loc} to deduce that $\mathbf{R}^0_\mathrm{loc}f_*M$ is finitely generated over $K\weak{\bm{x}}$, and its formation commutes with base change to closed points of $K\weak{\bm{x}}$. Hence $\mathbf{R}^0_\mathrm{loc}f_*M$ is finitely generated over $A$, and for every closed point $s:A\rightarrow K'$, the base change map
\[ \mathbf{R}^0_\mathrm{loc}f_* M \otimes_A K' \rightarrow H^0_\mathrm{loc}(M_s) \]
is an isomorphism. The fact that $\mathbf{R}^0f_*M$ is finitely generated (and thus projective) over $A$ now follows from the injection
\[\mathbf{R}^0f_*M \hookrightarrow \mathbf{R}^0_\mathrm{loc}f_*M.  \]
(There is in fact an easier way of seeing this, using \cite[Lemma 7.3.4]{Ked06a} but we will still need the full force of Corollary \ref{cor: base change R0loc} anyway). It remains to show the base change claim, and the key remaining input is the following. 

\begin{lemma} \label{lemma: injectivity} 
The direct image with compact support $\mathbf{R}^1f_!M$ is finitely generated projective over $A$, and for all closed points $s:A \rightarrow K'$  the base change map
\[ ( \mathbf{R}^1f_!M ) \otimes_A K' \rightarrow H^1_c(M_s) \]
is injective.
\end{lemma}

\begin{proof} To see that $\mathbf{R}^1f_!M$ is finitely generated, we use projectivity of $M$ to embed $M\otimes \mathcal{Q}^{\{\sigma_j\}}$ inside a number of copies of $\mathcal{Q}^{\{\sigma_j\}}$, and then apply \cite[Lemma 7.3.4]{Ked06a}. Since it has a natural $\nabla$-module structure it is thus projective.

For the base change claim, we may assume that $K=K'$. By translating we may assume that $s$ maps to the origin under the given finite \'etale morphism $K\weak{\bm{x}}\rightarrow A$. For $1\leq i\leq d$ let $A_i=A/(x_1,\dots,x_i)$; since $A_d$ is an \'etale $K$-algebra, it in fact suffices to show that the base change map
\[ ( \mathbf{R}^1f_!M_{A} ) \otimes_{A} A_{d} \rightarrow \mathbf{R}^1f_!M_{A_{d}}  \]
is injective. As we have already shown that each $\mathbf{R}^1f_!M_{A_i}$ is finite projective (and hence flat), it suffices to show that for all $i$ the base change map
\[ ( \mathbf{R}^1f_!M_{A_i} ) \otimes_{A_i} A_{i+1} \rightarrow \mathbf{R}^1f_!M_{A_{i+1}} \]
is injective. By induction on $d$, then, what we must show is that
\[ ( \mathbf{R}^1f_!M ) \otimes_{A} A_1 \rightarrow \mathbf{R}^1f_!M_{A_1} \]
is injective. One easily checks that the sequence
\[  0 \rightarrow x_1\mathcal{Q}_A^{\{\sigma_j\}}   \rightarrow \mathcal{Q}_A^{\{\sigma_j\}}   \rightarrow \mathcal{Q}_{A_1}^{\{\sigma_j\}} \rightarrow 0 \]
is exact, for example, by using the corresponding fact for each $\mathcal{R}_{A}^{u_j}$ and applying the nine lemma. Since $M$ is finite projective and $\ker \nabla$ is left exact, we deduce that
\[ 0\rightarrow x_1\mathbf{R}^1f_!M  \rightarrow \mathbf{R}^1f_!M\rightarrow \mathbf{R}^1f_!M_{A_1} \]
is exact. This finishes the proof. \end{proof}

We can now complete the proof of Theorem \ref{theo: bc for R0f}. Consider the diagram
\[ \xymatrix{ 
 0 \ar[r] & (\mathbf{R}^0f_*M) \otimes_{K\weak{\bm{x}}} K' \ar[r] \ar[d] & (\mathbf{R}^0_\mathrm{loc}f_*M) \otimes_{K\weak{\bm{x}}} K' \ar[d] \ar[r] & (\mathbf{R}^1f_!M) \otimes_{K\weak{\bm{x}}} K' \ar[d] \ar[r] & \\
 0 \ar[r] & H^0(M_s) \ar[r] & H^0_\mathrm{loc}(M_s) \ar[r] & H^1_c(M_s)  \ar[r] & 
} \]
which has exact rows, the top row being exact because $\mathbf{R}^0f_*M$, $\mathbf{R}^0_\mathrm{loc}f_*M$ and $\mathbf{R}^1f_!M$ are all finite projective over $A$. We already observed (using \ref{cor: base change R0loc}) that the middle vertical map is an isomorphism, and the right hand vertical arrow is injective by Lemma \ref{lemma: injectivity}. The left hand vertical map is therefore an isomorphism by the five lemma.

\section{The strong fibration theorem and the cohomology of punctured tubes}\label{sec: strong}

To prove Theorem \ref{theo: main dagger} for $\mathbf{R}^0f_*M$ we worked `algebraically', that is, within the language of dagger algebras. In order to deal with the $\mathbf{R}^1f_*M$ case we will need to do a little more geometry, and work in the language of frames. Again, we suppose that we are in Setup \ref{setup: main}, thus (in particular) we have a morphism of MW-type frames 
\[ \mathfrak{f}\colon(U,\overline{C},\mathfrak{C})\rightarrow ( S,\overline{S},\mathfrak{S} ) \]
enclosing an affine curve $f:U\rightarrow S$. We shall only consider the case when the base frame
\[ (S,\overline{S},\mathfrak{S})= (\A^d_k,\P^{d}_k,\widehat{\P}^{d}_\cur{V})\]
is the natural MW frame enclosing affine space over $k$. Again, let $A\rightarrow B$ be the induced morphism of dagger algebras as in \S\ref{sec: gnalk} (thus $A=K\weak{x_1,\ldots,x_d}$), and $M$ a $\nabla$-module over $B$ arising from an overconvergent isocrystal $E$ on $U/K$. We will assume that $M$ is $F$-able. Let
\[ \mathfrak{S}_0 = \widehat{\P}^{d-1}_\cur{V} \subset  \widehat{\P}^{d}_\cur{V}\]
be the hyperplane given in affine co-ordinates by $x_d=0$, and denote fibre product with $\mathfrak{S}_0$ over $\mathfrak{S}$ by $(-)_0$. Thus we have varieties $S_0,\overline{S}_0,U_0,C_0,\overline{C}_0$ and a formal scheme $\mathfrak{C}_0$. Set
\begin{align*}
\mathcal{R}_{(S_0,\mathfrak{S})}^+ &= \Gamma(]\overline{S}_0[_{\mathfrak{S}},j_{S_0}^\dagger\mathcal{O}_{]\overline{S}_0[_{\mathfrak{S}}}),
\end{align*}
we thus have an identification
\begin{align*} \mathcal{R}_{(S_0,\mathfrak{S})}^+ &= \mathcal{R}^{x_d+}_{K\weak{x_1,\ldots,x_{d-1}}}.
\end{align*}
We can also define
\[ \mathcal{R}_{(S_0,\mathfrak{S})}:= \mathrm{colim}_V \Gamma(V\cap ]\overline{S}_0[_{\mathfrak{S}},j_{S_0}^\dagger\mathcal{O}_{]\overline{S}_0[_{\mathfrak{S}}}) \]
where the colimit is over all strict neighbourhoods $V$ of $]\overline{S}\setminus \overline{S}_0[_\mathfrak{S}$ inside $\mathfrak{S}_K$. Again, we have an identification 
\[\mathcal{R}_{(S_0,\mathfrak{S})} = \mathcal{R}^{x_d}_{K\weak{x_1,\ldots,x_{d-1}}}. \] 
We similarly define
\begin{align*}
\mathcal{R}_{(U_0,\mathfrak{C})}^+ = \Gamma(]\overline{C}_0[_{\mathfrak{C}},j_{U_0}^{\dagger}\mathcal{O}_{]\overline{C}_0[_{\mathfrak{C}}})
\end{align*}
and 
\begin{align*}
\mathcal{R}_{(U_0,\mathfrak{C})} &= \mathrm{colim}_V \Gamma(V\cap ]\overline{C}_0[_{\mathfrak{C}},j_{U_0}^{\dagger}\mathcal{O}_{]\overline{C}_0[_{\mathfrak{C}}}),
\end{align*}
the colimit this time being over all strict neighbourhoods $V$ of $]\overline{C} \setminus \overline{C}_0[_\mathfrak{C}$ inside $\mathfrak{C}$. Since the closed immersion $U_0 \rightarrow U$ may no longer admit a smooth retraction (even locally on $U_0$), we cannot necessarily identify these rings with ordinary relative Robba rings. We can, however, compare them \emph{cohomologically} with ordinary relative Robba rings.

Choose an increasing sequence $\eta_n\rightarrow 1$, and let $V_n\subset ]\overline{S}_0[_{\mathfrak{S}}$ denote the open subspace defined by $\norm{x_d}\leq \eta_n$. Set
\begin{align*} R^{[0,\eta_n]}_{(S_0,\mathfrak{S})} &:=  \Gamma(V_n,j_{S_0}^\dagger\mathcal{O}_{]\overline{S}_0[_{\mathfrak{S}}}) \\
R^{[0,\eta_n]}_{(U_0,\mathfrak{C})} &:= \Gamma(\mathfrak{f}_K^{-1}(V_n),j_{U_0}^{\dagger}\mathcal{O}_{]\overline{C}_0[_{\mathfrak{C}}}),
\end{align*}
again we can interpret the former as a suitable `relative Tate algebra' of radius $\eta_n$ over $K\weak{x_1,\ldots,x_{d-1}}$ (we don't make this precise). Similarly, we let $V_n^m\subset ]\overline{S}_0[_{\fr{S}}$ denote the subspace defined by $\eta_m\leq \norm{x_d}\leq \eta_n$, and set   
\begin{align*} R^{[\eta_m,\eta_n]}_{(S_0,\mathfrak{S})} &:=  \Gamma(V^m_n,j_{S_0}^\dagger\mathcal{O}_{]\overline{S}_0[_{\mathfrak{S}}}) \\
R^{[\eta_m,\eta_n]}_{(U_0,\mathfrak{C})} &:= \Gamma(\mathfrak{f}_K^{-1}(V^m_n),j_{U_0}^{\dagger}\mathcal{O}_{]\overline{C}_0[_{\mathfrak{C}}}).
\end{align*}
From $M$, we obtain via base change $\nabla$-modules $M\otimes_B \mathcal{R}_{(U_0,\mathfrak{C})}^+$ and $M\otimes_B \mathcal{R}_{(U_0,\mathfrak{C})}$, with cohomology groups 
\[ \mathbf{R}^i\mathfrak{f}_*(M\otimes_B \mathcal{R}_{(U_0,\mathfrak{C})}^\#) := H^i\left( M\otimes_B \mathcal{R}_{(U_0,\mathfrak{C})}^\# \rightarrow M\otimes_B \mathcal{R}_{(U_0,\mathfrak{C})}^\# \otimes_B \Omega^1_{B/A} \right) \]
over $\mathcal{R}_{(S_0,\mathfrak{S})}^\#$, for $i=0,1$ and $\#\in \{+,\emptyset\}$. Similarly, we have cohomology groups 
\[ \mathbf{R}^i\mathfrak{f}_*(M\otimes_B R^{[0,\eta_n]}_{(U_0,\mathfrak{C})}) := H^i\left( M\otimes_B R_{(U_0,\mathfrak{C})}^{[0,\eta_n]} \rightarrow M\otimes_B R_{(U_0,\mathfrak{C})}^{[0,\eta_n]} \otimes_B \Omega^1_{B/A} \right) \]
over $ R^{[0,\eta_n]}_{(S_0,\mathfrak{S})}$ and 
\[ \mathbf{R}^i\mathfrak{f}_*(M\otimes_B R^{[\eta_m,\eta_n]}_{(U_0,\mathfrak{C})}) := H^i\left( M\otimes_B R_{(U_0,\mathfrak{C})}^{[\eta_m,\eta_n]} \rightarrow M\otimes_B R_{(U_0,\mathfrak{C})}^{[\eta_m,\eta_n]} \otimes_B \Omega^1_{B/A} \right) \]
over $ R^{[\eta_m,\eta_n]}_{(S_0,\mathfrak{S})}$. On the other hand, if we set $A_0=A/(x_d)=K\weak{x_1,\ldots,x_{d-1}}$ and $B_0=B/(x_d)$, then base changing along $B\rightarrow B_0$ gives an overconvergent $\nabla$-module $M_0$ over $B_0$.

\begin{theorem} \label{theo: ugly!} In the above situation, assume that $\mathrm{Irr}_M^\mathrm{tot}$ is constant, that $\mathbf{R}^0f_*M=0$, and that $\mathbf{R}^1f_{0*} M_0$ is finitely generated over $A_0$. Then:
\begin{enumerate} \item \label{num: ugly1} the natural maps
\begin{align*}  \mathbf{R}^1\mathfrak{f}_*(M\otimes_B \mathcal{R}_{(U_0,\mathfrak{C})}^+) &\rightarrow \lim_n  \mathbf{R}^1\mathfrak{f}_*(M\otimes_B R^{[0,\eta_n]}_{(U_0,\mathfrak{C})}) \\
\mathbf{R}^1\mathfrak{f}_*(M\otimes_B \mathcal{R}_{(U_0,\mathfrak{C})}) &\rightarrow \underset{m}{\mathrm{colim}} \lim_n  \mathbf{R}^1\mathfrak{f}_*(M\otimes_B R^{[\eta_m,\eta_n]}_{(U_0,\mathfrak{C})}) 
\end{align*}
are injective, and 
\[\mathbf{R}^0\mathfrak{f}_*(M\otimes_B \mathcal{R}_{(U_0,\mathfrak{C})}^+) =0 =\mathbf{R}^0\mathfrak{f}_*(M\otimes_B \mathcal{R}_{(U_0,\mathfrak{C})})  . \]
\item\label{num: ugly2} for each $n\geq m \geq 1$,
\[ \mathbf{R}^0\mathfrak{f}_*(M\otimes_B R^{[0,\eta_n]}_{(U_0,\mathfrak{C})})=0=\mathbf{R}^0\mathfrak{f}_*(M\otimes_B R^{[\eta_m,\eta_n]}_{(U_0,\mathfrak{C})}) \]
and there are isomorphisms
\begin{align*} 
 \mathbf{R}^1\mathfrak{f}_*(M\otimes_B R^{[0,\eta_n]}_{(U_0,\mathfrak{C})})&\cong \mathbf{R}^1f_{0*} M_0 \otimes_{A_0}R^{[0,\eta_n]}_{(S_0,\mathfrak{S})}  \\
\mathbf{R}^1\mathfrak{f}_*(M\otimes_B R^{[\eta_m,\eta_n]}_{(U_0,\mathfrak{C})})&\cong \mathbf{R}^1f_{0*} M_0 \otimes_{A_0}R^{[\eta_m,\eta_n]}_{(S_0,\mathfrak{S})}.
\end{align*}
\end{enumerate}
\end{theorem}

First let us show that the second claim implies the first. 

\begin{proof}[Proof that Theorem \ref{theo: ugly!} (\ref{num: ugly2}) $\Rightarrow$ Theorem \ref{theo: ugly!} (\ref{num: ugly1})]

First consider the groups $\mathbf{R}^i\mathfrak{f}_*(M\otimes_B \mathcal{R}_{(U_0,\mathfrak{C})}^+)$, which by definition are the kernel and cokernel of the connection
\[\nabla: M \otimes_B \mathcal{R}^+_{(U_0,\mathfrak{C})}  \rightarrow M \otimes_B \mathcal{R}^+_{(U_0,\mathfrak{C})} \otimes_B \Omega^1_{B/A} .\]
By assumption, for all $n\geq 1$ the sequence
\[  0 \rightarrow  M \otimes_B R^{[0,\eta_n]}_{(U_0,\mathfrak{C})} \overset{\nabla}{\longrightarrow} M \otimes_B R_{(U_0,\mathfrak{C})}^{[0,\eta_n]} \otimes_B \Omega^1_{B/A} \rightarrow \mathbf{R}^1\mathfrak{f}_*(M\otimes_B R_{(U_0,\mathfrak{C})}^{[0,\eta_n]})  \rightarrow 0 \]
is exact. Hence applying $\lim_n$ gives an exact sequence
\[ 0 \rightarrow  M \otimes_B \mathcal{R}^+_{(U_0,\mathfrak{C})} \overset{\nabla}{\longrightarrow} M \otimes_B \mathcal{R}^+_{(U_0,\mathfrak{C})} \otimes_B \Omega^1_{B/A} \rightarrow \lim_n\left( \mathbf{R}^1\mathfrak{f}_*(M\otimes_B R_{(U_0,\mathfrak{C})}^{[0,\eta_n]}\right)  \]
and proves the required vanishing of $\mathbf{R}^0\mathfrak{f}_*$ and injectivity claim for $\mathbf{R}^1\mathfrak{f}_*$. For $\mathbf{R}^i\mathfrak{f}_*(M\otimes_B \mathcal{R}_{(U_0,\mathfrak{C})})$ we fix $m$ and apply the same argument, before taking the colimit in $m$ and using the fact that filtered colimits are exact.
\end{proof}

The proof of the second part of Theorem \ref{theo: ugly!} will be via another base change argument, using the strong fibration theorem. The idea will be to replace the frame $(U_0,\overline{C}_0,\fr{C})$ by another frame in which we can do cohomological calculations more easily. So let $\mathfrak{D}=\mathfrak{C}_0 \times_\cur{V} \P^1_\cur{V}$. Then there exists a modification of frames (see Definition \ref{defn: mod})
\[ (U_0,\overline{C}'_0,\mathfrak{D}')\rightarrow (U_0,\overline{C}_0,\mathfrak{D})\]
 and a smooth proper morphism of frames
\[ (U_0,\overline{C}'_0,\mathfrak{D}') \rightarrow (S_0,\overline{S}_0, \mathfrak{S}). \]
extending the obvious rational map $\mathfrak{D}\dashrightarrow \mathfrak{S}$. We consider the fibre product diagram of frames 
\[\xymatrix{  & (U_0,\overline{C}''_0,\mathfrak{C}\times_\mathfrak{S} \mathfrak{D}') \ar[dr]^{\mathfrak{f}'} \ar[dl]_{\mathfrak{g}'} \\ (U_0,\overline{C}_0,\mathfrak{C})\ar[dr]_{\mathfrak{f}} &   & (U_0,\overline{C}'_0, \mathfrak{D}') \ar[r]\ar[dl]^{\mathfrak{g}} &   (U_0,\overline{C}_0,\mathfrak{D}) \\ & (S_0,\overline{S}_0,\mathfrak{S}).  } \]
where $\overline{C}_0''$ is the closure of $U_0$ inside the special fibre of $\mathfrak{C}\times_{\mathfrak{S}} \mathfrak{D}'$. If we denote by $E_0$ the restriction of $E$ to $U_0$, then $E_0$ has a realisations $E_{0,\mathfrak{C}}$ on $]\overline{C}_0[_{\mathfrak{C}}$, $E_{0,\mathfrak{D}}$ on $]\overline{C}_0[_{\mathfrak{D}}$, $E_{0,\mathfrak{D}'}$ on $]\overline{C}'_0[_{\mathfrak{D}'}$ and $E_{0,\mathfrak{C}\times_\mathfrak{S} \mathfrak{D}'}$ on $]\overline{C}_0''[_{\mathfrak{C}\times_\mathfrak{S} \mathfrak{D}'}$. Thus we have
\[ \mathcal{R}_{B_0}^{x_d+} \cong \Gamma(]\overline{C}'_0[,j_{U_0}^\dagger\mathcal{O}_{]\overline{C}_0'[_{\fr{D}'}})\]
as well as a similar interpretation for $\mathcal{R}_{B_0}^{x_d}$. Define
\begin{align*}
  R_{B_0}^{[0,\eta_n],x_d}&:=\Gamma(\mathfrak{g}_K^{-1}(V_n),j_{U_0}^\dagger\mathcal{O}_{]\overline{C}_0'[_{\fr{D}'}}) \\
R_{B_0}^{[\eta_m,\eta_n],x_d}&:=\Gamma(\mathfrak{g}_K^{-1}(V^m_n),j_{U_0}^\dagger\mathcal{O}_{]\overline{C}_0'[_{\fr{D}'}})
 \end{align*}
so we have cohomology groups
\[ \mathbf{R}^i\mathfrak{g}_*(M_0\otimes_{B_0} R^{[0,\eta_n],x_d}_{B_0}) := H^i\left( M_0\otimes_{B_0} R_{B_0}^{[0,\eta_n],x_d} \rightarrow M_0 \otimes_{B_0} R_{B_0}^{[0,\eta_n],x_d} \otimes_{B_0} \Omega^1_{B_0/A_0} \right) \]
over $ R^{[0,\eta_n]}_{(S_0,\mathfrak{S})}$ and 
\[ \mathbf{R}^i\mathfrak{g}_*(M_0\otimes_{B_0} R^{[\eta_m,\eta_n],x_d}_{B_0}) := H^i\left( M_0\otimes_{B_0} R_{B_0}^{[\eta_m,\eta_n],x_d} \rightarrow M_0 \otimes_{B_0} R_{B_0}^{[\eta_m,\eta_n],x_d} \otimes_{B_0} \Omega^1_{B_0/A_0} \right)  \]
over $ R^{[\eta_m,\eta_n]}_{(S_0,\mathfrak{S})}$.

\begin{lemma} For each $n\geq m\geq 1$, and $i=0,1$, there are isomorphisms
\begin{align*}
 \mathbf{R}^i\mathfrak{g}_*(M_0\otimes_{B_0} R^{[0,\eta_n],x_d}_{B_0}) &\cong  \mathbf{R}^i\mathfrak{f}_*(M\otimes_B R^{[\eta_m,\eta_n]}_{(U_0,\mathfrak{C})}) \\
 \mathbf{R}^i\mathfrak{g}_*(M_0\otimes_{B_0} R^{[\eta_m,\eta_n],x_d}_{B_0}) &\cong  \mathbf{R}^i\mathfrak{f}_*(M\otimes_B R^{[\eta_m,\eta_n]}_{(U_0,\mathfrak{C})})
\end{align*}
\end{lemma}

\begin{proof} Since coherent $j^\dagger_{U_0}\mathcal{O}_{]\overline{C}_0[_\fr{C}}$-modules are acylclic on both $\mathfrak{f}_K^{-1}(V_n)$ and $\mathfrak{f}_K^{-1}(V_n^m)$, and similarly coherent $j^\dagger_{U_0}\mathcal{O}_{]\overline{C}'_0[_\fr{D}}$-modules are acylclic on both $\mathfrak{g}_K^{-1}(V_n)$ and $\mathfrak{g}_K^{-1}(V_n^m)$, this follows from the strong fibration theorem.
\end{proof}

Hence Theorem \ref{theo: ugly!} follows from Theorem \ref{theo: robba bc 94} below.

\begin{theorem} \label{theo: robba bc 94} Suppose that $M_0$ is an $F$-able $\nabla$-module over $B_0$, satisfying the conditions of Theorem \ref{theo: main dagger} (i.e. finiteness and base change for relative cohomology). Then the base change maps 
\begin{align*} \mathbf{R}^if_{0*}M_0 \otimes_{A_0} R_{(S_0,\fr{S})}^{[0,\eta_n]} \rightarrow \mathbf{R}^i\mathfrak{g}_*(M_0\otimes_{B_0} R^{[0,\eta_n],x_d}_{B_0})  \\
 \mathbf{R}^if_{0*}M_0 \otimes_{A_0} R_{(S_0,\fr{S})}^{[\eta_m,\eta_n]} \rightarrow \mathbf{R}^i\mathfrak{g}_*(M_0\otimes_{B_0} R^{[\eta_m,\eta_n],x_d}_{B_0})  \end{align*}
are isomorphisms for $i=0,1$.
\end{theorem}

\subsection{Proof of Theorem \ref{theo: robba bc 94}}

In order to prove Theorem \ref{theo: robba bc 94}, we will first need to do a little bit of functional analysis. For the basic terminology of non-archimedean functional analysis, in particular the notions of locally convex vector spaces, Banach spaces, compact linear maps, and so on, we refer the reader to \cite{Sch02}. Recall that by now we are assuming $K$ to be discretely valued, in particular it is spherically complete.

\begin{definition} We call a locally convex $K$-vector space $V$ an LS-space if it is isomorphic to a countable colimit
\[ V\cong \mathrm{colim}_{i\geq 1} V_i \]
of Banach spaces such that the transition maps $V_i\rightarrow V_{i+1}$ are injective, and compact in the sense of \cite[\S16]{Sch02}
\end{definition}

Note that any LS space is separated by \cite[Lemma 16.9]{Sch02}, and that any separated quotient of an LS space is also an LS space. The basic example of an LS space that we have in mind is a $K$-dagger algebra.

\begin{lemma} Any $K$-dagger algebra $A$ is a LS space.
\end{lemma}

\begin{proof}
 First let us consider the case $A=K\weak{x_1,\ldots,x_d}=K\weak{\bm{x}}=\mathrm{colim}_{\lambda>1} K\tate{\lambda^{-1}\bm{x}}$. What we need to show is that if $\lambda>\lambda'>1$ then the map
\[ K\tate{\lambda^{-1}\bm{x}} \rightarrow K\tate{\lambda'^{-1}\bm{x}} \]
is compact. To do so, we take the unit ball $B_1\subset K\tate{\lambda^{-1}\bm{x}}$, i.e. the set of elements of $\lambda$-norm at most $1$. The closure of $B_1$ inside $K\tate{\lambda'^{-1}\bm{x}}$ is then complete, since $K\tate{\lambda'^{-1}\bm{x}}$ itself is. If we take $\epsilon>0$ and let $B'_\epsilon$ denote the ball of radius $\epsilon$ in $K\tate{\lambda'^{-1}\bm{x}}$, then for any element $\alpha\in K$ of norm $\geq \lambda^{-1}$ the finite set
\[ S = \left\{\left. a^{i_1+\ldots+i_d}x_1^{i_1}\ldots x_d^{i_d} \in K[x_1,\ldots,x_d]\right\vert (\lambda'\lambda^{-1})^{i_1+\ldots i_d} > \epsilon \right\} \]
of elements of $K\tate{\lambda'^{-1}\bm{x}}$ satisfies
\[ B_1 \subset \mathcal{V}\cdot S + B_\epsilon'.\]
Thus the closure of $B_1$ in $K\tate{\lambda'^{-1}\bm{x}}$ is compactoid and complete, thus bounded and $c$-compact by \cite[Proposition 12.7]{Sch02}. The map
\[ K\tate{\lambda^{-1}\bm{x}} \rightarrow K\tate{\lambda'^{-1}\bm{x}} \]
is therefore compact as claimed.

In general, we note that the LS topology on $K\weak{\bm{x}}$ is finer than the affinoid topology induced by the inclusion $K\weak{\bm{x}}\hookrightarrow K\tate{\bm{x}}$. Since any ideal of $K\weak{\bm{x}}$ is closed for the affinoid topology by \cite[\S1.4]{GK00}, it is therefore closed for the LS topology. Thus any $K$-dagger algebra is a separated quotient of an LS space, and hence an LS space.
\end{proof}

Similarly, any finitely generated module $M$ over a $K$-dagger algebra has a `canonical' LS topology, with the property that any $A$-linear map from $M$ into a topological $A$-module is continuous.

\begin{proposition} Let $V=\mathrm{colim}_{i\geq 1}V_i$ be an LS space. Then for any bounded subset $B\subset V$ there exists some $i$ such that $B$ is a bounded subset of $V_i$.
\end{proposition}

\begin{proof}
This is \cite[Lemma 16.9]{Sch02}.
\end{proof}

\begin{corollary} Let $V=\mathrm{colim}_{i\geq 1}V_i$ be an LS space, and $W$ a Banach space. Then any continuous linear map $W\rightarrow V$ factors through some $V_i$.
\end{corollary}

\begin{proof}
For any choice of norm $\Norm{\,\cdot\,}$ on $W$, the unit ball $\{\left. w\in W \right\vert \Norm{w}\leq 1 \}$ in $W$ is bounded and generates $W$ as a $K$-vector space.
\end{proof}

\begin{corollary} \label{cor: ind Ban} Let $V=\mathrm{colim}_{i\geq 1}V_i$ and $W=\mathrm{colim}_{i\geq 1}W_i$ be two LS spaces. Then any continuous linear map $V\rightarrow W$ is induced by a unique map
\[ \{ V_i \}_{i\geq 1} \rightarrow \{W_i\}_{i\geq 1}\]
in the category of ind-Banach spaces over $K$.
\end{corollary}

We can now finally deduce the result we will need to prove Theorem \ref{theo: robba bc 94}. 

\begin{corollary} Let $A$ be a $K$-dagger algebra, and \[ 0\rightarrow M \rightarrow N \rightarrow P \rightarrow 0\]
a short exact sequence of topological $A$-modules\footnote{that is, the maps are continuous, but are not \emph{a priori} assumed to be strict}, such that all three are LS spaces over $K$, and $P$ is finite free as an $A$-module. Then the LS topology on $P$ is the canonical one, and $N\cong M\oplus P$ in the category of ind-Banach spaces over $K$.
\end{corollary}

\begin{remark} \label{cor: split exact}By Corollary \ref{cor: ind Ban} the claim is independent of the presentation of $M,N$ or $P$ as colimits of Banach spaces with compact transition maps.
\end{remark}

\begin{proof}
First of all, choose finitely many elements $n_1,\ldots,n_r\in N$ inducing a basis of $P$ as an $A$-module. Then the map
\begin{align*}
 A^{\oplus r} &\rightarrow N \\
 (a_1,\ldots,a_r) &\mapsto \sum_{i=1}^r a_i n_i
 \end{align*}
is continuous for the canonical LS topology on $A^{\oplus r}$, and induces a continuous bijection $A^{\oplus r}\rightarrow P$. This continuous bijection is a topological isomorphism by \cite[II.36, Prop 10]{Bou81}, thus the topology on $P$ is the canonical one, the surjection $N\rightarrow P$ is strict, and the map $N\rightarrow P$ admits a continuous section $P\rightarrow N$. Since the image of $M\hookrightarrow N$ therefore has a topological complement, we deduce from \cite[Proposition 3.5]{Cre98} that $M\rightarrow N$ is strict, and hence the splitting $P\rightarrow N$ induces an isomorphism $N\cong M\oplus P$ of locally convex $K$-vector spaces. Finally, we apply Corollary \ref{cor: ind Ban} to conclude.
\end{proof}

We can now give the proof of Theorem \ref{theo: robba bc 94}.

\begin{proof}[Proof of Theorem \ref{theo: robba bc 94}]
We will give the proof for $R^{[\eta_m,\eta_n]}$, the proof for $R^{[0,\eta_n]}$ being essentially the same. First of all, as usual, since the claim is straightforward for modules of the form $N\otimes_{A_0} B_0$, we can successively replace $M_0$ by the cokernel of 
\[ \mathbf{R}^0f_{0*}M_0 \otimes_{A_0} B_0 \hookrightarrow M_0, \]
and reduce to consider the case when $\mathbf{R}^0f_{0*}M_0=0$.

As a projective $B_0$-module, $M_0$ is canonically an LS space (as it is a direct summand of a free $B_0$-module), this structure being induced by a family of partially defined norms $\Norm{\,\cdot\,}_\lambda$ on $M_0$, coming from affinoid norms on fringe algebras of $B_0$ arising from a fixed presentation $K\weak{x_1,\ldots,x_{n}}\rightarrow B_0$. We can then concretely describe $M_0 \otimes_{B_0} R_{B_0}^{[\eta_m,\eta_n],x_d}$ as the set of formal series
\[ \sum_i m_ix_d^i \]
with $m_i\in M_0$, such that there exists some $\lambda$ such that $\Norm{m_i}_\lambda$ exists for all $i$ and
\begin{align*} \Norm{m_i}_\lambda \eta_m^{i} &\rightarrow 0\;\;\text{ as }\;\;i \rightarrow -\infty \\
 \Norm{m_i}_\lambda \eta_n^{i} &\rightarrow 0\;\;\text{ as }\;\;i \rightarrow +\infty.
\end{align*}
We equip the set of such series with the obvious $R_{B_0}^{[\eta_m,\eta_n],x_d}$-module structure, there is of course an entirely similar description of $M_0 \otimes_{B_0} R_{B_0}^{[\eta_n,\eta_m],x_d} \otimes_{B_0} \Omega^1_{B_0/A_0}$. Let
\[ \nabla:M_0 \rightarrow M_0 \otimes_{B_0} \Omega^1_{B_0/A_0}\]
be the $A_0$-linear connection on $M_0$. Then the quotient topology on $\mathbf{R}^1f_*M_0$ is separated - this can be seen, for example, by comparing with the $p$-adic topology over the completed fraction field $L_0$ of $A_0$ and using \cite[Proposition 8.4.5]{Ked06a}. Hence the exact sequence 
\[ 0 \rightarrow M_0 \overset{\nabla}{\rightarrow}M_0 \otimes_{B_0} \Omega^1_{B_0/A_0} \rightarrow \mathbf{R}^1f_{0*}M_0 \rightarrow 0 \]
of topological $A_0$-modules satisfies the hypotheses of Corollary \ref{cor: split exact}, and is therefore split in the category of ind-Banach spaces over $K$.
 
This implies that if $\sum_im_ix_d^i$ is a series in $M_0 \otimes_{B_0} R_{B_0}^{[\eta_m,\eta_n],x_d}$, then $\sum_i\nabla(m_i)x_d^i$ also satisfies the convergence condition defining $M_0 \otimes_{B_0} R_{B_0}^{[\eta_m,\eta_n],x_d}$ inside $M\pow{x_d,x_d^{-1}}$, and the map 
\[ \nabla : M_0\otimes_{B_0} R_{B_0}^{[\eta_m,\eta_n],x_d} \rightarrow M_0 \otimes_{B_0} R_{B_0}^{[\eta_m,\eta_n],x_d} \otimes_{B_0} \Omega^1_{B_0/A_0} \]
is given by $\sum_im_ix_d^i \mapsto \sum_i\nabla(m_i)x_d^i$. It is then clear that the kernel of $\nabla$ on $M_0\otimes_{B_0} R_{B_0}^{[\eta_m,\eta_n],x_d} $ is zero, which proves the base change claim for $\mathbf{R}^0f_{0*}M_0$. To deal with the $\mathbf{R}^1f_{0*}M_0$ case, choose elements $\bm{e}_1,\ldots,\bm{e}_n$ in $M_0 \otimes_{B_0} \Omega^1_{B_0/A_0}$ lifting a basis of $\mathbf{R}^1f_{0*}M_0$.  Since $\mathbf{R}^0f_*M_0=0$, every $m \in M_0 \otimes_{B_0} \Omega^1_{B_0/A_0}$ can be written \emph{uniquely} as
\[ m = \nabla(n)+\sum_j \alpha_{j} \bm{e}_j \]
for elements $n \in M_0$ and $\alpha_{j}\in A_0$. Thus given any $\sum_i m_ix_d^i  \in M_0 \otimes_{B_0} R_{B_0}^{[\eta_m,\eta_n],x_d} \otimes_{B_0} \Omega^1_{B_0/A_0}$ we can write each $m_i=\nabla(n_i)+\sum_j\alpha_{ij}\bm{e}_j$, and again by split exactness of the sequence
\[ 0 \rightarrow M_0 \overset{\nabla}{\rightarrow}M_0 \otimes_{B_0} \Omega^1_{B_0/A_0} \rightarrow \mathbf{R}^1f_{0*}M_0 \rightarrow 0 \]
of ind-Banach spaces, it follows that the sums
\[ \sum_i n_ix_d^i, \;\;\;\;\text{ and }\;\;\;\;  \sum_i \alpha_{ij}x_d^i, \;\;1\leq j\leq n \]
satisfy the convergence conditions required to define elements of inside $M_0\otimes_{B_0}  R_{B_0}^{[\eta_m,\eta_n],x_d} $ and $R^{[\eta_m,\eta_n]}_{(S_0,\fr{S})}$ respectively. Therefore we can write
\[ \sum_i m_ix_d^i = \nabla\left( \sum_i n_ix_d^i \right) +\sum_j\left(\sum_i\alpha_{ij}x_d^{i} \right) \bm{e}_j, \]
for unique elements $\sum_i\alpha_{ij}x_d^{i}\in R^{[\eta_m,\eta_n]}_{(S_0,\fr{S})}$ and $\sum_in_ix_d^i\in M_0\otimes_{B_0}  R_{B_0}^{[\eta_m,\eta_n],x_d} $, and this implies that the $\bm{e}_i$ also form a basis for $\mathbf{R}^1\mathfrak{g}_*\left(M_0 \otimes_{B_0}R_{B_0}^{[\eta_m,\eta_n],x_d} \right)$ as an $ R^{[\eta_m,\eta_n]}_{(S_0,\fr{S})}$-module.
\end{proof}

\section{Base change for \texorpdfstring{$\mathbf{R}^1f_*$}{R1f}}\label{sec: universal}

We now have the necessary tools to prove the $\mathbf{R}^1f_*M$ case of Theorem \ref{theo: main dagger}. Again, we consider Setup \ref{setup: main}, thus (in particular) we have a morphism of MW-type frames 
\[ (U,\overline{C},\mathfrak{C})\rightarrow ( S,\overline{S},\mathfrak{S} ) \]
enclosing an affine curve $f:U\rightarrow S$, with induced morphism $A\rightarrow B$ of $K$-dagger algebras as in \S\ref{sec: gnalk}. We let $M$ be a $\nabla$-module on $B$ obtained as the realisation of an overconvergent isocrystal $E$ on $U/K$. 
 
\begin{theorem} \label{theo: bc for R1f} Assume that $M$ is $F$-able, and has constant total irregularity $\mathrm{Irr}_M^\mathrm{tot}$. Then $\mathbf{R}^1f_*M$ is finitely generated over $A$, and for any closed point $s:A\rightarrow K'$ the base change map
\[ \mathbf{R}^1f_*M \otimes_A K' \rightarrow H^1(M_s) \]
is an isomorphism.
\end{theorem}

\begin{remark} \label{rem: r1 bc} As in Remark \ref{rem: r0 bc}, it follows that formation of $\mathbf{R}^1f_*M$ commutes with arbitrary base change $A\rightarrow A'$ of MW-type $K$-dagger algebras, and any Frobenius structure on $M$ induces one on $\mathbf{R}^1f_*M$.
\end{remark}

The key case to consider will be when we have $\mathbf{R}^0f_*M=0$.

\begin{theorem} \label{theo: special finiteness} Assume that $M$ is $F$-able. Suppose that for all closed points $s:A\rightarrow K'$ we have
\[ H^0(M_s)=0,\;\;\dim_{K'}H^1(M_s)=m\]
for some non-negative integer $m$, independent of $s$. Then $\mathbf{R}^1f_*M$ is locally free of rank $m$, and for any closed point $s:A\rightarrow K'$ the base change map
\[ \mathbf{R}^1f_*M \otimes_A K' \rightarrow H^1(M_s) \]
is an isomorphism.
\end{theorem}

\begin{proof}[Theorem \ref{theo: special finiteness} $\Rightarrow$  Theorem \ref{theo: bc for R1f}]
Given a general $M$ as in Theorem \ref{theo: bc for R1f} we have by Theorem \ref{theo: bc for R0f} an injection
\[ \mathbf{R}^0f_*M \otimes_A B \rightarrow M \]
of $F$-able $\nabla$-modules. Since relatively constant $\nabla$-modules have zero irregularity, we deduce that the hypotheses of Theorem \ref{theo: bc for R1f} remain true for the cokernel of this injection. Moreover, it follows easily from the projection formula that the conclusions of Theorem \ref{theo: bc for R1f} hold for the $\nabla$-module $\mathbf{R}^0f_*M \otimes_A B$. Appealing to the five lemma and iterating, we may therefore reduce to considering the case when $\mathbf{R}^0f_*M=0$. Theorem \ref{theo: bc for R0f} then implies that $H^0(M_s)=0$ for all such $s$, and hence \cite[ Corollaire 5.0-12]{CM01} implies that $\dim_{K'}H^1(M_s)$ is constant. Thus we may apply Theorem \ref{theo: special finiteness}.
\end{proof}

The situation here is the opposite to the one we had in \S\ref{sec: bc for R10} - the hard part is showing finiteness of $\mathbf{R}^1f_*M$, the base change claim will then follow relatively easily. To prove Theorem \ref{theo: special finiteness} we will proceed by induction on the dimension $d=\dim A$, the case $d=0$ amounting to finiteness of rigid cohomology with coefficients for smooth curves \cite[\S6]{Ked06a}. The main structure of the proof will be essentially geometric, working on the `weak formal scheme' $\mathrm{Spf}(A^\mathrm{int})$. Theorem \ref{theo: generic pushforwards} tells us that $\mathbf{R}^1f_*M$ becomes coherent on some open subspace $W \subset \mathrm{Spf}(A^\mathrm{int})$, and we will use constancy of $\dim_{K'}H^1$ together with the induction hypothesis to successively enlarge the open set $W$. A basic outline of the argument we will use can be found in the proof of the following lemma, and was already used in the proof of Proposition \ref{prop: spread}.
 
 \begin{lemma} \label{lemma: R1 inj} Assumptions as in Theorem \ref{theo: special finiteness}. For any dagger localisation $A\rightarrow A'$ the map
\[ \mathbf{R}^1f_*M \rightarrow \mathbf{R}^1f_*M_{A'}  \]
is injective.
\end{lemma}

\begin{proof}
Write $\nabla$ for the $A$-linear connection on $M$, and for any dagger localisation $A\rightarrow A''$, write $B_{A''}=B\otimes^\dagger_A A''$. We need to show that if $m \in M_{A'}$ is such that $\nabla(m)\in M\otimes_B \Omega^1_{B/A}$, then in fact $m \in M$. By the dagger analogue of Tate's acyclicity theorem, the question is local on $A$ in the sense that it suffices to produce a dagger open cover $\{A \rightarrow A_i\}$ such that $m \in M_{A_i}$ for all $i$.

Suppose therefore that we have some collection of dagger localisations $\mathcal{C} = \{A \rightarrow A_i\}$ (not necessarily covering $A$) such that $m \in M_{A_i}$ for all $i$. A non-empty such $\mathcal{C}$ exists by hypothesis. Let $U_\mathcal{C}$ denote the union of the images of the induced open immersions $U_i \rightarrow S$ on the reductions. We shall show that if $U_\mathcal{C}\neq  S$, then we can find another dagger localisation $A\rightarrow A''$ such that $m\in M_{A''}$, and adding $A''$ to $\mathcal{C}$ enlarges $U_\mathcal{C}$. The result will then follow by Noetherian induction.

Now, if $U_\mathcal{C}\neq S$, then after possibly enlarging $K$, which is harmless, we may choose a smooth $k$-rational point $z$ on the (reduced) complement of $U_\mathcal{C}$ in $S$. Localising around $z$ we may by \cite[Theorem 1]{Ked05} pick a finite \'etale map $(x_1,\ldots,x_d):S \rightarrow \A^d_k$ such that $S \setminus U_\mathcal{C}$ maps into the hyperplane $\{x_d=0\}$. It suffices to produce some $A''$ as above such that $\spec{\overline{A}''}\subset S$ contains $z$. After applying Lemma \ref{lemma: change frame} we may assume that our \'etale morphism $S\rightarrow \A^d_k$ extends to a proper, \'etale, Cartesian morphism of frames 
\[(S,\overline{S},\mathfrak{S}) \rightarrow (\A^d_k,\P^d_k,\widehat{\P}^d_\cur{V}), \]
in particular it lifts to a finite \'etale map $K\weak{x_1,\ldots,x_d}\rightarrow A$. Pushing forward along this morphism we may therefore assume that
\[(S,\overline{S},\mathfrak{S})= (\A^d_k,\P^{d}_k,\widehat{\P}^{d}_\cur{V}) \]
and that $U_\mathcal{C} \supset \{x_d\neq 0\} $ as open subschemes of $\A^d_k$. Thus we have $A'=K\weak{x_1,\ldots,x_d,x_d^{-1}}$, and we are now in a position to apply the results of \S\ref{sec: strong} above; we will freely use the notations introduced there. We have a Cartesian diagram of rings
\[ \xymatrix { B \ar[r] \ar[d] & \mathcal{R}^+_{(U_0,\mathfrak{C})} \ar[d] \\
B\weak{x_d^{-1}} \ar[r] & \mathcal{R}_{(U_0,\mathfrak{C})} } \]
and an $F$-able $\nabla$-module $M$ over $B$, with constant total irregularity. We need to show that the induced map
\[ \mathbf{R}^1f_*M \rightarrow \mathbf{R}^1f_*\left( M_{A'}\right) \]
is injective. It therefore suffices to show that
\[ \nabla:\frac{M_{A'}}{M}\rightarrow \frac{M_{A'}}{M} \]
and since the above square is Cartesian (and $M$ is projective), that 
\[ \nabla:  M\otimes_B \frac{\mathcal{R}_{(U_0,\mathfrak{C})}}{\mathcal{R}^+_{(U_0,\mathfrak{C})}} \rightarrow M\otimes_B \frac{\mathcal{R}_{(U_0,\mathfrak{C})}}{\mathcal{R}^+_{(U_0,\mathfrak{C})}}  \]
is injective. By the induction hypothesis we know that $M_0$ satisfies both conditions of Theorem \ref{theo: main dagger}. Thus we may apply Theorem \ref{theo: ugly!} to deduce that
\[ \mathbf{R}^0\mathfrak{f}_*\left( M\otimes_B \mathcal{R}_{(U_0,\mathfrak{C})} \right)=0, \]
and
\[ \mathbf{R}^1\mathfrak{f}_*\left( M\otimes_B \mathcal{R}^+_{(U_0,\mathfrak{C})} \right) \hookrightarrow \mathbf{R}^1\mathfrak{f}_*\left( M\otimes_B \mathcal{R}_{(U_0,\mathfrak{C})} \right). \]
We can now use the long exact sequence associated to 
\[ 0 \rightarrow M\otimes_B \mathcal{R}^+_{(U_0,\mathfrak{C})} \rightarrow M\otimes_B \mathcal{R}_{(U_0,\mathfrak{C})} \rightarrow M\otimes_B \frac{\mathcal{R}_{(U_0,\mathfrak{C})}}{\mathcal{R}^+_{(U_0,\mathfrak{C})}} \rightarrow 0\]
to conclude.
 \end{proof}

\begin{lemma} \label{lemma: local R1} Hypotheses as in Theorem \ref{theo: special finiteness}. If there exists an open cover $\{A\rightarrow A_i\}$ such that the conclusions of Theorem \ref{theo: special finiteness} hold for each $M_{A_i}$, then they hold for $M$. 
\end{lemma}

 \begin{proof}
Choose a dagger localisation $A\rightarrow A'$ such that the conclusions of Theorem \ref{theo: generic pushforwards} hold for $M_{A'}$. By shrinking $A'$ we may assume that $A_i \rightarrow A'$ for all $i$, and hence that $A_{ij}:=A_i\otimes^\dagger_A A_j \hookrightarrow A'$ for all $i,j$. By comparing with closed points of $A'$, we therefore have base change isomorphisms
 \[
 \mathbf{R}^1f_*M_{A_i} \otimes_{A_i} A' \isomto  \mathbf{R}^1f_*M_{A'} \]
 which we can use to embed $\mathbf{R}^1f_*M_{A_i}$ and $\mathbf{R}^1f_*M_{A_i} \otimes_{A_i} A_{ij}$ inside $\mathbf{R}^1f_*M_{A'}$. Let $N_{ij}$ denote the sum of $\mathbf{R}^1f_*M_{A_i} \otimes_{A_i} A_{ij}$ and $\mathbf{R}^1f_*M_{A_j} \otimes_{A_j} A_{ij}$ inside $\mathbf{R}^1f_*M_{A'}$, this is therefore an overconvergent $\nabla$-module over $A_{ij}$. Since
 \[  N_{ij} \hookrightarrow \mathbf{R}^1f_*M_{A'} \]
 we can deduce by applying \cite[Proposition 5.3.1]{Ked07} to the kernel of 
\[ N_{ij} \otimes_{A_{ij}} A' \rightarrow \mathbf{R}^1f_*M_{A'} \]
that in fact the map $N_{ij} \otimes_{A_{ij}} A' \rightarrow \mathbf{R}^1f_*M_{A'}$ remains injective. We thus find that
 \[ \left(\mathbf{R}^1f_*M_{A_i} \otimes_{A_i} A_{ij}+\mathbf{R}^1f_*M_{A_j} \otimes_{A_j} A_{ij}\right)\otimes_{A_{ij}} A' \isomto \mathbf{R}^1f_*M_{A'}, \]
 and in particular, the cokernel of the natural map
 \[ \mathbf{R}^1f_*M_{A_i} \otimes_{A_i} A_{ij} \rightarrow \mathbf{R}^1f_*M_{A_i} \otimes_{A_i} A_{ij}+\mathbf{R}^1f_*M_{A_j} \otimes_{A_j} A_{ij} \]
 has to vanish after applying $-\otimes_{A_{ij}} A'$. Since this cokernel is an overconvergent $\nabla$-module, it must already be zero over $A_{ij}$, and thus we deduce that
 \[ \mathbf{R}^1f_*M_{A_i} \otimes_{A_i} A_{ij} = \mathbf{R}^1f_*M_{A_j} \otimes_{A_j} A_{ij} \]
inside $\mathbf{R}^1f_*M_{A'}$. Hence by descent for coherent sheaves on dagger spaces, the intersection $ N=\bigcap_i \mathbf{R}^1f_*M_{A_i}$ of all the $\mathbf{R}^1f_*M_{A_i}$ inside $\mathbf{R}^1f_*M_{A'}$ is a locally free $A$-module of rank $m$ whose base change to $A_i$ is exactly $\mathbf{R}^1f_*M_{A_i}$. By Lemma \ref{lemma: R1 inj} above, $\mathbf{R}^1f_*M$ naturally embeds into $N$: it is thus finite, and hence a $\nabla$-module, locally free of rank $r\leq m$. However, since the base change map
\[ \mathbf{R}^1f_*M_{A} \otimes_A K' \rightarrow H^1(M_s) \] 
associated to any closed point $s$ is surjective (which is immediate from the definitions), we deduce that in fact $r=m$, $\mathbf{R}^1f_*M=N$, and the base change map along closed points $A\rightarrow K'$, as well as to each $A_i$, is an isomorphism. \end{proof}

We need one more lemma before we can prove Theorem \ref{theo: special finiteness}, which in a way is the key result making the whole approach work.

\begin{lemma} \label{lemma: gluing} Let $N_1,N_2$ be finite projective modules of rank $m$ over $K\weak{x_1,\ldots,x_d,x_d^{-1}}$ and $\mathcal{R}^{x_d+}_{K\weak{x_1,\ldots,x_{d-1}}}$ respectively, and let
\[\alpha: N_1\otimes_{K\weak{x_1,\ldots,x_d,x_d^{-1}}} \mathcal{R}^{x_d}_{K\weak{x_1,\ldots,x_{d-1}}} \isomto N_2\otimes_{\mathcal{R}^{x_d+}_{K\weak{x_1,\ldots,x_{d-1}}}} \mathcal{R}^{x_d}_{K\weak{x_1,\ldots,x_{d-1}}} \]
be an $\mathcal{R}_{K\weak{x_1,\ldots,x_{d-1}}}^{x_d}$-linear isomorphism. Then the intersection of $N_1$ and $\alpha^{-1}(N_2)$ inside $N_1\otimes_{K\weak{x_1,\ldots,x_d,x_d^{-1}}}\mathcal{R}^{x_d}_{K\weak{x_1,\ldots,x_{d-1}}}$ is a finite projective $K\weak{x_1,\ldots,x_d}$-module of rank $m$.
\end{lemma}

\begin{remark} It is in order to be able to apply this key lemma that makes it vital to work with overconvergent, rather than just convergent, relative cohomology groups.
\end{remark}

\begin{proof} First choose $\lambda>1$ close enough to $1$ that there exists a locally free sheaf $\cur{E}_\lambda$ on the rigid analytic space
\[ U_\lambda := \left\{ \norm{x_i}\leq \lambda,\;\;\norm{x_d} \geq\lambda^{-1} \right\}\]
whose module of global sections tensored with $K\weak{x_1,\ldots,x_d,x_d^{-1}}$ is precisely $N_1$. Next choose $\lambda^{-1}<\rho<1$ and $1<\eta_\rho<\lambda$ close enough to $1$ such that there exists a locally free sheaf $\mathcal{E}_\rho$ on
\[ U_{\rho}:= \left\{ \norm{x_i}\leq \eta_{\rho},\;\;\norm{x_d} \leq \rho \right\} \]
whose module of global sections tensored with $\cup_{\eta>1} K\tate{\eta^{-1}x_1,\ldots,\eta^{-1}x_{d-1},\rho^{-1}x_d}$ coincides with $N_2 \otimes \cup_{\eta>1} K\tate{\eta^{-1}x_1,\ldots,\eta^{-1}x_{d-1},\rho^{-1}x_d}$. After possibly increasing $\lambda$, the isomorphism $\alpha$ is defined over
\[ \cap_{\lambda^{-1} \leq \rho' < 1 } \cup_{\eta>1} K\tate{\eta^{-1}x_1,\ldots,\eta^{-1}x_{d-1},\lambda x_d^{-1},\rho'^{-1}x_d} \]
and hence (after possibly decreasing $\eta_\rho$) induces an isomorphism
\[ \left.\mathcal{E}_\lambda\right\vert_{U_\lambda
 \cap U_{
 \rho}} \isomto  \left.\mathcal{E}_\rho\right\vert_{U_\lambda\cap U_\rho}  \]
of locally free sheaves on
\[ U_{\lambda} \cap U_{\rho} = \left\{ \norm{x_i}\leq \eta_\rho,\;\;\lambda^{-1}\leq \norm{x_d}\leq \rho \right\}. \]
Thus $\mathcal{E}_\lambda$ and $\mathcal{E}_{\rho}$ glue to give a locally free sheaf $\mathcal{E}$ on
\[  U_{\lambda} \cup U_{\rho} = \left\{ \norm{x_i}\leq \lambda \right\}. \]
Set $N=\Gamma(U_\lambda\cup U_\rho,\mathcal{E}) \otimes K\weak{x_1,\ldots,x_d}$. This is a then a finite projective (and therefore free) module over $K\weak{x_1,\ldots,x_d}$ such that $N_1 = N \otimes K\weak{x_1,\ldots,x_d,x_d^{-1}}$ and $N_2= N \otimes \mathcal{R}^{x_d+}_{K\weak{x_1,\ldots,x_{d-1}}}$. The result then follows from the fact that the digram
\[\xymatrix{ K\weak{x_1,\ldots,x_d} \ar[r]\ar[d] & \mathcal{R}^{x_d+}_{K\weak{x_1,\ldots,x_{d-1}}} \ar[d] \\ K\weak{x_1,\ldots,x_d,x_d^{_1}}\ar[r] & \mathcal{R}^{x_d}_{K\weak{x_1,\ldots,x_{d-1}}} }   \]
of rings is Cartesian.
\end{proof}

\begin{proof}[Proof of Theorem \ref{theo: special finiteness}]
The claim we are trying to prove, i.e. finite generation of $\mathbf{R}^1f_*M$ and commutation with base change to closed points, is local on $A$ by Lemma \ref{lemma: local R1}, and we can now argue entirely similarly to the proof of Lemma \ref{lemma: R1 inj} above.

In other words, we know by Theorem \ref{theo: generic pushforwards} and Lemma \ref{lemma: base change 1} that after making a localisation $A\rightarrow A'$ the higher direct image $\mathbf{R}^1f_*M$ becomes locally free and commutes with base change to closed points. By extending $K$ and using Noetherian induction, it suffices to show that the same also holds over some dagger localisation of $A$ containing the residue disc of a given smooth rational point of the complement $S \setminus \spec{\overline{A}'}$.

Localising around this point, applying \cite[Theorem 1]{Ked05} and lifting, and using Lemma \ref{lemma: local R1} above, we can reduce to the case when $A=K\weak{x_1,\ldots,x_d}$ and $A'=K\weak{x_1,\ldots,x_d,x_d^{-1}}$. Again, we are now in a position to apply the results from \S\ref{sec: strong} above, and we will freely use the notation introduced there. Consider the commutative diagram
\[ \xymatrix{ \mathbf{R}^1f_*M  \ar[r] \ar[d] & \mathbf{R}^1\mathfrak{f}_*\left( M \otimes_B \mathcal{R}^+_{(U_0,\mathfrak{C})} \right) \ar[d] \\ \mathbf{R}^1f_*M_{A'}   \ar[r] & \mathbf{R}^1\mathfrak{f}_*\left( M \otimes_B \mathcal{R}_{(U_0,\mathfrak{C})} \right). } \]
Applying Theorem \ref{theo: ugly!} produces (injective) maps
\begin{align*} 
 \mathbf{R}^1\mathfrak{f}_*\left( M \otimes_B \mathcal{R}^+_{(U_0,\mathfrak{C})} \right) &\rightarrow \mathbf{R}^1f_{0*}M_0 \otimes_{K\weak{x_1,\ldots,x_{d-1}}} \mathcal{R}_{K\weak{x_1,\ldots,x_{d-1}}}^{x_d+} \\
 \mathbf{R}^1\mathfrak{f}_*\left( M \otimes_B \mathcal{R}_{(U_0,\mathfrak{C})} \right) &\rightarrow \mathbf{R}^1f_{0*}M_0 \otimes_{K\weak{x_1,\ldots,x_{d-1}}} \mathcal{R}_{K\weak{x_1,\ldots,x_{d-1}}}^{x_d}
\end{align*} 
and hence a commutative diagram as follows:
\[ \xymatrix{ \mathbf{R}^1f_*M  \ar[r] \ar[d] & \mathbf{R}^1f_{0*}M_0 \otimes_{K\weak{x_1,\ldots,x_{d-1}}} \mathcal{R}_{K\weak{x_1,\ldots,x_{d-1}}}^{x_d+} \ar[d] \\ \mathbf{R}^1f_*M_{A'}   \ar[r] & \mathbf{R}^1f_{0*}M_0 \otimes_{K\weak{x_1,\ldots,x_{d-1}}} \mathcal{R}_{K\weak{x_1,\ldots,x_{d-1}}}^{x_d}.  } \]
I claim that the base change map
\[ \mathbf{R}^1f_*M_{A'}  \otimes_{A'} \mathcal{R}^{x_d}_{K\weak{x_1,\ldots,x_{d-1}}} \rightarrow \mathbf{R}^1f_{0*}M_0 \otimes_{K\weak{x_1,\ldots,x_{d-1}}} \mathcal{R}_{K\weak{x_1,\ldots,x_{d-1}}}^{x_d} \]
is an isomorphism. Indeed, since both sides are projective of the same rank, it suffices to prove that it is surjective. To see this surjectivity, first note that since the map
\[ \mathbf{R}^1f_*M \rightarrow \mathbf{R}^1f_{0*}M_0 \]
is surjective, the image of 
\[ \mathbf{R}^1f_*M \otimes_A \mathcal{R}^{x_d+}_{K\weak{x_1,\ldots,x_{d-1}}} \rightarrow  \mathbf{R}^1f_{0*}M_0 \otimes_{K\weak{x_1,\ldots,x_{d-1}}} \mathcal{R}_{K\weak{x_1,\ldots,x_{d-1}}}^{x_d+} \]
surjects onto $\mathbf{R}^1f_{0*}M_0$ via the natural quotient map $x_d\mapsto 0$. Since $\nabla$-modules are rigid, it follows that in fact 
\[ \mathbf{R}^1f_*M \otimes_A \mathcal{R}^{x_d+}_{K\weak{x_1,\ldots,x_{d-1}}} \rightarrow \mathbf{R}^1f_{0*}M_0 \otimes_{K\weak{x_1,\ldots,x_{d-1}}} \mathcal{R}_{K\weak{x_1,\ldots,x_{d-1}}}^{x_d+} \]
is surjective, hence so is
\[ \mathbf{R}^1f_*M_{A'} \otimes_{A'} \mathcal{R}^{x_d}_{K\weak{x_1,\ldots,x_{d-1}}} \rightarrow \mathbf{R}^1f_{0*}M_0 \otimes_{K\weak{x_1,\ldots,x_{d-1}}} \mathcal{R}_{K\weak{x_1,\ldots,x_{d-1}}}^{x_d} \]
simply by some diagram chasing. Finally, by Lemmas \ref{lemma: R1 inj} and \ref{lemma: gluing} we can deduce that $\mathbf{R}^1f_*M$ is contained in finite projective module of rank $m$, it is therefore finite over $K\weak{x_1,\ldots,x_d}$, and, since it admits an integrable connection, projective of rank $\leq m$. Since the base change map
\[ \mathbf{R}^1f_*M \otimes_A K' \rightarrow H^1(M_s)  \]
to each closed point of $A$ is surjective, we can conclude in fact $\mathbf{R}^1f_*M$ is in fact of rank $=m$ and that every such base change map is an isomorphism.
\end{proof}

Applying Remarks \ref{rem: r0 bc} and \ref{rem: r1 bc}, this completes the proof of Theorem \ref{theo: main dagger}.

\section{Partially overconvergent cohomology}\label{sec: conv}

Theorem \ref{theo: main dagger} is a statement concerning relative Monsky--Washnitzer cohomology, and as such only applies when the base variety $S$ is smooth and affine. In order to be able to deduce a statement for any base $S$ which is separated and of finite type over $k$ (and indeed, for more general curves than allowed by Setup \ref{setup: main}), we will need to `complete along the base' $A$ and prove a similar result for partially (vertically) overconvergent cohomology. We will again consider Setup \ref{setup: main}, thus (in particular) we have a morphism of MW-type frames 
\[ \mathfrak{f}:(U,\overline{C},\mathfrak{C})\rightarrow ( S,\overline{S},\mathfrak{S} ) \]
enclosing an affine curve $f:U\rightarrow S$, with induced morphism $A\rightarrow B$ of $K$-dagger algebras as in \S\ref{sec: gnalk}. We let $M$ be a $\nabla$-module on $B$ obtained as the realisation of an overconvergent isocrystal $E$ on $U/K$. Write
\[ A=\mathrm{colim}_\lambda A_\lambda \;\;\;\;\text{ and }\;\;\;\;B=\mathrm{colim}_\lambda B_\lambda\]
as colimits of smooth affinoid algebras over $K$, such that $A_\lambda \rightarrow B_\lambda$ for all $\lambda$. Let $\widehat{A}$ denote the affinoid completion of $A$, and set
\[ B_{\widehat{A}}:= \mathrm{colim}_\lambda \widehat{A} \widehat{\otimes}_{A_\lambda} B_\lambda,\]
this is `relative dagger algebra' over $\widehat{A}$. Set $M_{\widehat{A}}= M\otimes_B B_{\widehat{A}}$ and define $\mathbf{R}^if_* M_{\widehat{A}}$ to be the cohomology groups of the complex
\[ M_{\widehat{A}} \overset{\nabla}{\rightarrow} M_{\widehat{A}} \otimes_B \Omega^1_{B/A}.\]
Our base change result is then the following.

\begin{theorem} \label{theo: conv} Assume that $M$ is $F$-able, and the equivalent conditions of Theorem \ref{theo: main dagger} hold for $M$. Then the base change map
\[ \mathbf{R}^if_*M \otimes_A \widehat{A} \rightarrow \mathbf{R}^if_*M_{\widehat{A}}\]
is an isomorphism for $i=0,1$. 
\end{theorem}

We will need some preliminaries on topological modules over the Banach ring $\widehat{A}$. This ring is a Tate ring in the sense of Huber \cite[\S1.1]{Hub96} that is also separated, complete, reduced, and admits a Noetherian ring of definition $\widehat{A}^\mathrm{int}\subset \widehat{A}$, consisting of elements of supremum norm $\leq 1$.

\begin{definition} A topological $\widehat{A}$-module $N$ is said to be \emph{locally convex} if there exists a neighbourhood base of $0$ consisting of $\widehat{A}^\mathrm{int}$-lattices in $N$. 
\end{definition}

For clarity, we will sometimes refer to  a topology being locally convex rel. $\widehat{A}$. As with the case of vector spaces over a non-archimedean field, a locally convex topology on an $\widehat{A}$-module $N$ is determined by its collection of open $\widehat{A}^\mathrm{int}$-lattices. Locally convex topologies are exactly those which can be defined using a collection of norms on $N$, all of which are compatible with the supremum norm on $\widehat{A}$.

\begin{example} \begin{enumerate}
\item Any $\widehat{A}$-module $N$ admits a finest locally convex topology, for which all $\widehat{A}^\mathrm{int}$-lattices are open. We will call this the strong topology on $N$. If $N$ is finitely generated, then this is the quotient topology arising from any surjection $\widehat{A}^{\oplus n} \rightarrow N$, and $N$ is separated and complete with respect to this topology (since $\widehat{A}$ is an affinoid algebra over $K$, this follows from \cite[Proposition 2.1.9]{Ber90}).
\item Be warned that even finitely generated $\widehat{A}$-modules may admit distinct locally convex topologies. For example, there is a locally convex topology on $K\tate{x}$ (as a free module over itself) for which a basis of open lattices is given by $\Lambda_{n,m}=p^n\mathcal{V}\tate{x} + x^mK\tate{x}$, for $n,m\in \Z_{\geq 0}$. This is strictly weaker than the strong topology, and $K\tate{x}$ is not complete with respect to this topology. Its completion is $K\pow{x}$, endowed with the direct product topology via $K\pow{x}\cong \prod^\infty_{i=1} K$. In particular a separated, locally convex $\widehat{A}$-module may contain a dense, finitely generated, proper submodule.
\end{enumerate}
\end{example}

We can avoid the somewhat pathological behaviour of the second example by comparing with the situation over the completed fraction field $L$ of $\widehat{A}$, over which any finitely generated module has a unique separated, locally convex topology.

\begin{lemma} \label{lemma: sub} Let $N$ be a finite projective $\widehat{A}$-module, and $L$ the completed fraction field of $\widehat{A}$. Then the strong topology on $N$ is the subspace topology coming from the inclusion
\[ N\hookrightarrow N\otimes_{\widehat{A}} L. \]
\end{lemma}

\begin{proof}
Choose an isomorphism $N\oplus P \cong \widehat{A}^{\oplus n}$, thus both maps $N\hookrightarrow \widehat{A}^{\oplus n}$ and $N\otimes_{\widehat{A}}L \hookrightarrow L^{\oplus n}$ are strict for the strong topologies. We can therefore reduce to the case $N=\widehat{A}^{\oplus n}$. In this case, both topologies are induced by the canonical norm
\[ \Norm{(a_1,\ldots,a_n)}:= \max_{1\leq i\leq n} \Norm{a_i}_\mathrm{sup} \]
on $\widehat{A}^{\oplus n}$ coming from the supremum norm on $\widehat{A}$. 
\end{proof}

\begin{proof}[Proof of Theorem \ref{theo: conv}] First note that the equivalent conditions of Theorem \ref{theo: main dagger} are preserved under passing from $M$ to either the submodule $\mathbf{R}^0f_*M \otimes_A B$ or the quotient of $M$ by this submodule. Thus using induction and the five lemma it suffices to consider the two cases when $\mathbf{R}^0f_*M=0$ and when $M=N\otimes_A B$ for some $\pn$-module $N$ over $A$. The case when $M=N\otimes_A B$ is easily handled by the projection formula, we will concentrate on the latter. 

Let $L$ denote the completed fraction field of $A$. By Remark \ref{rem: bc to L2} we know that $\mathbf{R}^0f_*M=0\Rightarrow \mathbf{R}^0f_*M_L=0$ and hence $\mathbf{R}^0f_*M_{\widehat{A}}\hookrightarrow  \mathbf{R}^0f_*M_L=0$, this gives the base change claim for $\mathbf{R}^0f_*M$. To prove the base change claim for $\mathbf{R}^1f_*M$, we give $B_{\widehat{A}}$ the (locally convex rel. $\widehat{A}$) inductive limit topology coming from the affinoid topology on each $\widehat{A}\otimes_{A_\lambda} B_\lambda$. This then induces a locally convex (rel. $\widehat{A}$) topology on any finitely generated $B_{\widehat{A}}$-module, such as $M_{\widehat{A}}$ or $M_{\widehat{A}} \otimes_B \Omega^1_{B/A}$. We equip $\mathbf{R}^1f_*M_{\widehat{A}}$ with the quotient topology via
\[M_{\widehat{A}} \rightarrow M_{\widehat{A}} \otimes_B \Omega^1_{B/A},\]
which makes it a (potentially non-separated) locally convex $\widehat{A}$-module. We can play the same game with $M_L$, to obtain a locally convex topology on the finite dimensional $L$-vector space $\mathbf{R}^1f_*M_L$, which in fact \emph{is} separated by \cite[Proposition 8.4.4]{Ked06a}. The natural map
\[ \mathbf{R}^1f_*M_{\widehat{A}} \rightarrow \mathbf{R}^1f_*M_L \]
is then continuous. Since the map
\[ \mathbf{R}^1f_*M \otimes_A L\rightarrow \mathbf{R}^1f_*M_L \]
is an isomorphism (by Remark \ref{rem: bc to L2}), it follows that both maps
\[ \mathbf{R}^1f_*M \otimes_A {\widehat{A}}\rightarrow \mathbf{R}^1f_*M_{\widehat{A}} \;\;\;\;\text{and}\;\;\;\;\mathbf{R}^1f_*M \otimes_A {\widehat{A}}\rightarrow \mathbf{R}^1f_*M_{L}\;\;\;\;\]
are injective, and since $M\otimes_A \widehat{A} \rightarrow M_{\widehat{A}}$ has dense image, so does
\[ \mathbf{R}^1f_*M \otimes_A {\widehat{A}}\rightarrow \mathbf{R}^1f_*M_{\widehat{A}}. \]
Let $Q$ be the maximal separated quotient of $\mathbf{R}^1f_*M_{\widehat{A}}$, i.e. the quotient by the closure of $\{0\}$. Then we have a factorisation
\[ \xymatrix{ \mathbf{R}^1f_*M \otimes_A \widehat{A} \ar[r] & \mathbf{R}^1f_*M_{\widehat{A}} \ar[r]\ar[d]  & \mathbf{R}^1f_*M_L \\ & Q. \ar[ur]  } \]
Since $\mathbf{R}^1f_*{M_L}$ is separated, it follows that $Q\hookrightarrow \mathbf{R}^1f_*{M_L}$, and since $\mathbf{R}^1f_*M \otimes_A 	\widehat{A} \hookrightarrow \mathbf{R}^1f_*M_L$ is injective, so is $\mathbf{R}^1f_*M\otimes_A \widehat{A} \rightarrow Q$. Since $\mathbf{R}^1f_*M\otimes_A \widehat{A} \rightarrow \mathbf{R}^1f_*M_{\widehat{A}}$ has dense image, so does $\mathbf{R}^1f_*M\otimes_A \widehat{A}\rightarrow Q$. Now using Lemma \ref{lemma: sub} together with continuity of the map $Q\rightarrow \mathbf{R}^1f_*M_L$, we can see that the topology on $\mathbf{R}^1f_*M \otimes_A \widehat{A}$ induced by the inclusion
\[ \mathbf{R}^1f_*M \otimes_A \widehat{A} \hookrightarrow Q \]
has to be finer than the strong topology, it is therefore equal to the strong topology. Hence $\mathbf{R}^1f_*M \otimes_A \widehat{A}$ is complete with respect to this topology, and thus has closed image in $Q$; therefore $\mathbf{R}^1f_*M \otimes_A \widehat{A} \isomto Q$.

One more application of Lemma \ref{lemma: sub} tells us that the topology on $Q$ must also be the strong topology (since it maps continuously into $\mathbf{R}^1f_*M_L$), and hence the map $\mathbf{R}^1f_*M \otimes_A \widehat{A} \rightarrow Q$ is in fact a homeomorphism. In other words, the exact sequence
\[ 0\rightarrow \overline{\{0\}} \rightarrow \mathbf{R}^1f_*M_{\widehat{A}} \rightarrow Q \rightarrow 0 \]
admits a topological splitting, which implies that in fact $\overline{\{0\}}=\{0\}$ and $\mathbf{R}^1f_*M \otimes_A \widehat{A}\isomto \mathbf{R}^1f_*M_{\widehat{A}}$.
\end{proof}

We will use the above theorem in a slightly different, and more geometric, form. Let $\mathfrak{S}^{\circ}$ denote the open formal subscheme of $\mathfrak{S}$ whose underlying topological space is $S$, and let $\mathfrak{C}^{\circ}$ be the inverse image of $\mathfrak{S}^\circ$ under $\mathfrak{f}:\mathfrak{C}\rightarrow \mathfrak{S}$. Let
\[ \mathrm{sp}: \mathfrak{C}^{\circ}_{K} \rightarrow \mathfrak{C}^{\circ}\]
be the specialisation map, and
\[ \mathcal{O}_{\mathfrak{C}^{\circ},\Q}(^\dagger C\setminus U)=\mathrm{sp}_*j_U^\dagger\mathcal{O}_{\mathfrak{C}^{\circ}_{K}} \]
the sheaf of functions on $\mathfrak{C}^{\circ}$ with overconvergent singularities along $C\setminus U$. If we realise $E$ on $\mathfrak{C}^{\circ}_{K}$ and pushforward along the specialisation map we obtain a coherent $\mathcal{O}_{\mathfrak{C}^{\circ},\Q}(^\dagger C\setminus U)$-module $\mathrm{sp}_*E_{\mathfrak{C}^\circ}$ together with an integrable connection, whose module of global sections is precisely $M_{\widehat{A}}$.

\begin{corollary} \label{cor: conv} With assumptions as in Theorem \ref{theo: conv}, the relative de Rham cohomology sheaves
\[ \mathbf{R}^i\mathfrak{f}_* \left( \mathrm{sp}_*E_{\mathfrak{C}^\circ} \otimes_{\mathcal{O}_{\mathfrak{C}^\circ}} \Omega^*_{\mathfrak{C}^{\circ} / \mathfrak{S}^{\circ} } \right) \]
are coherent $\mathcal{O}_{\mathfrak{S}^{\circ},\Q}$-modules.
\end{corollary}

\section{Local acyclicity via arithmetic \texorpdfstring{$\mathcal{D}^\dagger$}{D}-modules} \label{sec: Dmod2}

We are now ready to prove our second local acyclicity result for smooth relative curves. This will be valid over not necessarily smooth bases $S$, but will involve the additional assumption that the resiude field $k$ is perfect; we will assume this from now on. Fix an arbitrary $k$-variety $S$, let $f:U\rightarrow S$ be a good curve, and $E\in \mathrm{Isoc}_F^\dagger(U/K)$. Then for any geometric point $\bar{s}\rightarrow S$ over a point $s\in S$ we can pullback $E$ to get an overconvergent isocrystal $E_{\bar{s}}$ on $U_{\bar{s}}$ over $K(\bar{s}):=K\otimes_{W(k)} W(k(\bar s))$. If we let $C_{\bar{s}}$ denote the smooth compactification of $U_{\bar{s}}$, then for every point $x\in C_{\bar{s}}\setminus U_{\bar{s}}$ we can apply the construction of \cite[\S7]{Cre98} to pullback $E_{\bar{s}}$ to a punctured formal neighbourhood of $x$ in $C_{\bar s}$ to obtain an overconvergent $\nabla$-module $M_x$ over a copy of the Robba ring $\mathcal{R}_{K(\bar s),x}$ at $x$.

\begin{definition} We define the Swan conductor of $E_{\bar{s}}$ at $x$ to be the irregularity of the overconvergent $\nabla$-module $M_x$,
\[\mathrm{Sw}_x(E_{\bar{s}}) := \mathrm{Irr}(M_x). \]
We define 
\[ \varphi_E(\bar s):= \sum_{x\in C_{\bar{s}}\setminus U_{\bar{s}}}  \mathrm{Sw}_x(E_{\bar s}).\]
\end{definition}
 The positive integer $\varphi_E(\bar s)$ only depends on the point $s\in S$ lying under $\bar{s}$, we thus obtain a function
\[ \varphi_E: S\rightarrow \Z_{\geq0}. \]
Note that this is not a direct analogue of the function $\varphi$ considered in \cite{Lau81}, which also includes a contribution from the rank of $E$. This minor difference notwithstanding, we have the following partial $p$-adic analogue of \cite[Corollary 2.1.2]{Lau81}.

\begin{theorem} \label{theo: main general}
Let $f:U\rightarrow S$ be a relative smooth affine curve over a $k$-variety $S$, and $E\in \mathrm{Isoc}^\dagger_F(U/K)$. Suppose that $f$ admits a smooth compactification $\bar{f}:C\rightarrow S$ such that the complementary divisor $C\setminus U$ is \'etale over $S$.
\begin{enumerate}
\item \label{mgi} The function $\varphi_E:S\rightarrow \Z_{\geq 0}$ is constructible and lower semi-continuous.
\item \label{mgii} $f_!\rho(E)\in D^{b,\dagger}_{\mathrm{isoc},F}(S/K)$ if and only if $\varphi_E$ is locally constant on $S$.
\end{enumerate}
\end{theorem}

\begin{remark} \begin{enumerate}
\item  This result is weaker than \cite[Corollary 2.1.2]{Lau81} in that we assume the complement $C\setminus U$ is finite \'etale over $S$, whereas in \cite{Lau81} it is only required to be finite flat. A formalism of vanishing and nearby cycles in $p$-adic cohomology is developed in \cite{Abe19}, it would be interesting to see whether Abe's machinery can be used to prove a $p$-adic version of the more general result. 
\item If the equivalent conditions of Theorem \ref{theo: main general}(\ref{mgii}) are satisfied, then the constructible cohomology sheaves 
\[ \mathbf{R}^if_!E: = {}^c\hspace{-1mm}\cur{H}^i(f_!\rho(E)) \in \mathrm{Isoc}_F^\dagger(S/K)  \]
are of formation commuting with arbitrary base change $S'\rightarrow S$. If in addition $S$ is smooth over $k$, then the constructible cohomology sheaves
\[ \mathbf{R}^if_*E: = {}^c\hspace{-1mm}\cur{H}^i(f_+\rho(E)) \in \mathrm{Isoc}_F^\dagger(S/K) \]
are also overconvergent isocrystals, and of formation commuting with base change $S'\rightarrow S$ of \emph{smooth} varieties. Moreover, in this case we have perfect pairings 
\[ \mathbf{R}^if_*E \otimes \mathbf{R}^{2-i}f_!E \rightarrow \mathcal{O}^\dagger_{S/K}(-1) \]
of overconvergent isocrystals, which are compatible with any given Frobenius structure on $E$. Thus even in the smooth case, Theorem \ref{theo: main general} gives us slightly more than Theorem \ref{theo: main dagger}.
\end{enumerate} 
\end{remark}

Let us first consider Theorem \ref{theo: main general}(\ref{mgi}), which boils down to two claims:
\begin{enumerate}
\item there exists an open subset $U\subset S$ such that $\varphi_E$ is constant on $U$ (we apply this successively to the complement of $U$ in $S$ and so forth);
\item if $\eta,s\in S$ and $s\in \overline{\{\eta\}}$ then $\varphi_E(s)\leq \varphi_E(\eta)$. 
\end{enumerate}
Note that if $a:S'\rightarrow S$ is any morphism, and $s'\in S'$, then $\varphi_{a^*E}(s')=\varphi_E(a(s'))$. Hence to prove the first of these we may replace $S$ by any variety $a:S'\rightarrow S$ flat over $S$, and for the second we may replace $S$ by an alteration followed by the inclusion of an open affine containing $s$ (and hence $\eta$). In either case, we can assume that $S$ is smooth and affine. Working one connected component at a time, we may assume $S$ connected.

Now applying Lemma \ref{prop: good to ssl} we may assume that we are in the situation of Setup \ref{setup: main}. The first claim then follows from Corollary \ref{cor: base change 2} together with the Grothendieck--Ogg--Shafarevich formula
\[  \chi(U_{\bar{s}},E_{\bar s}) = \chi(U_{\bar s})\cdot \mathrm{rank}\,E_{\bar s} - \sum_{x\in C_{\bar s}\setminus U_{\bar s} } \mathrm{Sw}_x(E_{\bar s}), \]
see for example \cite[Corollaire 5.0-12]{CM01}. For the second, we may replace $S$ by a suitable alteration of $\overline{\{\eta\}}$, and thus assume that $\eta$ is the generic point of $S$. The claim then follows from Proposition \ref{prop: lower}. 

To prove Theorem \ref{theo: main general}(\ref{mgii}) we first suppose that $f_!\rho(E)\in D^{b,\dagger}_{\mathrm{isoc},F}(S/K)$; we must show that $\varphi_E$ is locally constant. We may clearly assume that $S$ is connected, and since we already know that $\varphi_E$ is constructible, it suffices to show that it is constant on the set $\norm{S}$ of closed point of $S$. If $i_s:s\rightarrow S$ is the inclusion of a closed point, inducing a Cartesian diagram
\[ \xymatrix{ X_s \ar[r]^{i_s'} \ar[d]_{f_s} & X \ar[d]^f \\ s \ar[r]^{i_s}  & S, } \]
then we have that
\[ i_s^+f_!\rho(E) \cong f_{s!}i'^+_s\rho(E) \cong f_{s!} \rho(i_s^*E). \]
Since $f_!\rho(E) \in D^{b,\dagger}_{\mathrm{isoc},F}(S/K)$ and $i_s^+$ is exact for the constructible $t$-structure, we deduce the existence of objects
\[ \mathbf{R}^if_!E := {}^c\hspace{-1mm}\mathcal{H}^i(f_!\rho(E)) \in \mathrm{Isoc}^\dagger_F(S/K) \]
such that $i_{s}^*\mathbf{R}^if_!E \cong H^i_{c,\rig} (U_{s}/K(s),E_{s})$ for all $s$. In particular, we can see that the \emph{compactly supported} Euler characteristic
\[ s\mapsto \chi_c(U_{s},E_{s})=\chi_c(U_{\bar s},E_{\bar s}) \]
is constant on $\norm{S}$. Since $\mathrm{Sw}_x(E^\vee_{\bar s})=\mathrm{Sw}_x(E_{\bar s})$, applying Poincar\'e duality and the $p$-adic Grothendieck--Ogg--Shafarevich formula tells us that
\begin{align*} \chi_c(U_{\bar s},E_{\bar s}) &= \chi(U_{\bar s},E^\vee_{\bar s}) = \chi(U_{\bar s})\cdot \mathrm{rank}\,E^\vee_{\bar s} - \sum_{x\in C_{\bar s}\setminus U_{\bar s} } \mathrm{Sw}_x(E^\vee_{\bar s}) \\
&= \chi(U_{\bar s})\cdot \mathrm{rank}\,E_{\bar s} - \sum_{x\in C_{\bar s}\setminus U_{\bar s} } \mathrm{Sw}_x(E_{\bar s}).
\end{align*} 
Thus constancy of $\chi_c(U_s,E_s)$ implies that of $\varphi_E(s)$, and we are done.

To prove the converse implication, then, let us assume that $\varphi_E$ is constant, and take an alteration $a:S'\rightarrow S$ with $S'$ smooth. Then we have a Cartesian diagram
\[ \xymatrix{ X'\ar[r]^{a'} \ar[d]_{f'} & X \ar[d]^f \\ S'\ar[r]^a & S } \]
and isomorphisms
\[ a^+f_!\rho(E) \cong f'_!a'^+\rho(E) \cong f'_!\rho(a'^*E). \]
By \cite[Lemma 3.3]{Abe19}, together with $t$-exactness of $a^+$, it therefore suffices to show that $f'_!\rho(a'^*E)\in D^{b,\dagger}_{\mathrm{isoc},F}(S'/K)$, and thus replacing $S$ by $S'$ we may assume that $S$ is smooth. Since the question is local on $S$, we may also assume that it is affine and connected. Again applying \cite[Lemma 3.3]{Abe19} and using the fact that smoothness of a given constructible module (in the sense of \cite[Definition 3.2]{Abe19}) is clearly Zariski local, we may invoke the results of \cite[\S3]{Voe96} on the h-topology (specifically, the fact that the topology generated by Zariski covers and proper surjective maps is finer than the \'etale topology) to show that smoothness of a given constructible module is in fact \'etale local. Hence after taking a suitable \'etale cover of $S$ as in \ref{prop: good to ssl}, we may assume that we are in the situation of Setiup \ref{setup: main}.

Now by Poincar\'e duality
\[ f_!\rho(E) \cong \mathbf{D}_Sf_+ \rho(E^\vee) [-2\dim S],  \]
it therefore suffices to show that $f_+\rho(E^\vee)\in D^{b,\dagger}_{\mathrm{isoc},F}(S/K) $. Since $\varphi_E=\varphi_{E^\vee}$ we may replace $E$ by $E^\vee$, and since $S$ is smooth we have $\rho(E) \cong \mathrm{sp}_+E[-\dim S]$; we must therefore show that $\varphi_E$ constant $\Rightarrow$ $f_+\mathrm{sp}_+E\in D^{b,\dagger}_{\mathrm{isoc},F}(S/K)$. Now let
\[ \mathfrak{f} : (C,\overline{C},\mathfrak{C}) \rightarrow (S,\overline{S},\mathfrak{S}) \]
be a morphism of frames extending a good compactification of $f$ as in Setup \ref{setup: main}. Let $\mathfrak{S}^{\circ} \subset \mathfrak{S}$ denote the open formal subscheme with underlying topological space $S$ and let $\mathfrak{C}^{\circ}$ denote the fibre product of $\mathfrak{S}^{\circ}$ with $\mathfrak{C}$ over $\mathfrak{S}$. Let $\hat{E}$ denote the image of $E$ inside the category $\mathrm{Isoc}_F((U,C)/K)$ of isocrystals on $U$ overconvergent along $C$ (which are extensions of isocrystals admitting Frobenius structures). Since the diagram
\[ \xymatrix{     \mathrm{Isoc}_F^\dagger(U/K) \ar[r]^{\mathrm{sp_+}}\ar[d] & D^{b,\dagger}_{\mathrm{hol},F}(U/K) \ar[r]^{f_+}\ar[d]  & D^{b,\dagger}_{\mathrm{hol},F}(S/K)  \ar[d] \\  \mathrm{Isoc}_F((U,C)/K) \ar[r]^{\mathrm{sp_+}} & D^b_{\mathrm{hol},F}((U,C)/K) \ar[r]^{f_+} & D^b_{\mathrm{hol},F}(S/K)  } \]
commutes up to natural isomorphism, it suffices by Lemma \ref{lemma: isoc conv} to show that 
\[ f_+\mathrm{sp}_+\hat{E} \in D^b_{\mathrm{isoc},F}(S/K). \]
By Lemma \ref{lemma: good lifts} the functor $f_+$ on the category of `convergent' holonomic modules can be computed in terms of the realisation $\mathfrak{C}^{\circ} \rightarrow \mathfrak{S}^{\circ}$ via
\[ \mathfrak{f}_+ : D^b_\mathrm{coh}(\mathcal{D}^\dagger_{\mathfrak{C}^{\circ},\Q}) \rightarrow D^b_\mathrm{coh}(\mathcal{D}^\dagger_{\mathfrak{S}^{\circ},\Q}), \]
and moreover it suffices to show that the $\mathcal{O}_{\mathfrak{S}^{\circ},\Q}$-modules underlying cohomology sheaves of $\mathfrak{f}_+\mathrm{sp}_+\hat{E}$ are coherent. In this case the construction of the functor $\mathrm{sp}_+E$ is very simple. Explicitly, we realise $E$ on $\mathfrak{C}^\circ_K$, and pushforward the resulting module with integrable connection $E_{\mathfrak{C}^\circ}$ via the specialisation map
\[ \mathrm{sp}: \mathfrak{C}^\circ_K \rightarrow \mathfrak{C}^\circ.\]
This results in a coherent $\mathcal{O}_{\mathfrak{C}^\circ,\Q}(^\dagger C\setminus U)$-module with integrable connection, which by \cite[Theorem 2.3.15]{CT12} extends to the structure of an overholonomic (and in particular, coherent) $\mathcal{D}^\dagger_{\mathfrak{C}^\circ,\Q}$-module (remember that $E$ is $F$-able). This $\mathcal{D}^\dagger_{\mathfrak{C}^\circ,\Q}$-module is nothing other than $\mathrm{sp}_+\hat{E}$. Finally using \cite[(4.3.6.3)]{Ber02} we can identify
\[ \mathfrak{f}_+\mathrm{sp}_+\hat{E}[-1] \cong \mathbf{R}\mathfrak{f}_*\left(\mathrm{sp}_*E_{\mathfrak{C}^\circ} \otimes_{\mathcal{O}_{\mathfrak{C}^\circ}} \Omega^*_{\mathfrak{C}^{\circ}/\mathfrak{S}^{\circ}} \right) \]
and therefore apply Corollary \ref{cor: conv} to conclude.

\bibliographystyle{../../Templates/bibsty}
\bibliography{../../Templates/lib.bib}

\end{document}